\documentclass[a4paper,10pt]{amsart}

\usepackage[english]{babel}

\usepackage{lmodern}

\usepackage[utf8]{inputenc}

\usepackage[babel=true]{csquotes}

\usepackage{comment}

\usepackage{amsfonts,amssymb,amsmath,eucal,pinlabel,array,hhline}

\usepackage{amsfonts}
\usepackage{amsthm} 
\usepackage{amsmath}
\usepackage{amssymb}

\usepackage{booktabs}

\usepackage{pifont}

\usepackage{bbding}

\usepackage{multirow} 
\usepackage{kbordermatrix}

\usepackage{brunnian}

\usetikzlibrary{arrows}
\usetikzlibrary{matrix}

\usepackage{tikz}
\usepackage{tikz-cd}
\usepackage{pgfplots}
\usetikzlibrary{knots}
\usetikzlibrary{arrows,decorations.markings}
\usepackage{amsfonts}
\usepackage{amssymb}
\usepackage{graphicx} 
\usepackage{amsthm}
\usepackage[english]{babel}
\usepackage{hyperref}
\usepackage{enumerate}   
\hypersetup{
	colorlinks=true,
	linkcolor=black,
	filecolor=blue,      
	urlcolor=blue,
}

\theoremstyle{definition}
\newtheorem{definition}{Definition}[section] 

\theoremstyle{plain}
\newtheorem{lemma}[definition]{Lemma} 
\newtheorem{proposition}[definition]{Proposition}
\newtheorem{theorem}[definition]{Theorem}
\newtheorem{notation}[definition]{Notation}

\newtheorem{Conjecture}[definition]{Conjecture}
 \newtheorem{numerical}[definition]{Numerical test}

\theoremstyle{remark}
\newtheorem{remark}[definition]{Remark}

\newcommand{\Vol}{\mathrm{Vol}}

\newcommand{\C}{\mathbb{C}}
\newcommand{\R}{\mathbb{R}}

\newcommand{\N}{\mathbb{N}}

\tikzstyle{face}=[black]
\tikzstyle{tore}=[cyan,circle,draw,fill=cyan!10]
\tikzstyle{rod}=[black,draw,circle,fill=black!10]
\tikzstyle{two-hole}=[blue,circle,draw,fill=blue!10]
\tikzstyle{front}=[very thick]
\tikzstyle{back}=[loosely dashed,very thick]
\tikzstyle{frontbis}=[very thick,decoration={markings,mark=at position 1 with
	{\arrow[scale=1.2,>=stealth]{>}}},postaction={decorate}]
\tikzstyle{backbis}=[loosely dashed,very thick, decoration={markings,mark=at position 1 with
	{\arrow[scale=1.2,>=stealth]{>}}},postaction={decorate}]
\tikzstyle{frontb}=[very thick]
\tikzstyle{backb}=[loosely dashed,very thick]
\tikzstyle{labels}=[midway, fill=white!30, scale=1.5]

\tikzstyle{comb}=[black,line width=0.8mm]
\tikzstyle{link}=[gray,line width=0.3mm]

\tikzstyle{sign+}=[red,draw,circle,fill=red!10]
\tikzstyle{sign-}=[blue,draw,circle,fill=blue!10]

\tikzset{%
	symbol/.style={
		draw=none,
		every to/.append style={
			edge node={node [sloped, allow upside down, auto=false]{$#1$}}
		},
	},
}

\def\-{\scalebox{1.8}[1]{-}}

\usepackage{listings}
\usepackage{xcolor}

\definecolor{codegreen}{rgb}{0,0.6,0}
\definecolor{codegray}{rgb}{0.5,0.5,0.5}
\definecolor{codepurple}{rgb}{0.58,0,0.82}
\definecolor{backcolour}{rgb}{0.95,0.95,0.92}

\lstdefinestyle{mystyle}{
	backgroundcolor=\color{backcolour},   
	commentstyle=\color{codegreen},
	keywordstyle=\color{magenta},
	numberstyle=\tiny\color{codegray},
	stringstyle=\color{codepurple},
	basicstyle=\ttfamily\footnotesize,
	breakatwhitespace=false,         
	breaklines=true,                 
	captionpos=b,                    
	keepspaces=true,                 
	numbers=left,                    
	numbersep=5pt,                  
	showspaces=false,                
	showstringspaces=false,
	showtabs=false,                  
	tabsize=2
}

\title[The Chen--Yang volume conjecture for knots in handlebodies]
{The Chen--Yang volume conjecture for knots in handlebodies}
\author{Fathi Ben Aribi and James Gosselet}

\address{
	UCLouvain, IRMP, Chemin du Cyclotron 2 \\
	1348 Louvain-la-Neuve \\
	Belgium}
\email{fathi.benaribi@uclouvain.be}

\address{
	UCLouvain, IRMP, Chemin du Cyclotron 2 \\
	1348 Louvain-la-Neuve \\
	Belgium}
\email{james.gosselet@student.uclouvain.be}

\makeatletter
\@namedef{subjclassname@2020}{%
	\textup{2020} Mathematics Subject Classification}
\makeatother

\subjclass[2020]{57K16; 57K32}
\keywords{Turaev--Viro invariants; volume conjectures; hyperbolic volume; triangulations of $3$-manifolds.}

\begin{document}

	\maketitle
	
	\begin{abstract}
		In 2015, Chen and Yang proposed a volume conjecture that stated that certain Turaev--Viro invariants of an hyperbolic $3$-manifold should grow exponentially with a rate equal to the hyperbolic volume.
		
		Since then, this conjecture has been proven or numerically
		tested for several hyperbolic $3$-manifolds, either closed or with
 boundary, the boundary being either a family of tori or a family of higher genus surfaces. The current paper now provides new numerical checks of this volume conjecture for $3$-manifolds with one toroidal boundary component and one geodesic boundary component.
		
		More precisely, we study a family of hyperbolic $3$-manifolds $M_g$ introduced by Frigerio. Each $M_g$ can be seen as the complement of a knot in an handlebody of genus $g$.
		
		We provide an explicit code that computes the Turaev--Viro invariants of these manifolds $M_g$, and we then numerically check the Chen--Yang volume conjecture for the first six members of this family.
		
		Furthermore, we propose an extension of the volume conjecture, where the second coefficient of the asymptotic expansion only depends on the topology of the boundary of the manifold. We numerically check this property for the manifolds $M_2, \ldots, M_7$ and we also observe that the second coefficient grows linearly in the Euler characteristic $\chi(\partial M_g)$.
	\end{abstract}
	
	\tableofcontents

\section{Introduction}\label{sec:intro}

	Quantum topology began in 1984 with the definition of the Jones Polynomial \cite{Jones}. Since then, several new invariants of knots and 3-manifolds were defined and, inspired by quantum field theories, the \emph{volume conjecture} of Kashaev \cite{Kashaevvolume} and its variants rose to be among  the most studied conjectures in quantum topology. What is intriguing about these volume conjectures is that they link quantum invariants of manifolds to the hyperbolic structure of these manifolds. In \cite{chen2018volume}, Chen--Yang proposed a volume conjecture using a family  $\left\{ TV_{r,2} \right\}_{r \in \mathbb{N}_{\geq 3}}$ of Turaev-Viro type invariants for compact 3-manifolds.
	\begin{Conjecture}[Conjecture 1.1, \cite{chen2018volume}, Chen-Yang]\label{conjecture}
		Let $M$ be a compact hyperbolic 3-manifold. Then for $r$ running over all odd integers such that $r \geq 3$,
		$$\lim_{r \to \infty} \frac{2 \pi }{r-2} \log \left( TV_{r,2} (M) \right) = \Vol(M), $$ where $\Vol(M)$ is the hyperbolic volume of $M$.
	\end{Conjecture}
	This conjecture has since then been proven \cite{Detcherry, Ohtsuki, wong2020asymptotics} or numerically tested \cite{chen2018volume, MaRo} for several hyperbolic $3$-manifolds, mostly for manifolds without boundary, but also manifolds with toroidal boundary and manifolds with totally geodesic boundary of genus $g \geq2$. However, there is currently no test of this conjecture for an hyperbolic $3$-manifold with a boundary which has both a toroidal component and a totally geodesic boundary component of genus $g \geq2$. In this paper, we thus propose to test the Chen--Yang  volume conjecture for a family of $3$-manifolds $\left\{ M_g \right\}_{g \in \mathbb{N}_{\geq 2}}$, each $M_g$ being the exterior of a knot in a handlebody of genus $g$. These manifolds were constructed and studied by Frigerio \cite{Frig}.
	
	We will actually go one step further, as we will test an extension of Conjecture \ref{conjecture} stated as follows:	
		\begin{Conjecture}\label{conj:vol:bc}
		Let $M$ be a compact hyperbolic 3-manifold of volume $\Vol(M)$. Then for $r$ running over all odd integers such that $r \geq 3$,
		$$ TV_{r,2} (M) \underset{r \to \infty}{\sim} \omega \cdot r^{b} \cdot e^{\frac{\Vol(M)}{2 \pi}(r-2)} \cdot \left (
		1 + O\left (\frac{1}{r-2}\right )
		\right ),
		$$
where $\omega, b \in \R$ are independent of $r$,		or equivalently,
		$$ \frac{2 \pi }{r-2} \log \left( TV_{r,2} (M) \right) 
		\underset{r \to \infty}{\sim} 
		 \Vol(M) + b \dfrac{2 \pi \ln(r-2)}{r-2} + c \dfrac{1}{r-2} + O\left (\frac{1}{(r-2)^2}\right ), $$ where $b \in \R$ and  $c \in \C$ are independent of $r$.
	\end{Conjecture}

Variants of Conjecture \ref{conj:vol:bc} have been previously stated, proven and numerically checked, notably by Chen--Yang \cite[Section 6]{chen2015volume} (for some manifolds with boundary), and by Ohtsuki \cite{Ohtsuki} and Gang-Romo-Yamasaki \cite{GRY} (for certain closed manifolds, via the usual quadratic relation between Reshetikhin--Turaev invariants and Turaev--Viro invariants). In the present paper, we aim to numerically test Conjecture \ref{conj:vol:bc} for Frigerio's manifolds $M_g$.

A yet stronger conjecture surmises that the coefficient $b$ in Conjecture \ref{conj:vol:bc} should only depend on the topology of the boundary $\partial M$. We go further and conjecture that $b$ is linear in $\chi(\partial M)$, as stated as follows:

	\begin{Conjecture}\label{conj:vol:b:aff}
	Let $\mathcal{M}$ be the set of compact hyperbolic 3-manifolds for which  Conjecture \ref{conj:vol:bc} holds.
Then, for all $M \in \mathcal{M}$, the coefficient $b(M)$ in Conjecture \ref{conj:vol:bc} only depends on the topology of $\partial M$, as an affine function of the Euler characteristic $\chi(\partial M)$.
\end{Conjecture}

Conjecture \ref{conj:vol:b:aff}, as stated here, might be too strong to be true. Nevertheless, in this paper, it follows from numerical computations that Conjecture \ref{conj:vol:b:aff} appears to hold  for Frigerio's manifolds $M_g$.

After reviewing preliminaries about triangulations, hyperbolic volumes,  colorings of triangulations and Turaev--Viro invariants in Section \ref{sec:prelim}, 	
we discuss the manifolds $M_g$  in Section \ref{sec:frigerio}: in particular we describe their ideal triangulations $\mathcal{T}_g$, constructed by Frigerio. We observe that these ideal triangulations admit an \textit{ordered structure}, which means we can orient all edges  in a coherent way with face gluings and such that no triangle admits a cycle.
	\begin{proposition}[Proposition \ref{prop:ordered}]\label{orderedintro}
		The triangulations $\mathcal{T}_g$ of the manifolds $M_g$ admit an ordered structure.
	\end{proposition}
This property will not be used for the rest of the paper, but may be useful in the future for studying other quantum invariants which are only defined on ordered triangulations \cite{aribi2020geometric, kashaev2012tqft}, in the specific cases of these  manifolds $M_g$.

For two fixed integers $r \geqslant 3$ and $s \geqslant 1$, 	
the Turaev-Viro invariant $TV_{r,s} (M,\mathcal{T})$ of a triangulated manifold $(M,\mathcal{T})$ is defined as a sum over a (usually large) set of admissible colorings of the triangulation $\mathcal{T}$.
Hence, in order to compute the Turaev-Viro invariants $TV_{r,s} (M_g,\mathcal{T}_g)$, we first need to describe the associated set of   admissible colorings $\mathcal{A}_r(M_g,\mathcal{T}_g)$. The main theorem of this paper provides a description of $\mathcal{A}_r(M_g,\mathcal{T}_g)$ that is both clearer than the original definition and easier to transform into computer code (in Section \ref{sec:codetvr}). We now phrase it without technical details:
	\begin{theorem}[Theorem \ref{thm:allowedstates}]\label{r}
		An admissible coloring of $\mathcal{T}_g$ has to satisfy a certain set of admissibility conditions, which are equivalent to another explicit and more convenient set of conditions.
	\end{theorem}
	The proof of Theorem \ref{r} is quite lengthy  but can give insights on how to simplify admissibility conditions for other examples of triangulations. The new set of conditions given by Theorem \ref{r} will be used later on in Section \ref{sec:codetvr}. 
	
	In Section \ref{sec:res}, we study the specific cases of
	  $(M_g,\mathcal{T}_g)$ for $2 \leqslant g \leqslant 7$.
	  We numerically compute
	   their  hyperbolic volume (thanks to results of Frigerio \cite{Frig} and Ushijima \cite{ushijima2006volume}) and their logarithmic Turaev--Viro invariants
	   $QV_{r,2}(M_g):=\frac{2\pi}{r-2} \log(TV_{r,2}(M_g,\mathcal{T}_g))$ 
	    for several values of $r$.
	   For increasing values of $r$, we observe a convergence as expected in Conjecture \ref{conjecture} (for $2 \leqslant g \leqslant 7$), and then a surprising pattern break for $g\in \{2,3\}$.
	\begin{numerical}[Sections \ref{sec:num:disc:M2}, \ref{sec:num:disc:M3} and \ref{sec:M4-7}]
		Conjectures \ref{conjecture} and \ref{conj:vol:bc} appear to hold numerically for the manifolds $M_2, \ldots, M_7$, barring  possible numerical errors for $M_2$ and $M_3$.
		
		More precisely, the graph of the function $QV_{r,2}(M_2)$ (resp. $QV_{r,2}(M_3)$) shows a converging behavior up to $r=33$ (resp. $r=31$) and an unexpected increase after $r=33$ (resp. $r=31$).
			The graph of the function $QV_{r,2}(M_g)$ (for $4 \leqslant g \leqslant 7$) shows a converging behavior.
	\end{numerical}
	We offer hypotheses to explain the previous pattern breaks, and we furthermore compute an interpolating function for the data, which not only fits the pre-break values quite well, but also provides a promising candidate for the next term $b$ in the expected asymptotic expansion of $QV_{r,2}(M_g)$ (see Conjecture \ref{conj:vol:bc}).

	In Section \ref{sec:code} we provide our entire code (written in \textit{SageMath}), with annotations for clarity. In particular, we program functions that compute the hyperbolic volumes  and Turaev--Viro invariants for the manifolds $M_g$, and we list several of these values in the tables of Figures \ref{fig:volMg} and \ref{fig:table:qviro2-7}.
	
	Finally, we observe a linear behavior for the second coefficient $b$, as expected in Conjecture \ref{conj:vol:b:aff}.
	
		\begin{numerical}[Section \ref{sec:num:aff}]
		Conjecture \ref{conj:vol:b:aff} appears to hold numerically for the manifolds $M_2, \ldots, M_7$.
	\end{numerical}

The results in this article follow in part from  the Master's thesis \cite{Gos} of the second author.

\section{Materials and methods}\label{sec:matmet}

	All the calculations were performed on SageMath (a free open-source \mbox{mathematics} software system using Python 3), on a computer equipped with a 
	Intel® Core™ i5-8500 CPU @ 3.00GHz × 6 processor.
	
\section{Preliminaries}\label{sec:prelim}
	
	\subsection{Ideal triangulations}

In  this section, we follow some conventions of \cite{chen2018volume} and \cite{ushijima2006volume}.
A \textit{pseudo $3$-manifold} is a topological space $M$ such that each point $p$ of $M$ has an (open) neighborhood $U_p$ that is homeomorphic to a cone over a surface.
	A \textit{triangulation} $\mathcal{T}$ of a pseudo $3$-manifold $M$ consists of a disjoint union $T_1 \sqcup \ldots \sqcup T_N$
	 of finitely many  Euclidean tetrahedra  and of a collection of affine homeomorphisms between pairs of faces in $T_1 \sqcup \ldots \sqcup T_N$ such that the quotient space $\left (T_1 \sqcup \ldots \sqcup T_N\right )/\sim$  is
	homeomorphic to $M$. 
For $i\in \{0,1,2,3\}$, we denote $\mathcal{T}^{i}$ (resp. $\mathcal{T}^{i,\sim}$) the $i$-skeleton of the disjoint union of tetrahedra in $\mathcal{T}$ (resp. the $i$-skeleton of the quotient space homeomorphic to $M$).
We say that an element $\nu$ of the set $\mathcal{T}^{0,\sim}$ of vertices of $M$ is \textit{regular} (resp. \textit{ideal},  \textit{hyperideal}) if its associated neighborhood $U_\nu$ is a cone over a sphere (resp. over a torus, over a surface of genus at least $2$).
Finally, we say that a compact $3$-manifold $N$ with boundary admits an \textit{ideal triangulation} $\mathcal{T}$ if no vertices in $\mathcal{T}^{0,\sim}$ are regular, and $N$ is obtained from the quotient space of $\mathcal{T}$ by removing open neighborhoods $U_\nu$ of all  vertices $\nu$ in $\mathcal{T}^{0,\sim}$.

	\subsection{Hyperbolic volume}\label{sec:prelim:hyp:vol}

	In \cite{ushijima2006volume}, Ushijima gives a volume formula for a generalized hyperbolic tetrahedron 
	with a given angle structure. 
	Let us detail this formula and its components.
Let $T=T(A,B,C,D,E,F)$ be a generalized tetrahedron in the hyperbolic space $\mathbb{H}^3$ whose angle structure is as in Figure \ref{fig:dihconfig} ($A, \ldots, F \in [0,\pi]$ are dihedral angles).

	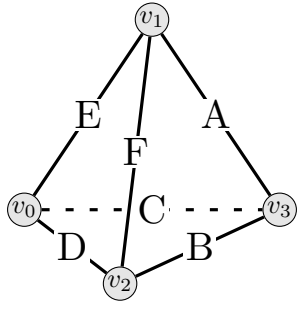
\begin{figure}[!h]	
	\centering
	\begin{tikzpicture}[scale=0.7,inner sep=0.2mm, centered]
	
	\node[rod] (a) at (-2.4,0) {$v_0$};
	\node[rod] (b) at (0,3.6) {$v_1$};
	\node[rod] (c) at (-0.6,-1.4) {$v_2$};
	\node[rod] (d) at (2.4,0) {$v_3$};
	
	\draw[backb] (a) -- (d) node[labels] {C};
	\draw[frontb] (a) -- (b) node[labels] {E};
	\draw[frontb] (a) -- (c) node[labels] {D};
	\draw[frontb] (d) -- (b) node[labels] {A};
	\draw[frontb] (c) -- (d) node[labels] {B};
	\draw[frontb] (c) -- (b) node[labels] {F};
	\end{tikzpicture}
	\caption{The configuration of dihedral angles on $T$.}\label{fig:dihconfig}
\end{figure}

Let $G$ denote the associated Gram matrix: 
$$G=\left(\begin{array}{cccc}
1 & - \cos \, A & - \cos\, B & - \cos\, F \\ 
- \cos\, A & 1 & - \cos\, C & - \cos\, E \\ 
- \cos\, B & - \cos\, C & 1 & - \cos\, D \\ 
- \cos\, F & - \cos\, E & - \cos\, D & 1
\end{array}\right).$$

Let $\mathrm{Li_2}(z)$ be the dilogarithm function defined for $z \in \C \setminus [1,\infty)$ in \cite[Introduction]{ushijima2006volume} by the analytic continuation of the following integral:
$$\mathrm{Li_2}(x) : = - \int_0^x \frac{\log(1-t)}{t} \text{d}t \quad \text{for } x \in \mathbb{R}_{>0}.$$

Let $I$ be the principal square root of $-1$ , $a:= \exp(I\cdot A)$, $b:= \exp(I\cdot B)$, ..., $f:= \exp(I\cdot F)$ and let $U(z,T)$ be the complex valued function defined as follows:
\begin{align*}
U(z,T):= \frac{1}{2}(&\mathrm{Li_2}(z) + \mathrm{Li_2}(abdez)+ \mathrm{Li_2}(acdfz) + \mathrm{Li_2}(bcefz)\\ &- \mathrm{Li_2}(-abcz)-\mathrm{Li_2}(-aefz)-\mathrm{Li_2}(-bdfz)-\mathrm{Li_2}(-cdez)).
\end{align*}

We denote by $z_{+}$ and $z_{-}$ the two complex numbers defined as follows:
\begin{align*}
z_\pm := & -2\frac{\sin \, A\sin\,D+\sin\,B\sin\,E+\sin\,C\sin\,F \pm  \sqrt{\text{det} \, G}}{ad+be+cf+abf+ace+bcd+def+abcdef},
\end{align*}
where $\sqrt{\text{det} \, G}\in I \mathbb{R}_{>0}$ is the principal square root of $\text{det} \, G$.

\begin{proposition}[Ushijima, \cite{ushijima2006volume}  Theorem 1.1]\label{prop:formula:ush}
	The hyperbolic volume $\Vol(T)$ of a generalized tetrahedron $T=T(A,B,C,D,E,F)$ is given as follows:
	$$\Vol(T) = \frac{1}{2}\mathcal{I}(U(z_-,T)-U(z_+,T)),$$ 
	where $\mathcal{I}$ means the imaginary part.
\end{proposition}

In Section \ref{sec:computing:ushijima} we provide a code to compute the previous functions and formula.

	\subsection{Admissible colorings}
	
We will mostly follow the notations, conventions and definitions of \cite[Section 2]{chen2018volume}.
For the remainder of this paper, let us fix a pair $(r,s)\in \mathbb{N}^2$ such that $r\geq3$ and $s\geq1$. 
We will only specify $s=2$ when studying the Chen--Yang volume conjecture.
\begin{notation}
	Let $\frac{\mathbb{N}}{2}=\left\{0,\frac{1}{2},1,\frac{3}{2},...\right\}$ denote the set of non-negative half-integers. 
	Let $\frac{\mathbb{N}_{odd}}{2}=\left\{\frac{1}{2},\frac{3}{2},\frac{5}{2},...\right\}$ denote the set of non-negative half-odd-integers. 
	Let $I_r$ denote the subset $\left\{0,\frac{1}{2},1,...,\frac{r-2}{2}\right\}$ of $\frac{\mathbb{N}}{2}$.

As a convention, we let $\sqrt{-x}=I\sqrt{x}$ for $x\geq0$.  
\end{notation}

\begin{definition}\label{admissibility} A triple $(i,j,k)$ of elements of $I_r$ is called \emph{admissible} if it satisfies the following conditions:
	\begin{enumerate}[(i)]
		\item \begin{enumerate}
			\item $i+j\geq k$,
			\item $j+k\geq i$,
			\item $k+i\geq j$,
		\end{enumerate} 
		\item $i+j+k \in \mathbb{N}$,
		\item $i+j+k \leq r\-2$.
	\end{enumerate}
\end{definition}

\begin{definition}\label{admcolor}Let $M$ be an hyperbolic compact 3-manifold with boundary that admits an ideal triangulation $\mathcal{T}$.
		A\emph{ coloring at level r of $(M,\mathcal{T})$} is an application $c:\mathcal{T}^{1,\sim} \to I_r$.
		The coloring is called \emph{admissible} if for every $T \in \mathcal{T}^{3}$, the triples $\left(c([e_{01}]), c([e_{02}]), c([e_{03}])\right)$, $\left(c([e_{01}]), c([e_{12}]), c([e_{13}])\right)$, $\left(c([e_{02}]), c([e_{12}]), c([e_{23}])\right)$ and $\left(c([e_{03}]), c([e_{13}]), c([e_{23}])\right)$ are admissible, where $[e_{kl}] \in \mathcal{T}^{1,\sim}$ denotes the equivalence class under $\sim$ of the edge $e_{kl}$ of $T$ (for $k,l \in \{0,1,2,3\}$ such that $k\neq l$).
\end{definition}

	\subsection{Turaev--Viro invariants}

Recall that  $r,s\in \mathbb{N}$ are such that $r\geq3$ and $s\geq1$.

	For $n \in \mathbb{N}$, the \emph{quantum number} $[n]$ is the real number defined by $$[n]:=\frac{\sin\left(\frac{ns\pi}{r}\right)}{\sin\left(\frac{s\pi}{r}\right)} \in \mathbb{R}.$$
	For $n \in \mathbb{N}$, the \emph{quantum factorial} $[n]!$ is defined by $$[n]!:=[n][n\-1]...[2][1]\in \mathbb{R},$$ and, as a convention, $[0]!=1$.
	For an admissible triple $(i,j,k) \in (I_r)^3$, we define $$\Delta(i,j,k):=\sqrt{\frac{[i+j\-k]![i\-j+k]![\-i+j+k]!}{[i+j+k+1]!}}.$$

\begin{definition}[Quantum $6j$-symbols]\label{quantumsymbols1} 
	Let $(i,j,k,l,m,n)$ be a 6-tuple of elements of $I_r$ such that $(i,j,k)$, $(j,l,n)$, $(i,m,n)$ and $(k,l,m)$ are admissible.

	Let $T_1=i+j+k$, $T_2=j+l+n$, $T_3=i+m+n$, $T_4=k+l+m$, $Q_1=i+j+l+m$, $Q_2=i+k+l+n$ and $Q_3=j+k+m+n$.

	Then \emph{the quantum $6j$-symbol} for the 6-tuple $(i,j,k,l,m,n)$ is defined by 
	\begin{align*}
	\begin{vmatrix} 
	i & j & k \\ 
	l & m & n
	\end{vmatrix}:= & \sqrt{\-1}^{-2(i+j+k+l+m+n)} \Delta(i,j,k)\Delta(j,l,n)\Delta(i,m,n)\Delta(k,l,m) \\
	\cdot & \sum_{z=\max\{T_1,T_2,T_3,T_4\}}^{\min\{Q_1,Q_2,Q_3\}}\frac{(\-1)^z[z+1]!}{[z\-T_1]![z\-T_2]![z\-T_3]![z\-T_4]![Q_1\-z]![Q_2\-z]![Q_3\-z]!}.
	\end{align*}
\end{definition}

\begin{proposition}[Allowed symbol permutations]\label{prop:allowedperm}Let $(i,j,k,l,m,n)$ be a 6-tuple of elements of $I_r$ such that $(i,j,k)$, $(j,l,n)$, $(i,m,n)$, $(k,l,m)$ are admissible.

	Then, we have the following allowed permutations:
	$$\begin{vmatrix} 
	i & j & k \\ 
	l & m & n
	\end{vmatrix} 
	=
	\begin{vmatrix} 
	j & i & k \\ 
	m & l & n
	\end{vmatrix}
	= 
	\begin{vmatrix} 
	i & k & j \\ 
	l & n & m
	\end{vmatrix}
	= 
	\begin{vmatrix} 
	i & m & n \\ 
	l & j & k
	\end{vmatrix}
	= 
	\begin{vmatrix} 
	l & m & k \\ 
	i & j & n
	\end{vmatrix}
	= 
	\begin{vmatrix} 
	l & j & n \\ 
	i & m & k
	\end{vmatrix}.$$
\end{proposition}

Let $M$ be a pseudo-$3$-manifold that admits an  triangulation $\mathcal{T}$.
	Let $R \subset \mathcal{T}^{0,\sim}$ denote the set of regular vertices. We define the \emph{regular vertices term} as $N:=\left(\sum_{i\in I_r}w_i^2\right)^{-|R|}$.
	Let $c: \mathcal{T}^{1,\sim} \to I_r$ be an admissible coloring at level $r$ of $(M,\mathcal{T})$. For $\eta \in  \mathcal{T}^{1,\sim}$, we define the \emph{edge term} $|\eta|_c:=w_{c(\eta)}$, where $w_i:=(\-1)^{2i}[2i+1]$ for $i \in I_r$.
	For $T \in \mathcal{T}^{3}$, we define the \emph{tetrahedron term} as $$|T|_c=\begin{vmatrix} 
	c([e_{01}]) & c([e_{02}]) & c([e_{12}]) \\ 
	c([e_{23}]) & c([e_{13}]) & c([e_{03}])
	\end{vmatrix},$$
	where $[e_{kl}] \in \mathcal{T}^{1,\sim}$ denotes the equivalence class under $\sim$ of the edge $e_{kl}$ of $T$ (for $k,l \in \{0,1,2,3\}$ such that $k\neq l$).

\begin{definition}[Turaev--Viro invariant]\label{TuraevVirosum}
	Let $m=\left|\mathcal{T}^{1,\sim}\right|$.
	Let 	$\mathcal{A}_r(M,\mathcal{T}):=$
$$\left \{(c(\eta_0), ..., c(\eta_m)) \ | \ c :\mathcal{T}^{1,\sim} \to I_r\text{ is an admissible coloring at level } r \text{  of }(M,\mathcal{T})\right \}.$$ 
	We define the \emph{Turaev--Viro invariant} of $M$ as
	$$TV_{r,s}(M,\mathcal{T}):= N \sum_{(c(\eta_0), ..., c(\eta_m)) \in \mathcal{A}_r(M,\mathcal{T})} \prod_{i=0}^m |\eta_i|_c  \prod_{T \in X^3} |T|_c. $$
\end{definition}

In \cite[Theorem 2.6]{chen2018volume}, Chen and Yang prove that the previously defined Turaev--Viro invariant of $M$ does not depend on the triangulation $\mathcal{T}$.

	\subsection{Volume conjecture}

We can now state the Chen--Yang volume conjecture:

	\begin{Conjecture}[\cite{chen2018volume}, Conjecture 1.1]\label{conj:vol}
		Let $M$ be a compact hyperbolic 3-manifold. Then for $r$ running over all odd integers such that $r \geq 3$,
		$$\lim_{r \to \infty} QV_{r,2}(M) = 
		\lim_{r \to \infty} \frac{2 \pi }{r-2} \log \left( TV_{r,2} (M) \right) = \Vol(M), $$ where $ QV_{r,2}(M) := \frac{2 \pi }{r-2} \log \left( TV_{r,2} (M) \right)$ and
		 $\Vol(M)$ is the hyperbolic volume of $M$.
	\end{Conjecture}

	\begin{remark}\label{rem:conj:vol:more}
In \cite[Section 6.1]{chen2015volume}, Chen and Yang discuss the potential asymptotic behavior of $QV_{r,2}(M)$ in more detail than in Conjecture \ref{conj:vol}. In particular, observations for specific $M$ with one boundary component lead them to ask if
	$$QV_{r,2}(M) = \Vol(M) + b \dfrac{\ln(r-2)}{r-2} + c\dfrac{1}{r-2} + O\left (\dfrac{1}{(r-2)^2}\right ),$$
	where the numbers $b,c$ would depend only on $M$. In the examples they study, they find $b$ to be close to $\pi$ when $M$ has one toroidal boundary component and $-3 \pi$ when $M$ has one geodesic boundary component of genus $2$. We rephrased this general question as Conjecture \ref{conj:vol:bc}.
	\end{remark}

We will study in Sections \ref{sec:num:disc:M2} and \ref{sec:code:asymp} how likely these asymptotic behaviors seem for the manifolds $M_2, M_3, \ldots, M_7$ (which have two boundary components each, of different genera).

\section{On Frigerio's manifolds $M_g$}\label{sec:frigerio}

\subsection{Frigerio's construction}

This section follows closely the construction in \cite{Frig}.
Let $g \in \mathbb{N}_{\geq 2}$.
Let $S^3$ be the one point compactification of $\mathbb{R}^3$.

\begin{figure}[!h]
	\centering
	\includegraphics[scale=0.17]{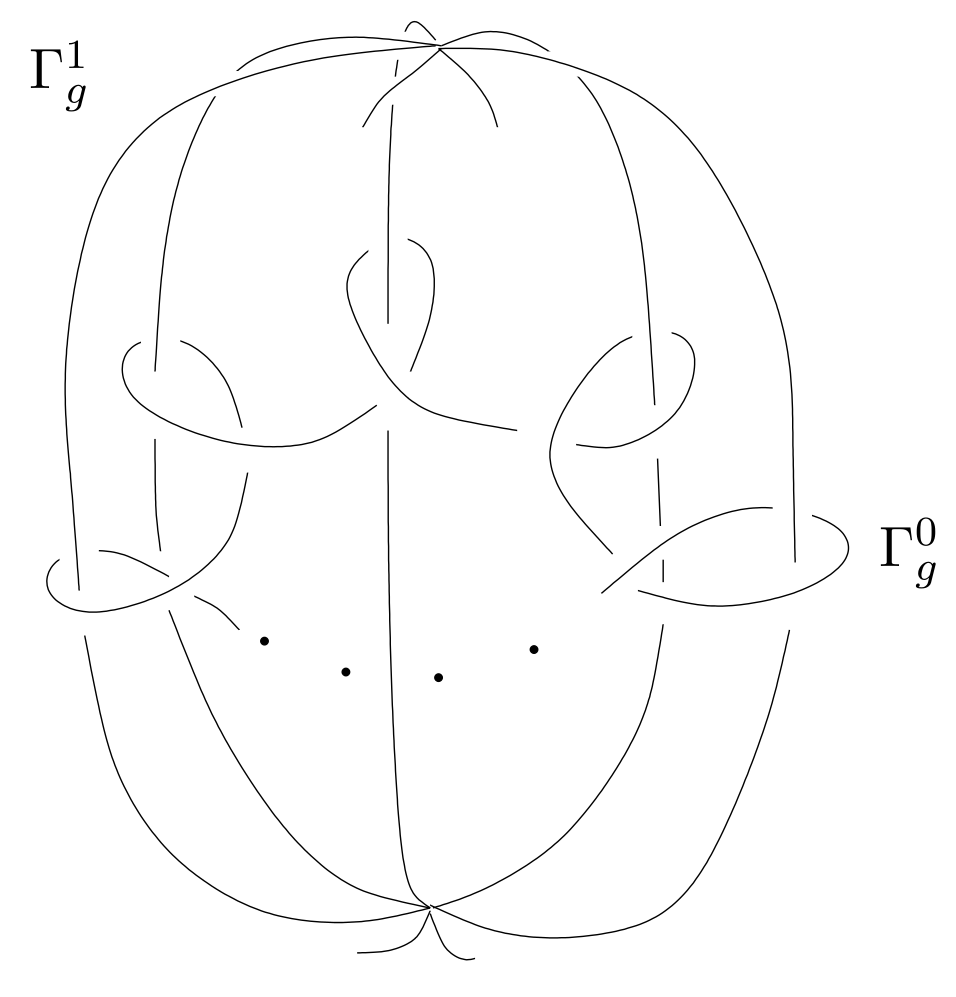}
	\caption{$\Gamma_g$ has two components: $\Gamma^0_g$ is a knot and $\Gamma^1_g$ is a graph with $g+1$ edges and two vertices. Source: \cite[Figure 1]{Frig}}\label{graphg}
\end{figure}

\begin{figure}[!h]
	\centering
	\includegraphics[scale=0.15]{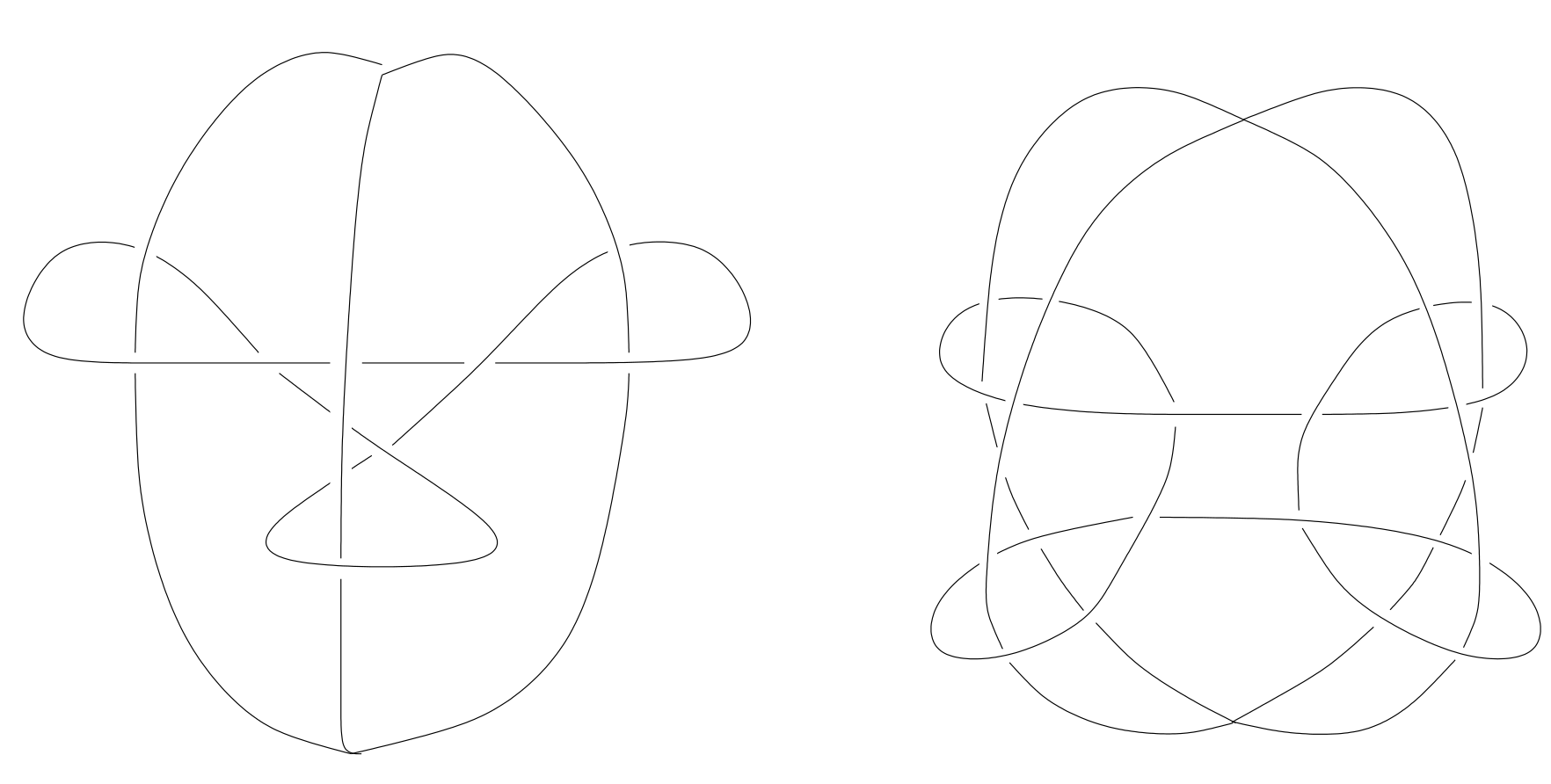}
	\caption{The graphs $\Gamma_2$ (on the left) and $\Gamma_3$ (on the right). Source: \cite[Figure 2]{Frig}}\label{graph23}
\end{figure}

	Let $\Gamma_g \subset S^3$ be the graph shown in Figure \ref{graphg} (see Figure \ref{graph23} for the cases $g=2$ and $g=3$). Let us denote by $\Gamma^0_g$ and $\Gamma^1_g$ the two connected components of $\Gamma_g$, where  $\Gamma^0_g$ is a knot and $\Gamma^1_g$ has two vertices and $g+1$ edges. Let $U(\Gamma^0_g)$ and $U(\Gamma^1_g)$  denote open regular neighbourhoods of $\Gamma^0_g$ and $\Gamma^1_g$. 
		Then $M_g$ is defined as the compact 3-manifold $M_g := S^3\setminus\left\{U(\Gamma^0_g),U(\Gamma^1_g)\right\}$ with boundary $\partial M = (\partial M)_0 \cup (\partial M)_1$ such that $(\partial M)_0=\partial U(\Gamma^0_g) \cong S^1 \times S^1$ and  $ (\partial M)_1=\partial U(\Gamma^1_g) \cong \Sigma_g$ a genus $g$ surface.

\begin{remark}
	The 3-manifold $M_g$ is the exterior of a knot in the handlebody of genus $g$ as seen in Figure \ref{handle}. 
\end{remark}

\begin{figure}[!h]
	\centering
	\includegraphics[scale=0.15]{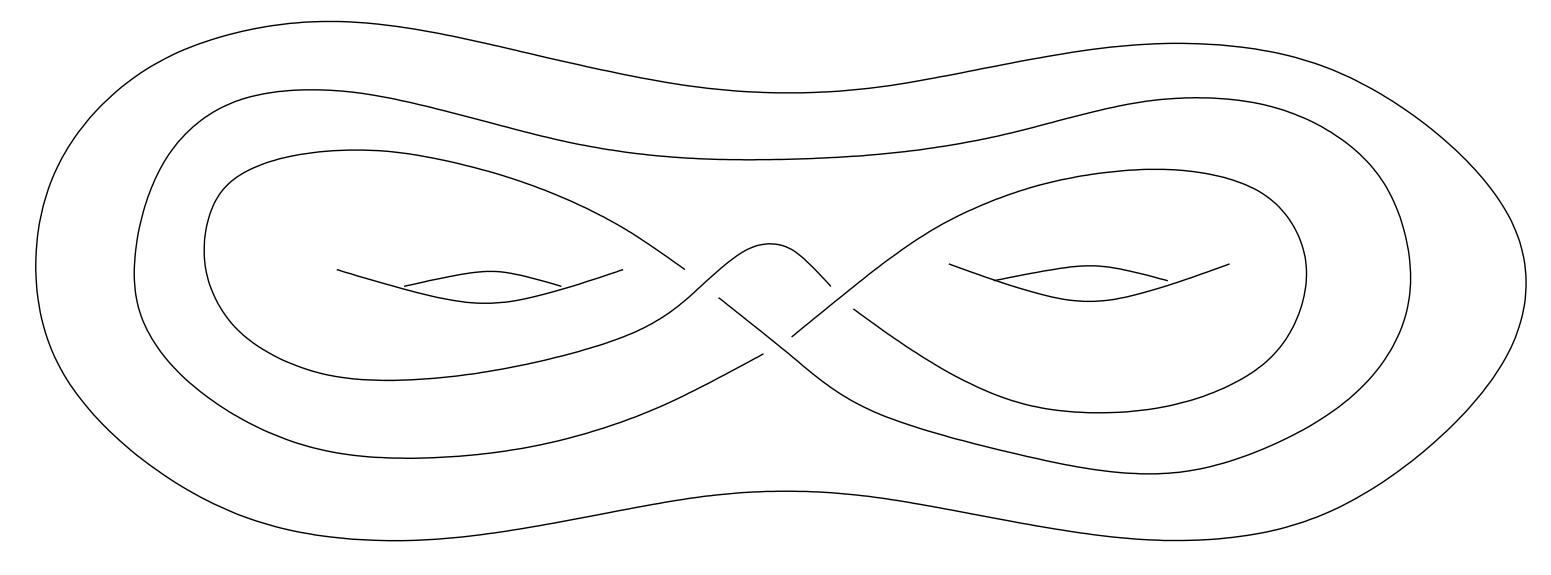}
	\caption{The 3-manifold $M_2$ as a knot complement in the handlebody of genus 2. Source: \cite[Figure 3]{Frig}}\label{handle}
\end{figure}

In \cite[Section 2]{Frig}, Frigerio constructs an ideal triangulation of $M_g$. This construction is illustrated for the cases $g=2$ and $g=3$ in Figure \ref{P3}.
Let $P_{g}$ be the double cone with apices $ap_1$ and $ap_2$ and based on the regular ($2g+2$)-gon whose vertices are $p_0$, $p_1$, ..., $p_{2g+1}$. Let $\mathring{P}_{g}$ be $P_{g}$ with its vertices removed. 
Let $Y_{g}$ be the topological space obtained by gluing the faces of $\mathring{P}_{g}$ according to the following rules:
\begin{itemize}
	\item For any $i=0,2,...,2g$ , the face $[ap_1,p_i,p_{i+1}]$ is identified with the face $[p_{i+1},p_{i+2},v_2]$ (with $v_1$ identified with $p_{i+1}$, $p_{i}$ identified with $p_{i+2}$ and $p_{i+1}$ identified with $ap_2$),
	\item For any $i=1,3,...,2g+1$ , the face $[ap_1,p_i,p_{i+1}]$ is identified with the face $[p_{i+2},v_2,p_{i+1}]$ (with $ap_1$ identified with $p_{i+2}$, $p_{i}$ identified with $ap_2$ and $p_{i+1}$ identified with $p_{i+1}$).
\end{itemize}

\begin{proposition}[Proposition 2.1, \cite{Frig}]\label{XM}For any $g\geq2$, $Y_{g}$ is homeomorphic to the interior of $M_{g}$.
\end{proposition}

We can then subdivide $P_{g}$ into $2g+2$ tetrahedra by adding the vertical edge between the two apices $ap_1$ and $ap_2$ (in red in Figure \ref{P3}). 
Such tetrahedra give an ideal triangulation of $M_g$.

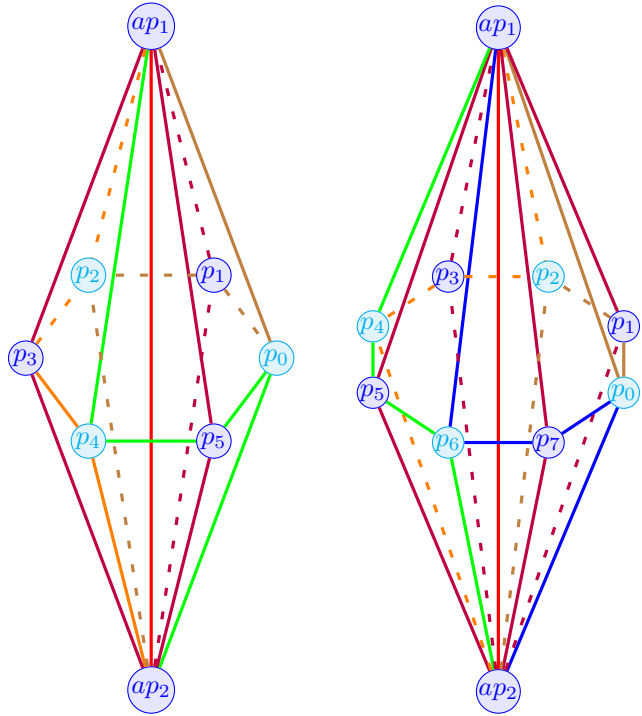
\begin{figure}[!h]
	\centering
	\begin{tikzpicture}[scale=0.55,inner sep=0.2mm, centered]
	
	\node[two-hole] (V1) at (0,8) {$ap_1$};
	\node[two-hole] (V2) at (0,-8) {$ap_2$};
	\node[tore] (P0) at (3,0) {$p_0$};
	\node[two-hole] (P1) at (1.5,2) {$p_1$};
	\node[tore] (P2) at (-1.5,2) {$p_2$};
	\node[two-hole] (P3) at (-3,0) {$p_3$};
	\node[tore] (P4) at (-1.5,-2) {$p_4$};
	\node[two-hole] (P5) at (1.5,-2) {$p_5$};
	
	\draw[red,very thick] (V2) -- (V1);
	\draw[front,brown] (V1) -- (P0);
	\draw[front,purple] (P5) -- (V1);
	\draw[front,green] (V1) -- (P4);
	\draw[front,purple] (P3) -- (V1);
	\draw[back,orange] (V1) -- (P2);
	\draw[back,purple] (P1) -- (V1);

	\draw[front,green] (V2) -- (P0);
	\draw[front,purple] (V2) -- (P5);
	\draw[front,orange] (V2) -- (P4);
	\draw[front,purple] (V2) -- (P3);
	\draw[back,brown] (V2) -- (P2);
	\draw[back,purple] (V2) -- (P1);
	
	\draw[front,green] (P5) -- (P0);
	\draw[front,orange] (P3) -- (P4);
	\draw[front,green] (P5) -- (P4);
	
	\draw[back,orange] (P3) -- (P2);
	\draw[back,brown] (P1) -- (P2);
	\draw[back,brown] (P1) -- (P0);
	
	\end{tikzpicture}
	\qquad
	\begin{tikzpicture}[scale=0.55,inner sep=0.05mm, centered]

	\node[two-hole] (V1) at (0,8) {$ap_1$};
	\node[two-hole] (V2) at (0,-8) {$ap_2$};
	\node[tore] (P0) at (3,-0.8) {$p_0$};
	\node[two-hole] (P1) at (3,0.8) {$p_1$};
	\node[tore] (P2) at (1.2,2) {$p_2$};
	\node[two-hole] (P3) at (-1.2,2) {$p_3$};
	\node[tore] (P4) at (-3,0.8) {$p_4$};
	\node[two-hole] (P5) at (-3,-0.8) {$p_5$};
	\node[tore] (P6) at (-1.2,-2) {$p_6$};
	\node[two-hole] (P7) at (1.2,-2) {$p_7$};

	\draw[very thick,red] (V2) -- (V1);
	\draw[front,brown] (V1) -- (P0);
	\draw[front,purple] (P5) -- (V1);
	\draw[front,green] (V1) -- (P4);
	\draw[back,purple] (P3) -- (V1);
	\draw[back,orange] (V1) -- (P2);
	\draw[front,purple] (P1) -- (V1);
	\draw[front,blue] (V1) -- (P6);
	\draw[front,purple] (P7) -- (V1);

	\draw[front,blue] (V2) -- (P0);
	\draw[front,purple] (V2) -- (P7);
	\draw[front,green] (V2) -- (P6);
	\draw[front,purple] (V2) -- (P5);
	\draw[back,orange] (V2) -- (P4);
	\draw[back,purple] (V2) -- (P3);
	\draw[back,brown] (V2) -- (P2);
	\draw[back,purple] (V2) -- (P1);
	
	\draw[front,blue] (P7) -- (P0);
	\draw[back,orange] (P3) -- (P4);
	\draw[front,green] (P5) -- (P4);
	\draw[front,green] (P5) -- (P6);
	\draw[front,blue] (P7) -- (P6);
	\draw[back,orange] (P3) -- (P2);
	\draw[back,brown] (P1) -- (P2);
	\draw[front,brown] (P1) -- (P0);
	
	\end{tikzpicture}

	\caption{Representation of $Y_2$ (left) and $Y_3$ (right) with the red vertical edge between the two apices and identified edges and vertices have the same color.}\label{P3}
\end{figure}

\subsection{Ordered triangulations and comb representation}

We refer to \cite{aribi2020geometric, kashaev2012tqft} for the following definitions.
A triangulation is called \textit{ordered} when it is endowed with an order on each quadruplet of vertices of each tetrahedron, such that the face gluings respect the vertex order. An equivalent property is that we can orient the edges of the triangulation in a compatible way with the face gluings and such that there are no cycles of length $3$.

\begin{figure}[!h]
	\centering
	
	\begin{tikzpicture}[scale=0.7,inner sep=0.2mm, centered]
	\node[rod] (ab) at (-12,0) {$0$};
	\node[rod] (bb) at (-8,0) {$2$};
	\node[rod] (cb) at (-6,0) {$3$};
	\node[rod] (db) at (-10,0) {$1$};
	\draw[comb] (ab) -- (db);
	\draw[comb] (cb) -- (bb);
	\draw[comb] (bb) -- (db);
	
	\node[rod] (a) at (-2.4,0) {$v_0$};
	\node[rod] (b) at (0,3.6) {$v_2$};
	\node[rod] (c) at (-0.6,-1.4) {$v_1$};
	\node[rod] (d) at (2.4,0) {$v_3$};
	
	\node (abb) at (-12,1) {};
	\node (bbb) at (-8,1) {};
	\node (cbb) at (-6,1) {};
	\node (dbb) at (-10,1) {};
	
	\draw[comb] (abb) -- (ab);
	\draw[comb] (bbb) -- (bb);
	\draw[comb] (cbb) -- (cb);
	\draw[comb] (dbb) -- (db);
	
	\draw[frontbis] (a) -- (b) ;
	\draw[frontbis] (a) -- (c) ;
	\draw[backbis] (a) -- (d);
	\draw[frontbis] (c) -- (b) ;
	\draw[frontbis] (b) -- (d) ;
	\draw[frontbis] (c) -- (d) ;
	\end{tikzpicture}
	
	\caption{An ordered tetrahedron $T$ (right) and its associated comb $C(T)$ (left).}\label{combtetra}
\end{figure}
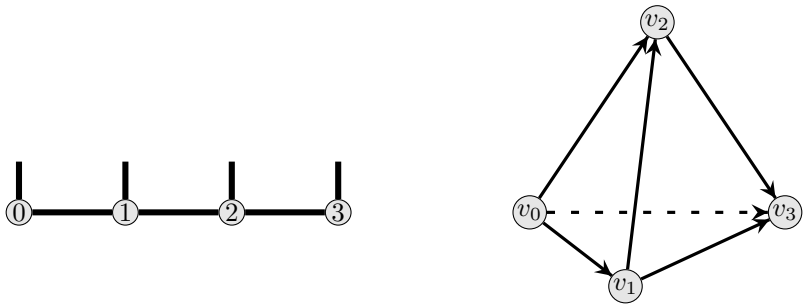

	A \textit{comb} $C$ is a line together with four spikes pointing in the same direction.
	The spikes are numbered 0, 1, 2 and 3 going from left to right, with the spikes pointing upward.
The \textit{comb representation} of an ordered triangulation consists in associating 
 a comb $C(T)$ to each tetrahedron $T$ of the triangulation, as in Figure \ref{combtetra}; each spike numbered $i$ corresponds to a face $f^i$ opposed to a vertex $v_i$ for $i \in \{0,1,2,3\}$, and we join by a line the spike $i$ of $C(T)$ to the spike $j$ of $C(T')$ if   the $i$-th face of $T$ is glued to the $j$-th face of $T'$.

 The comb representation provides a compact way of representing an ordered triangulation, while containing all the information of this triangulation. Moreover,  comb representations are convenient for studying  certain quantum invariants of ordered triangulations (such as the Teichmüller TQFT \cite{aribi2020geometric}). In this sense, it is interesting to observe when a given triangulation admits an ordered structure. This is the case for Firgerio's triangulations $\mathcal{T}_g$, as we now state:
 
 \begin{proposition}\label{prop:ordered}
The triangulations $\mathcal{T}_g$ of the manifolds $M_g$ admit an ordered structure.
 	 \end{proposition}
 
 The proof is  quite standard, in that we exhibit a specific ordered structure (more details are in \cite{Gos}). For brevity we will only refer to the examples of $g=2$ and $g=3$ represented in Figures \ref{fig:T2} and \ref{fig:T3}. The associated comb representation for $g=2$ is drawn on Figure \ref{comb2}.

\begin{figure}[!h]
	\centering
	\begin{tikzpicture}[scale=4,inner sep=0.2mm, centered]
	\draw [comb,domain=-30:60] plot ({cos(\x)/2}, {sin(\x)/2});
	\draw [comb,domain=90:180] plot ({cos(\x)/2}, {sin(\x)/2});
	\draw [comb,domain=210:300] plot ({cos(\x)/2}, {sin(\x)/2});
	
	\draw [comb,domain=0:90] plot ({cos(\x)}, {sin(\x)});
	\draw [comb,domain=120:210] plot ({cos(\x)}, {sin(\x)});
	\draw [comb,domain=240:330] plot ({cos(\x)}, {sin(\x)});
	
	\node[rod] (a0) at ({cos(0)},{sin(0)}) {$0$};
	\node (ab0) at ({cos(0)*7/8},{sin(0)*7/8}) {};
	\draw[comb] (a0) -- (ab0);
	\node[rod] (a1) at ({cos(30)},{sin(30)}) {$1$};
	\node (ab1) at ({cos(30)*7/8},{sin(30)*7/8}) {};
	\draw[comb] (a1) -- (ab1);
	\node[rod] (a2) at ({cos(60)},{sin(60)}) {$2$};
	\node (ab2) at ({cos(60)*7/8},{sin(60)*7/8}) {};
	\draw[comb] (a2) -- (ab2);
	\node[rod] (a3) at ({cos(90)},{sin(90)}) {$3$};
	\node (ab3) at ({cos(90)*7/8},{sin(90)*7/8}) {};
	\draw[comb] (a3) -- (ab3);

	\node[rod] (b0) at ({cos(60)/2},{sin(60)/2}) {$0$};
	\node (bb0) at ({cos(60)*5/8},{sin(60)*5/8}) {};
	\draw[comb] (b0) -- (bb0);
	\node[rod] (b1) at ({cos(30)/2},{sin(30)/2}) {$1$};
	\node (bb1) at ({cos(30)*5/8},{sin(30)*5/8}) {};
	\draw[comb] (b1) -- (bb1);
	\node[rod] (b2) at ({cos(0)/2},{sin(0)/2}) {$2$};
	\node (bb2) at ({cos(0)*5/8},{sin(0)*5/8}) {};
	\draw[comb] (b2) -- (bb2);
	\node[rod] (b3) at ({cos(-30)/2},{sin(-30)/2}) {$3$};
	\node (bb3) at ({cos(-30)*5/8},{sin(-30)*5/8}) {};
	\draw[comb] (b3) -- (bb3);

	\node[rod] (c0) at ({cos(120)},{sin(120)}) {$0$};
	\node (cb0) at ({cos(120)*7/8},{sin(120)*7/8}) {};
	\draw[comb] (c0) -- (cb0);
	\node[rod] (c1) at ({cos(150)},{sin(150)}) {$1$};
	\node (cb1) at ({cos(150)*7/8},{sin(150)*7/8}) {};
	\draw[comb] (c1) -- (cb1);
	\node[rod] (c2) at ({cos(180)},{sin(180)}) {$2$};
	\node (cb2) at ({cos(180)*7/8},{sin(180)*7/8}) {};
	\draw[comb] (c2) -- (cb2);
	\node[rod] (c3) at ({cos(210)},{sin(210)}) {$3$};
	\node (cb3) at ({cos(210)*7/8},{sin(210)*7/8}) {};
	\draw[comb] (c3) -- (cb3);

	\node[rod] (d0) at ({cos(180)/2},{sin(180)/2}) {$0$};
	\node (db0) at ({cos(180)*5/8},{sin(180)*5/8}) {};
	\draw[comb] (d0) -- (db0);
	\node[rod] (d1) at ({cos(150)/2},{sin(150)/2}) {$1$};
	\node (db1) at ({cos(150)*5/8},{sin(150)*5/8}) {};
	\draw[comb] (d1) -- (db1);
	\node[rod] (d2) at ({cos(120)/2},{sin(120)/2}) {$2$};
	\node (db2) at ({cos(120)*5/8},{sin(120)*5/8}) {};
	\draw[comb] (d2) -- (db2);
	\node[rod] (d3) at ({cos(90)/2},{sin(90)/2}) {$3$};
	\node (db3) at ({cos(90)*5/8},{sin(90)*5/8}) {};
	\draw[comb] (d3) -- (db3);

	\node[rod] (e0) at ({cos(240)},{sin(240)}) {$0$};
	\node (eb0) at ({cos(240)*7/8},{sin(240)*7/8}) {};
	\draw[comb] (e0) -- (eb0);
	\node[rod] (e1) at ({cos(270)},{sin(270)}) {$1$};
	\node (eb1) at ({cos(270)*7/8},{sin(270)*7/8}) {};
	\draw[comb] (e1) -- (eb1);
	\node[rod] (e2) at ({cos(300)},{sin(300)}) {$2$};
	\node (eb2) at ({cos(300)*7/8},{sin(300)*7/8}) {};
	\draw[comb] (e2) -- (eb2);
	\node[rod] (e3) at ({cos(330)},{sin(330)}) {$3$};
	\node (eb3) at ({cos(330)*7/8},{sin(330)*7/8}) {};
	\draw[comb] (e3) -- (eb3);

	\node[rod] (f0) at ({cos(300)/2},{sin(300)/2}) {$0$};
	\node (fb0) at ({cos(300)*5/8},{sin(300)*5/8}) {};
	\draw[comb] (f0) -- (fb0);
	\node[rod] (f1) at ({cos(270)/2},{sin(270)/2}) {$1$};
	\node (fb1) at ({cos(270)*5/8},{sin(270)*5/8}) {};
	\draw[comb] (f1) -- (fb1);
	\node[rod] (f2) at ({cos(240)/2},{sin(240)/2}) {$2$};
	\node (fb2) at ({cos(240)*5/8},{sin(240)*5/8}) {};
	\draw[comb] (f2) -- (fb2);
	\node[rod] (f3) at ({cos(210)/2},{sin(210)/2}) {$3$};
	\node (fb3) at ({cos(210)*5/8},{sin(210)*5/8}) {};
	\draw[comb] (f3) -- (fb3);

	\draw[link] (ab1) -- (bb1);
	\draw[link] (ab2) -- (bb0);
	\draw[link] (ab3) -- (db3);
	
	\draw[link] (cb1) -- (db1);
	\draw[link] (cb2) -- (db0);
	\draw[link] (cb3) -- (fb3);
	
	\draw[link] (eb1) -- (fb1);
	\draw[link] (eb2) -- (fb0);
	\draw[link] (eb3) -- (bb3);
	
	\draw[link] (db2) -- ({cos(120)*11/16},{sin(120)*11/16});
	\draw[link,domain=120:100] plot ({cos(\x)*11/16}, {sin(\x)*11/16});
	\draw[link] ({cos(100)*11/16},{sin(100)*11/16}) -- ({cos(100)*19/16},{sin(100)*19/16});
	\draw[link,domain=-10:100] plot ({cos(\x)*19/16}, {sin(\x)*19/16});
	\draw[link] ({cos(-10)*19/16},{sin(-10)*19/16}) -- ({cos(-10)*13/16},{sin(-10)*13/16});
	\draw[link,domain=-10:0] plot ({cos(\x)*13/16}, {sin(\x)*13/16});
	\draw[link] (ab0) -- ({cos(0)*13/16},{sin(0)*13/16});
	
	\draw[link] (fb2) -- ({cos(240)*11/16},{sin(240)*11/16});
	\draw[link,domain=240:220] plot ({cos(\x)*11/16}, {sin(\x)*11/16});
	\draw[link] ({cos(220)*11/16},{sin(220)*11/16}) -- ({cos(220)*19/16},{sin(220)*19/16});
	\draw[link,domain=110:220] plot ({cos(\x)*19/16}, {sin(\x)*19/16});
	\draw[link] ({cos(110)*19/16},{sin(110)*19/16}) -- ({cos(110)*13/16},{sin(110)*13/16});
	\draw[link,domain=110:120] plot ({cos(\x)*13/16}, {sin(\x)*13/16});
	\draw[link] (cb0) -- ({cos(120)*13/16},{sin(120)*13/16});
	
	\draw[link] (bb2) -- ({cos(0)*11/16},{sin(0)*11/16});
	\draw[link,domain=0:-20] plot ({cos(\x)*11/16}, {sin(\x)*11/16});
	\draw[link] ({cos(-20)*11/16},{sin(-20)*11/16}) -- ({cos(-20)*19/16},{sin(-20)*19/16});
	\draw[link,domain=-130:-20] plot ({cos(\x)*19/16}, {sin(\x)*19/16});
	\draw[link] ({cos(-130)*19/16},{sin(-130)*19/16}) -- ({cos(-130)*13/16},{sin(-130)*13/16});
	\draw[link,domain=-130:-120] plot ({cos(\x)*13/16}, {sin(\x)*13/16});
	\draw[link] (eb0) -- ({cos(-120)*13/16},{sin(-120)*13/16});
	
	\end{tikzpicture}
	\caption{Comb representation of $\mathcal{T}_2$.}\label{comb2}
\end{figure}
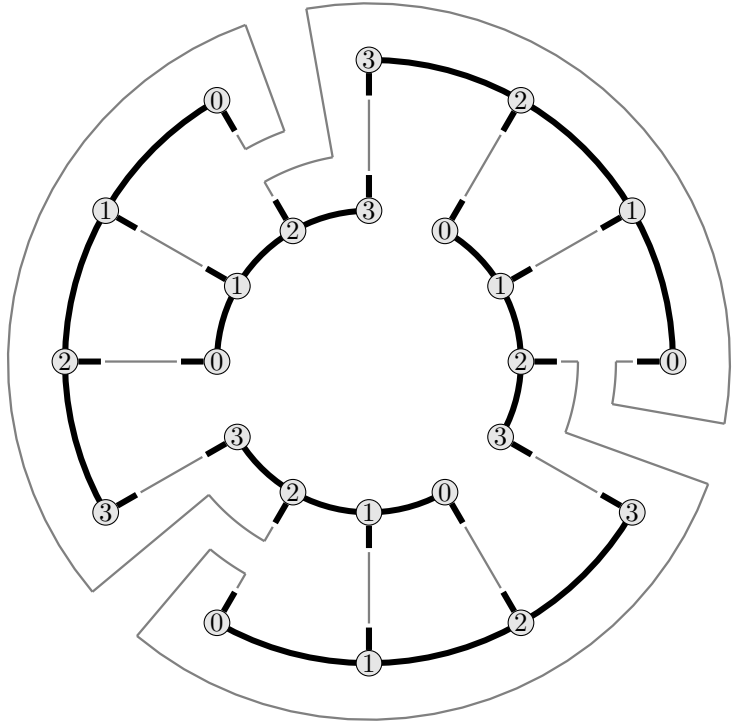

	\subsection{Hyperbolic structure}\label{sec:Mg:hyp}
	
In \cite[Section 2]{Frig}, Frigerio provides the unique angle structure on the tetrahedra of the ideal triangulation $\mathcal{T}_g$ that corresponds to the unique complete hyperbolic structure on the manifold $M_g$. 
Setting $\alpha_g=\frac{\pi}{2g+2}$ , $\beta_g=2 \alpha_g$, $\gamma_g= \arccos((2\cos\alpha_g)^{-1})$ and $\delta_g = \pi - 2 \gamma_g$, the complete angle structure on $\mathcal{T}_g$ follows the pattern of Figure \ref{fig:anglesg}: half of the $2g+2$ tetrahedra are as the one on the left, and the other half are as the one on the right.

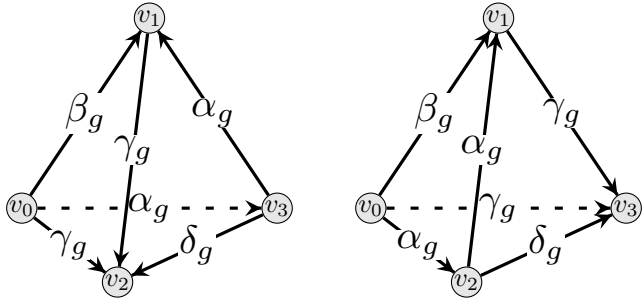
\begin{figure}[!h]
	\centering
	\begin{tikzpicture}[scale=0.7,inner sep=0.05mm, centered]

	\node[rod] (a) at (-2.4,0) {$v_0$};
	\node[rod] (b) at (0,3.6) {$v_1$};
	\node[rod] (c) at (-0.6,-1.4) {$v_2$};
	\node[rod] (d) at (2.4,0) {$v_3$};
	
	\draw[frontbis] (a) -- (b) node[labels] { $\beta_g$};
	\draw[frontbis] (a) -- (c) node[labels] { $\gamma_g$};
	\draw[backbis] (a) -- (d) node[labels] { $\alpha_g$};
	\draw[frontbis] (b) -- (c) node[labels] { $\gamma_g$};
	\draw[frontbis] (d) -- (b) node[labels] { $\alpha_g$};
	\draw[frontbis] (d) -- (c) node[labels] { $\delta_g$};
	
	\end{tikzpicture}
	\qquad
	\begin{tikzpicture}[scale=0.7,inner sep=0.05mm, centered]
	\node[rod] (a) at (-2.4,0) {$v_0$};
	\node[rod] (b) at (0,3.6) {$v_1$};
	\node[rod] (c) at (-0.6,-1.4) {$v_2$};
	\node[rod] (d) at (2.4,0) {$v_3$};
	
	\draw[frontbis] (a) -- (b) node[labels] { $\beta_g$};
	\draw[frontbis] (a) -- (c) node[labels] { $\alpha_g$};
	\draw[backbis] (a) -- (d) node[labels] { $\gamma_g$};
	\draw[frontbis] (c) -- (b) node[labels] { $\alpha_g$};
	\draw[frontbis] (b) -- (d) node[labels] { $\gamma_g$};
	\draw[frontbis] (c) -- (d) node[labels] { $\delta_g$};
	
	\end{tikzpicture}
	\caption{Frigerio's angle structure on $\mathcal{T}_g$}\label{fig:anglesg}
\end{figure}

Both tetrahedra of Figure \ref{fig:anglesg} have the same volume, and the volume of $M_g$ is exactly $2g+2$ times higher. Values of these volumes are computed with {SageMath} (see Section \ref{sec:codehyperbolicvolumeg}) and listed in Figure \ref{fig:volMg}.

	\subsection{Admissible colorings}

In this section we study the set of admissible colorings for the triangulations $\mathcal{T}_g$.
The following three lemmas follow from standard arguments, are not restricted to the manifolds $M_g$, and will be used in the proof of Theorem \ref{thm:allowedstates}. See \cite{Gos} for details.
	
\begin{lemma}\label{lemmaikk}
	Let $i,k \in I_r$. Then $(i,i,k)$ is admissible if and only if $k \in \mathbb{N}$ and $\frac{k}{2} \leq i \leq \frac{r-2-k}{2}$.
\end{lemma}

\begin{lemma}\label{ijklmn} Let $i,j,k,l,m \in I_r$ such that $(i,j,k)$ and $(k,l,m)$ are admissible, then $i+j+l+m\in \mathbb{N}$.
\end{lemma}

\begin{lemma}\label{lemmaijkn}
	Let $i,j,k \in I_r$ such that $k \in \mathbb{N}$. Then the following conditions are equivalent:
	\begin{enumerate}[(i)]
		\item $i+j+k \in \mathbb{N}$,
		\item $i+j \in \mathbb{N}$,
		\item either $i,j \in \mathbb{N}$ or $i,j \in \frac{\mathbb{N}_{odd}}{2}$.
	\end{enumerate}
\end{lemma}
\begin{remark}Lemma \ref{lemmaijkn} can be applied to condition (ii) of Definition \ref{admissibility} when $i$, $j$ or $k$ is in $\mathbb{N}$.
\end{remark}

Let $g \in \mathbb{N}_{\geq2}$. For a general coloring $c$ at level $r$ of $(M_g,\mathcal{T}_g)$, 
we need the following triples of elements of $I_r$ to be admissible:
\begin{center}
	\begin{tabular}{cccc}
		$(a,b,b)$, & $(a,c_g,c_{0})$, & $(b,c_g,c_g)$, & $(b,c_g,c_{0})$, \\
		and for all $i\in\left\{0,1,2, ...,g-1\right\}$, & $(a,c_i,c_{i+1})$, & $(b,c_i,c_i)$, & $(b,c_i,c_{i+1})$. \\
	\end{tabular}
\end{center}
These admissibility conditions can be rewritten like in the next theorem.

\begin{theorem}\label{thm:allowedstates}
	Let $a$, $b$, $c_0$, $c_1$, $c_2$, ..., $c_{g} \in I_r$.
	
	The conditions $(A)$, $(B)$, $(C)$, $(D)$, $(E)$, $(F)$ and $(G)$ defined below are simultaneously satisfied if and only if the conditions $(1)$, $(2)$, $(3)$, $(4)$, $(5)$ and $(6)$ defined below are simultaneously satisfied.
	
	The conditions  $(A)$, $(B)$, $(C)$, $(D)$, $(E)$, $(F)$ and $(G)$ are written:
	\begin{enumerate}[(A)]
		\item $(a,b,b)$ is admissible,\label{conditionA}
		\item $(a,c_g,c_{0})$ is admissible,\label{conditionB}
		\item $(b,c_g,c_g)$ is admissible,\label{conditionC}
		\item $(b,c_g,c_{0})$ is admissible,\label{conditionD}
		\item $\forall i\in\left\{0,1,2, ...,g-1\right\}$, $(a,c_i,c_{i+1})$ is admissible,\label{conditionE}
		\item $\forall i\in\left\{0,1,2, ...,g-1\right\}$, $(b,c_i,c_i)$ is admissible,\label{conditionF}
		\item $\forall i\in\left\{0,1,2, ...,g-1\right\}$, $(b,c_i,c_{i+1})$ is admissible,\label{conditionG}
	\end{enumerate}

	The conditions $(1)$, $(2)$, $(3)$, $(4)$, $(5)$ and $(6)$ are written:
	\begin{enumerate}[(1)]
		\item $a \in \mathbb{N}$,\label{condition1}
		\item $b \in \mathbb{N}$,\label{condition2}
		\item either $\forall i\in \left\{0,1,2...,g\right\}$, $c_i \in \mathbb{N}$ or $\forall i\in \left\{0,1,2,...,g\right\}$, $c_i \in \frac{\mathbb{N}_{odd}}{2}$, \label{condition3}
		\item $\frac{a}{2}\leq b \leq \frac{r-2-a}{2}$,\label{condition4}
		\item \begin{enumerate}
			\item $\forall i\in \left\{0,1,2...,g-1\right\}$, $\frac{b}{2} \leq c_{i+1} \leq \frac{r-2-b}{2}$, \label{condition5a}
			\item $\forall i\in \left\{0,1,2...,g-1\right\}$, $a - c_i \leq c_{i+1} \leq r - 2 - a - c_i$,\label{condition5b}
			\item $\forall i\in \left\{0,1,2...,g-1\right\}$, $c_i- a \leq c_{i+1} \leq a + c_i$, \label{condition5c}
			\item $\forall i\in \left\{0,1,2...,g-1\right\}$, $c_i- b \leq c_{i+1} \leq b + c_i$, \label{condition5d}
		\end{enumerate} 
		\item \begin{enumerate}
			\item $\frac{b}{2} \leq c_{0} \leq \frac{r-2-b}{2}$, \label{condition6a}
			\item $a - c_g \leq c_{0} \leq r - 2 - a - c_g$, \label{condition6b}
			\item $c_g- a \leq c_{0} \leq a + c_g$, \label{condition6c}
			\item $c_g- b \leq c_{0} \leq b + c_g$. \label{condition6d}
		\end{enumerate} 
	\end{enumerate}
	
\end{theorem}

\begin{proof}
	Let $a$, $b$, $c_0$, $c_1$, $c_2$, ..., $c_{g} \in I_r$.
	
	\underline{Step 1 : $(\ref{conditionA}) \iff  (\ref{condition1}) \wedge (\ref{condition4}) $};
	
	This step follows immediately from Lemma \ref{lemmaikk}.
	
	\underline{Step 2 : $(\ref{conditionA}) \wedge (\ref{conditionE}) \iff (\ref{condition1}) \wedge (\ref{condition3}) \wedge (\ref{condition4}) \wedge (\ref{condition5b})  \wedge (\ref{condition5c})$};
	
	By Definition \ref{admissibility}, Condition $(\ref{conditionE})$ is equivalent to the following conditions :
	\begin{enumerate}[(Ei)]
		\item \begin{enumerate}[(a)]
			\item $\forall i\in \left\{0,1,2...,g-1\right\}$, $c_i + c_{i+1} \geq a$, \label{ia}
			\item $\forall i\in \left\{0,1,2...,g-1\right\}$, $c_i + a \geq c_{i+1}$, \label{ib}
			\item $\forall i\in \left\{0,1,2...,g-1\right\}$, $a + c_{i+1} \geq c_i$, \label{ic}
		\end{enumerate} 
		\item $\forall i\in \left\{0,1,2...,g-1\right\}$, $c_i + c_{i+1} + a \in \mathbb{N}$, \label{ii}
		\item $\forall i\in \left\{0,1,2...,g-1\right\}$, $c_i + c_{i+1} + a \leq r\-2$. \label{iii}
	\end{enumerate}
	
	To prove Step 2, we first remark that: $$(E\ref{ia}) \wedge (E\ref{iii}) \iff (\ref{condition5b}),$$ $$(E\ref{ib}) \wedge (E\ref{ic}) \iff (\ref{condition5c}).$$
	
	It remains to prove that : $$(\ref{conditionA}) \wedge (E\ref{ii}) \iff (\ref{condition1}) \wedge (\ref{condition3}) \wedge (\ref{condition4}).$$
	It follows from Step 1 that :
	$$(\ref{conditionA}) \wedge (E\ref{ii}) \iff  (\ref{condition1}) \wedge (\ref{condition4}) \wedge (E\ref{ii}). $$
	From Lemma \ref{lemmaijkn}, it follows that $$(\ref{condition1}) \wedge (\ref{condition4}) \wedge (E\ref{ii}) \iff (\ref{condition1}) \wedge (\ref{condition4}) \wedge \left( \forall i\in \left\{0,1,2...,g-1\right\}, \left(c_i, c_{i+1} \in \mathbb{N}\right) \lor \left(c_i, c_{i+1} \in \frac{\mathbb{N}_{odd}}{2}\right) \right).  $$
	Finally, it follows by a quick induction that:
	$$(\ref{condition1}) \wedge (\ref{condition4}) \wedge \left( \forall i\in \left\{0,1,2...,g-1\right\}, \left(c_i, c_{i+1} \in \mathbb{N}\right) \lor \left(c_i, c_{i+1} \in \frac{\mathbb{N}_{odd}}{2}\right) \right) \iff (\ref{condition1}) \wedge (\ref{condition4}) \wedge (\ref{condition3}) .$$ 
	
	\underline{Step 3 : $(\ref{conditionA})\wedge (\ref{conditionB})  \wedge (\ref{conditionE}) \iff (\ref{condition1}) \wedge (\ref{condition3}) \wedge (\ref{condition4}) \wedge (\ref{condition5b}) \wedge (\ref{condition5c}) \wedge (\ref{condition6b}) \wedge (\ref{condition6c})$};
	
	From Step 2, it follows that: $$(\ref{conditionA}) \wedge (\ref{conditionE}) \iff (\ref{condition1}) \wedge (\ref{condition3}) \wedge (\ref{condition4}) \wedge (\ref{condition5b})  \wedge (\ref{condition5c}).$$
	
	By Definition \ref{admissibility}, Condition $(\ref{conditionB})$ is equivalent to the following conditions :
	\begin{enumerate}[(Bi)]
		\item \begin{enumerate}
			\item $c_g + c_{0} \geq a$, \label{iaB}
			\item $c_g + a \geq c_{0}$, \label{ibB}
			\item $a + c_{0} \geq c_g$, \label{icB}
		\end{enumerate} 
		\item $c_g + c_{0} + a \in \mathbb{N}$, \label{iiB}
		\item $c_g + c_{0} + a \leq r\-2$.  \label{iiiB}
	\end{enumerate}
	
	To prove Step 3, we first remark that: $$(B\ref{iaB}) \wedge (B\ref{iiiB}) \iff (\ref{condition6b}),$$
	$$(B\ref{ibB}) \wedge (B\ref{icB}) \iff (\ref{condition6c}).$$
	
	Finally, let us prove the following implication which will imply Step 3: $$(\ref{conditionA}) \wedge (\ref{conditionE}) \implies (B\ref{iiB}).$$
	From Step 2, it follows that $$(\ref{conditionA}) \wedge (\ref{conditionE}) \implies (\ref{condition1}) \wedge (\ref{condition3}).$$ Furthermore, using Lemma \ref{lemmaijkn}, we have that: $$(\ref{condition1}) \wedge (\ref{condition3}) \implies (B\ref{iiB}).$$  
	
	\underline{Step 4 : $(\ref{conditionC}) \wedge (\ref{conditionF}) \iff (\ref{condition2}) \wedge (\ref{condition5a}) \wedge (\ref{condition6a})$};
	
	This step follows immediately from Lemma \ref{lemmaikk}.
	
	\underline{Step 5 : $(\ref{conditionA}) \wedge (\ref{conditionB})\wedge (\ref{conditionC})\wedge (\ref{conditionE})\wedge (\ref{conditionF}) \wedge (\ref{conditionG}) \iff $}\\
	\underline{$ (\ref{condition1})\wedge(\ref{condition2})\wedge (\ref{condition3})\wedge(\ref{condition4})\wedge (\ref{condition5a})\wedge (\ref{condition5b})\wedge (\ref{condition5c})\wedge (\ref{condition5d})\wedge (\ref{condition6a})\wedge (\ref{condition6b}) \wedge (\ref{condition6c})$};
	
	From Step 3, it follows that: $$(\ref{conditionA})\wedge (\ref{conditionB})  \wedge (\ref{conditionE}) \iff (\ref{condition1}) \wedge (\ref{condition3}) \wedge (\ref{condition4}) \wedge (\ref{condition5b}) \wedge (\ref{condition5c}) \wedge (\ref{condition6b}) \wedge (\ref{condition6c}).$$
	
	From Step 4, it follows that: $$(\ref{conditionC}) \wedge (\ref{conditionF}) \iff (\ref{condition2}) \wedge (\ref{condition5a}) \wedge (\ref{condition6a}).$$
	
	By Definition \ref{admissibility}, Condition $(\ref{conditionG})$ is equivalent to the following conditions :
	\begin{enumerate}[(Gi)]
		\item \begin{enumerate}
			\item $\forall i\in \left\{0,1,2...,g-1\right\}$, $c_i + c_{i+1} \geq b$, \label{iab}
			\item $\forall i\in \left\{0,1,2...,g-1\right\}$, $c_i + b \geq c_{i+1}$, \label{ibb}
			\item $\forall i\in \left\{0,1,2...,g-1\right\}$, $b + c_{i+1} \geq c_i$, \label{icb}
		\end{enumerate} 
		\item $\forall i\in \left\{0,1,2...,g-1\right\}$, $c_i + c_{i+1} + b \in \mathbb{N}$, \label{iib}
		\item $\forall i\in \left\{0,1,2...,g-1\right\}$, $c_i + c_{i+1} + b \leq r\-2$. \label{iiib}
	\end{enumerate}
	
	Following what precedes, in order to prove Step 5, we only need to prove that:
	\begin{enumerate}
		\item[Step 5.1]$(\ref{conditionC}) \wedge (\ref{conditionF}) \implies (G\ref{iab}) \wedge (G\ref{iiib})$,
		\item[Step 5.2]$(G\ref{ibb}) \wedge (G\ref{icb}) \iff (\ref{condition5d})$,
		\item[Step 5.3]$(\ref{conditionA}) \wedge (\ref{conditionE}) \implies (G\ref{iib})$.
	\end{enumerate}
	
	Let us prove Step 5.1. From Step 4, it follows that $$(\ref{conditionC}) \wedge (\ref{conditionF}) \implies (\ref{condition5a}) \wedge (\ref{condition6a}).$$ Furthermore, we have by adding inequalities that : $$(\ref{condition5a}) \wedge (\ref{condition6a})\implies (G\ref{iab}) \wedge (G\ref{iiib}).$$
	
	Step 5.2 follows immediately from the definitions.
	
	Finally, let us prove Step 5.3. From Step 2, it follows that $$(\ref{conditionA}) \wedge(\ref{conditionE}) \implies (\ref{condition1}) \wedge (\ref{condition3}).$$ Furthermore, using Lemma \ref{lemmaijkn}, we have that: $$(\ref{condition1}) \wedge (\ref{condition3}) \implies (G\ref{iib}).$$  
	
	This concludes the proof of Step 5.
	
	\underline{Step 6 : $(\ref{conditionA}) \wedge (\ref{conditionB}) \wedge (\ref{conditionC}) \wedge (\ref{conditionD}) \wedge (\ref{conditionE}) \wedge (\ref{conditionF})   \wedge (\ref{conditionG}) \iff$}\\
	\underline{$(\ref{condition1}) \wedge(\ref{condition2}) \wedge (\ref{condition3}) \wedge (\ref{condition4}) \wedge (\ref{condition5a}) \wedge (\ref{condition5b}) \wedge (\ref{condition5c}) \wedge (\ref{condition5d}) \wedge (\ref{condition6a}) \wedge (\ref{condition6b}) \wedge (\ref{condition6c})  \wedge (\ref{condition6d})$};
	
	We prove Step 6 almost exactly as we proved Step 5. The only difference being the $(G)$ conditions become $(D)$ conditions and $(5d)$ becomes $(6d)$.
	
	This concludes the proof of the theorem.
\end{proof}

As we will see, Theorem \ref{thm:allowedstates} can be used to simplify the formula and computation of the Turaev--Viro invariants for the manifolds 
$M_g$, and to write a code for computing these invariants numerically (see Section \ref{sec:codetvr}).

\section{Annotated code}\label{sec:code}

The following codes have been written with SageMath, a free open-source mathematics software system using Python 3.
In this section, let us fix a pair $(r,s)\in \mathbb{N}^2$ such that $r\geq3$ and $s\geq1$.
Let $g\in \N_{\geqslant 2}$. Let us recall the following notation:

\begin{notation}
	Let $\frac{\mathbb{N}}{2}=\{0,\frac{1}{2},1,...\}$ denote the set of non-negative half integers. 
	Let $\frac{\mathbb{N}_{odd}}{2}=\{\frac{1}{2},\frac{3}{2},...\}$ denote the set of non-negative half-odd-integers. 
	Let $I_r$ denote the subset $\left\{0,\frac{1}{2},1,...,\frac{r-2}{2}\right\}$ of $\frac{\mathbb{N}}{2}$. 
\end{notation}
\subsection{Computing the hyperbolic volume of $M_g$}\label{sec:computing:hyp}
In this section, we construct a function which computes the volume of a tetrahedron from its dihedral angles, via Ushijima's volume formula (see Proposition \ref{prop:formula:ush}). 
We then apply this function to the manifolds $M_g$ with triangulations $\mathcal{T}_g$.

\subsubsection{The volume formula for hyperbolic tetrahedra}\label{sec:computing:ushijima}

We follow the steps and notation of Section  \ref{sec:prelim:hyp:vol}.
We first import the \textit{NumPy} package which is an easy to use,  open source package in Python. This package  contains the function \textbf{pi} giving an approximate value of $\pi$ that we will use later on. Functions imported from this package will be written in the code with the prefix \textbf{np.}, which means we access the function inside the package \textit{NumPy}.
The function \textbf{Gramdet(A,B,C,D,E,F)} computes the determinant of the Gram matrix associated to  $A, B, C, D, E, F \in \left[ 0, \pi \right]$. 
\begin{lstlisting}[language=Python, caption={The determinant of the Gram matrix of $T$}]
import numpy as np

def Gramdet(A,B,C,D,E,F):
G = matrix([	[1,-cos(A),-cos(B),-cos(F)],
[-cos(A),1,-cos(C),-cos(E)],
[-cos(B),-cos(C),1,-cos(D)],
[-cos(F),-cos(E),-cos(D),1]])
res = G.determinant()
return res
\end{lstlisting}

The \textbf{U(z,A,B,C,D,E,F)} function computes $U(z,T)$ for a given complex number $z$ and the dihedral angles $A,B,C,D,E$ and $F$ of a given generalized tetrahedron $T$. We use the \textbf{dilog(z)} function for $z$ a complex number to compute $\mathrm{Li_2}(z)$.
\begin{lstlisting}[language=Python, caption={The complex valued fonction $U(z,T)$}]
def U(z,A,B,C,D,E,F):
a = exp(I*A)
b = exp(I*B)
c = exp(I*C)
d = exp(I*D)
e = exp(I*E)
f = exp(I*F)

z1 = a*b*d*e*z
z2 = a*c*d*f*z
z3 = b*c*e*f*z
z4 = -a*b*c*z
z5 = -a*e*f*z
z6 = -b*d*f*z
z7 = -c*d*e*z

res = 1/2*(dilog(z)+dilog(z1)+dilog(z2)+dilog(z3)-dilog(z4)-dilog(z5)-dilog(z6)-dilog(z7))
return res
\end{lstlisting}

Finally, the \textbf{TetVolum(A,B,C,D,E,F)} function uses the formula of Proposition \ref{prop:formula:ush} to compute the hyperbolic volume of a tetrahedron given its dihedral angles $A$, $B$, $C$, $D$, $E$ and $F$.
\begin{lstlisting}[language=Python, caption={The hyperbolic volume of $T$}]
def TetVolum(A,B,C,D,E,F):
a = exp(I*A)
b = exp(I*B)
c = exp(I*C)
d = exp(I*D)
e = exp(I*E)
f = exp(I*F)

det = Gramdet(A,B,C,D,E,F)

zminus = -2*((sin(A)*sin(D)+sin(B)*sin(E)+sin(C)*sin(F)-sqrt(det))/(a*d+b*e+c*f+a*b*f+a*c*e+b*c*d+d*e*f+a*b*c*d*e*f))

zplus = -2*((sin(A)*sin(D)+sin(B)*sin(E)+sin(C)*sin(F)+sqrt(det))/(a*d+b*e+c*f+a*b*f+a*c*e+b*c*d+d*e*f+a*b*c*d*e*f))

res = imag((U(zminus,A,B,C,D,E,F)-U(zplus,A,B,C,D,E,F))/2)
return res
\end{lstlisting}

\subsubsection{Applying the volume formula on tetrahedra of $\mathcal{T}_g$}\label{sec:codehyperbolicvolumeg}

Following Frigerio's computation of the complete angle structure detailed in Section \ref{sec:Mg:hyp}, we define the function \textbf{HyperbolicVolume(g)} to compute the hyperbolic volume of $M_g$ given $g \in \{2,3,4,...\}$. Thanks to the symmetries in the complete angle structure, this function computes the hyperbolic volume of one tetrahedron and multiplies it by $2g+2$ (the number of tetrahedra in the triangulation). 
We display several values for this function for some values of $g$ in Figure \ref{fig:volMg}.

\begin{lstlisting}[language=Python, caption={The hyperbolic volume of $M_g$ for a given $g$}]
def HyperbolicVolume(g):
Alphag=np.pi/(2*g+2)
Betag=2*Alphag
Gammag=arccos(1/(2*cos(Alphag)))
Deltag=np.pi - 2*Gammag

res=(2*g+2)*TetVolum(Gammag,Deltag,Gammag,Alphag,Betag,Alphag)
return res
\end{lstlisting}

\begin{figure}[!h]
	\centering
	\begin{tabular}{|c|c|c|}
		\hline 
		$g$ & $\Vol(T_g)$ & $\Vol(M_g)$ \\ 
		\hline 
		2 & 2.007682006682397 & 12.046092040094381 \\ 
		\hline 
		3 & 2.2547631818606026 & 18.03810545488482 \\ 
		\hline 
		4 & 2.3603494908554774
		& 23.603494908554772 \\ 
		\hline 
		5 & 2.415787949187158 & 28.989455390245897 \\ 
		\hline 
		6 & 2.448617485457304 & 34.28064479640226 \\ 
		\hline 
		7 & 2.469695490891516 & 39.51512785426426 \\ 
		\hline 
		8 & 2.484045062029212 & 44.71281111652581 \\ 
		\hline 
		9 & 2.494259571737797 & 49.88519143475594 \\ 
		\hline 
		10 & 2.5017908556003303 & 55.039398823207264 \\ 
		\hline 
		100 & 2.5369350366401 & 512.4608774013002 \\ 
		\hline 
		1000 & 2.5373497508910896 & 5079.774201283962 \\ 
		\hline 
	\end{tabular} 
	\caption{The hyperbolic volumes of $M_g$ and one tetrahedron $T_g$ of $\mathcal{T}_g$ for various values of $g$.}\label{fig:volMg}
\end{figure}

\subsection{Computing the Turaev--Viro invariants $TV_{r,s}(M_g,\mathcal{T}_g)$}\label{sec:codetvr}
In this section, we will construct several functions in order to compute the Turaev--Viro invariants $TV_{r,s}(M_2,\mathcal{T}_2)$ and $TV_{r,s}(M_3,\mathcal{T}_3)$ given $(r,s)\in \mathbb{N}^2$ such that $r\geq3$ and $s\geq1$.  These functions can easily be extended to compute the other Turaev--Viro invariants $TV_{r,s}(M_g,\mathcal{T}_g)$ for $g\geq4$ (which we will do for $4 \leqslant g \leqslant 7$). 

We first import the \textit{NumPy} package as in Section \ref{sec:computing:ushijima}. We also import the \textit{cmath} package which gives us mathematical functions for complex numbers, such as square roots for negative numbers. Functions imported from \textit{cmath} will be written with the prefix \textbf{cm.}.
\begin{lstlisting}[language=Python, caption={Importing NumPy and cmath}]
import numpy as np
import cmath as cm
\end{lstlisting}
The function \textbf{quantum\_number(r,s,n)} returns the quantum number $[n]$ given $r$, $s$ and $n\in \mathbb{N}$.
\begin{lstlisting}[language=Python, caption={Quantum number}]
def quantum_number(r,s,n):
res=sin(s*n*(np.pi)/r)/sin(s*(np.pi)/r)
return res
\end{lstlisting}
The function \textbf{quantum\_factorial(r,s,n)} returns the factorial of a quantum number $[n]! $ given $r$, $s$ and $n\in \mathbb{N}$.

\begin{lstlisting}[language=Python, caption={Quantum factorial}] 
def quantum_factorial(r,s,n):
if n == 0:
return 1
return prod([quantum_number(r,s,i) for i in range(1, n+1)])
\end{lstlisting}
The function \textbf{q\_big\_delta\_coeff(i,j,k,r,s)} returns the coefficients
$\Delta(i,j,k)$
given $r$, $s$ and $(i,j,k)$ an admissible triple of elements of $I_r$.
Admissibility conditions 
ensure that $i + j - k$, $i + k - j$, $j + k - i$ and $i + j + k + 1$ are in $ \mathbb{N}$ but since the values of $i$, $j$ and $k$ can be half integers, the values $i + k - j$, $j + k - i$ and $i + j + k + 1$ are stored as rational type variable. To ensure that the function \textbf{quantum\_factorial()} works, we have to change their type from rational to integers using the $int()$ function.
Since \textbf{quantum\_number()} and thus \textbf{quantum\_factorial()} can return negative values,  the term we take the square root of might thus be negative. Hence we need to use the function \textbf{cm.sqrt()}, i.e the complex square root from the \textit{cmath} package. \\
\begin{lstlisting}[language=Python, caption={$\Delta$ coefficients}]
def q_big_delta_coeff(i,j,k,r,s):
argsqrt = (quantum_factorial(r,s,int(i + j - k)) * quantum_factorial(r,s,int(i + k - j)) * quantum_factorial(r,s,int(j + k - i))) / quantum_factorial(r,s,int(i + j + k + 1))

res = cm.sqrt(argsqrt)
return res
\end{lstlisting}
The function \textbf{q\_symbol(i, j, k, l, m, n, r, s)} returns the quantum $6j$-symbol
$\begin{vmatrix} 
i & j & k \\ 
l & m & n
\end{vmatrix}$
given $r$, $s$ and $(i,j,k,l,m,n)$ a sextuple of elements of $I_r$ such that $(i,j,k)$, $(j,l,n)$, $(i,m,n)$ and $(k,l,m)$ are admissible.

\begin{lstlisting}[language=Python, caption={Quantum \textit{$6j$}-symbols}]
def q_symbol(i, j, k, l, m, n, r, s):
prefac = q_big_delta_coeff(i, j, k, r, s) * q_big_delta_coeff(j, l, n, r, s) * q_big_delta_coeff(i, m, n, r, s) * q_big_delta_coeff(k, l, m, r, s)

zmin = max(i + j + k, j + l + n, i + m + n, k + l + m)
zmax = min(i + j + l + m, i + k + l + n, j + k + m + n)

sumres = 0
for z in range(int(zmin), int(zmax) + 1):
den = quantum_factorial(r,s,int(z - (i + j + k))) * quantum_factorial(r,s,int(z - (j + l + n))) * quantum_factorial(r,s,int(z - (i + m + n))) * quantum_factorial(r,s,int(z - (k + l + m))) * quantum_factorial(r,s,int(i + j + l + m - z)) * quantum_factorial(r,s,int(i + k + l + n - z)) * quantum_factorial(r,s,int(j + k + m + n - z))

sumres = sumres + (((-1) ** z) * quantum_factorial(r,s,int(z + 1))) / den

res = prefac * sumres * sqrt(-1)**(int(2*(i+j+k+l+m+n)))
return res
\end{lstlisting}
The function \textbf{edge(a,r,s)} yields the 
term $w_a:=(\-1)^{2a}[2a+1]$ 
for $r$, $s$ and $a \in I_r$.

\begin{lstlisting}[language=Python, caption={Edge term}]
def edge(a,r,s) :
res = ((-1)**(2*a))*quantum_number(r,s,int(2*a+1))
return res
\end{lstlisting}
The function \textbf{term(a,b,c,g,r,s)} computes one term of the sum in the Definition \ref{TuraevVirosum} applied to $(M_g,\mathcal{T}_g)$  given $r$, $s$ as before, $g \in \{2,3,4,... \}$, $a \in I_r$, $b \in I_r$ and $c=(c_0,c_1,c_2,...,c_{g}) \in I_r^{g+1}$  satisfying the conditions in Theorem \ref{thm:allowedstates}. This one term is
equal to the product of $w_a w_b w_{c_0} \ldots w_{c_g}$ with
$$ \scriptsize
\begin{vmatrix} 
a & b & b \\ 
c_0 & c_0 & c_g
\end{vmatrix} 
\begin{vmatrix} 
a & b & b \\ 
c_0 & c_0 & c_1
\end{vmatrix} \ldots
\begin{vmatrix} 
a & b & b \\ 
c_{i} & c_{i} & c_{i-1}
\end{vmatrix}
\begin{vmatrix} 
a & b & b \\ 
c_{i} & c_{i} & c_{i+1}
\end{vmatrix} \ldots 
\begin{vmatrix} 
a & b & b \\ 
c_{g} & c_{g} & c_{g-1}
\end{vmatrix}
\begin{vmatrix} 
a & b & b \\ 
c_{g} & c_{g} & c_{0}
\end{vmatrix}.
$$

\begin{lstlisting}[language=Python, caption={One term of the sum in $TV_{r,s}(M_g,\mathcal{T}_g)$ for a given admissible coloring}]
def term(a,b,c,g,r,s):
res=edge(a, r, s)*edge(b, r, s)

for i in range(g+1):
if i==0:
res = res * edge(c[i], r, s) * q_symbol(a, b, b, c[i], c[i], c[g], r, s) * q_symbol(a, b, b, c[i], c[i], c[i+1], r, s)
elif i==g:
res = res * edge(c[i], r, s) * q_symbol(a, b, b, c[i], c[i], c[i-1], r, s) * q_symbol(a, b, b, c[i], c[i], c[0], r, s)
else:
res = res * edge(c[i], r, s) * q_symbol(a, b, b, c[i], c[i], c[i-1], r, s) * q_symbol(a, b, b, c[i], c[i], c[i+1], r, s)

return res
\end{lstlisting}

The function \textbf{turaevvirog2(r,s)} computes $TV_{r,s}(M_2,\mathcal{T}_2)$, while the function \textbf{turaevvirog3(r,s)}  computes $TV_{r,s}(M_3,\mathcal{T}_3)$ given $r$ and $s$.
These two functions consists in listing all the admissible colorings in the sense of Theorem \ref{thm:allowedstates} and computing the associated term contributing to the sum in $TV_{r,s}(M_g,\mathcal{T}_g)$ via the \textbf{term()} function for $g\in \{ 2,3\}$.

We now detail the reasoning behind the following code. Let $a,b,c_0,...,c_g \in I_r$.

\underline{The $a$ loop:}
For the coloring to be admissible, $a$ has to satisfy: \begin{itemize}
	\item [(1)]$a \in \mathbb{N}$.
\end{itemize}
The first loop is over the label $a$ which covers all integers from $0$ to $\frac{r-2}{2}$. We will refer to it as the \textit{$a$ loop}. The \textbf{floor()} function ensures that the loop ends at $\left \lfloor{\frac{r-2}{2}}\right \rfloor $.

\underline{The $b$ loop:}
For $a$ fixed and an admissible coloring, $b$ has to satisfy: \begin{itemize}
	\item[(2)] $b \in \mathbb{N}$,
	\item[(4)] $\frac{a}{2}\leq b \leq \frac{r-2-a}{2}$.
\end{itemize}
We  thus create, inside the $a$ loop, a second loop over the label $b$ (called the \textit{$b$ loop}), which covers all integers from $\frac{a}{2}$ to $\frac{r-2-a}{2}$. The \textbf{floor()} and \textbf{ceil()} functions ensure that the loop ends at $\left \lfloor{\frac{r-2-a}{2}}\right \rfloor $  and starts at $\left \lceil{\frac{a}{2}}\right \rceil $.

\underline{The two kinds of nested loops inside the $b$ loop:}
For an admissible coloring, the labels $c_i$ (for $i \in \{0,1,2,...,g\}$) have to satisfy:\begin{itemize}
	\item[(3)]either ($\forall i \in \{0,1,2,...,g\} $, $c_i\in \mathbb{N}$) or ($\forall i \in \{0,1,2,...,g\} $, $c_i \in \frac{\mathbb{N}_{odd}}{2}$).
\end{itemize} We denote those two categories as \textit{integer states} and \textit{half-integer states}.
Thus, we create two different loops inside the $b$ loop (one for each).

\underline{The first family of nested loops, for integer states:}
Let us assume ($\forall i \in \{0,\ldots,g\} $, $c_i\in \mathbb{N}$). For a fixed $b$ and an admissible coloring, the label $c_0$ has to satisfy:\begin{itemize}
	\item[(6a)] $\frac{b}{2} \leq c_{0} \leq \frac{r-2-b}{2}$.
\end{itemize} We  thus create, inside the $b$ loop, a loop over the label $c_0$, the \textit{$c_0$ loop}, which covers all integers from $\left \lceil{\frac{b}{2}}\right \rceil $ to $\left \lfloor{\frac{r-2-b}{2}}\right \rfloor $. This $c_0$ loop will contain a new $c_1$ \textit{loop}, which will in turn contain a $c_2$ \textit{loop}, and so on until a final $c_g$ \textit{loop}. Let us now detail how this induction works.
For $i \in \{0,...,g-1\}$, for fixed $a$, $b$, $c_i$ and an admissible coloring, the label $c_{i+1}$ has to satisfy the following conditions: \begin{itemize}
	\item[(5a)]$\frac{b}{2} \leq c_{i+1} \leq \frac{r-2-b}{2}$,
	\item[(5b)]$a - c_i \leq c_{i+1} \leq r-2-a-c_i$,
	\item[(5c-d)]$c_i - \min(a,b)\leq c_{i+1} \leq c_i + \min(a,b)$.
\end{itemize}
For all $i \in \{0,...,g-1\}$, we create in the loop $c_i$ two variables $m_{i+1}$ and $M_{i+1}$: $m_{i+1}$ is the maximum of the three lower bounds on $c_{i+1}$ and $M_{i+1}$ is the minimum of the three upper bounds on $c_{i+1}$. We  thus create, inside the $c_i$ loop, a loop over the label $c_{i+1}$ (called the $c_{i+1}$ \textit{loop}), which ranges to all integers from $\left \lceil{m_{i+1}}\right \rceil $ to $\left \lfloor{M_{i+1}}\right \rfloor $. Finally, for fixed $a$, $b$, $c_0, \ldots, c_g$, for the coloring to be admissible, the following conditions  also need to be satisfied:  
\begin{itemize}
	\item[(6b)]$a - c_g \leq c_{0} \leq r-2-a-c_g$,
	\item[(6c-d)]$c_g - \min(a,b)\leq c_{0} \leq c_g + \min(a,b)$.
\end{itemize} Hence we add an \textbf{if} loop that checks those two conditions. If they are satisfied, we then compute the term associated to the coloring $(a, b, c_0, c_1, ..., c_g)$ via \textbf{term()}. 

\underline{The second family of nested loops, for half-integer states:}
Once we have covered all the integer states, we create inside the $b$ loop, a second loop over the label $c_0$, which covers all half-integers from $\left \lfloor{\frac{b}{2}}\right \rfloor + \frac{1}{2}$ to $\left \lceil{\frac{r-2-b}{2}}\right \rceil - \frac{1}{2}$. The function \textbf{np.arange()} allows us to start and end loops at non-integer values and specify the step of the loop to be 1. 
The rest of this new $c_0$ loop is constructed in the same way as for the $c_0$ loop for integer states, but the nested loops are over half-integers instead of integers.

\underline{Conclusion:}
As proved in Theorem \ref{thm:allowedstates}, these loops go over all admissible states, thus we obtain a numerical value of the exact formula $TV_{r,s}(M_g,\mathcal{T}_g)$ for $g\in \{ 2,3\}$. In the same way, one can construct a function which computes $TV_{r,s}(M_g,\mathcal{T}_g)$ for $g \geq 4$ (one would simply need to add more nested loops in the two families).
In the current project, we did exactly this: we defined the functions \textbf{turaevvirog4(r,s)}, \textbf{turaevvirog5(r,s)}, \textbf{turaevvirog6(r,s)} and \textbf{turaevvirog7(r,s)}, but for the sake of brevity we do not write their codes here (their codes can be easily guessed from the detailed examples of \textbf{turaevvirog2(r,s)} and \textbf{turaevvirog3(r,s)}).

\begin{lstlisting}[language=Python, caption={$TV_{r,s}(M_2,\mathcal{T}_2)$}]
def turaevvirog2(r,s):
res= 0
for a in range(floor((r-2)/2)+1):
for b in range(ceil(a/2),floor((r-2-a)/2)+1):
m=min(a,b)
for c0 in range(ceil(b/2),floor((r-2-b)/2)+1):
m1=max(c0-m,a-c0,b/2)
M1=min(c0+m,r-2-a-c0,(r-2-b)/2)
for c1 in range(ceil(m1),floor(M1)+1):
m2=max(c1-m,a-c1,b/2)
M2=min(c1+m,r-2-a-c1,(r-2-b)/2)
for c2 in range(ceil(m2),floor(M2)+1):
if (-m<=c2-c0<=m) and (a<=c2+c0<=r-2-a):
c=[c0,c1,c2]
res=res+term(a,b,c,2,r,s)

for c0 in np.arange(floor(b/2)+1/2,ceil((r-2-b)/2),1):
m1=max(c0-m,a-c0,b/2)
M1=min(c0+m,r-2-a-c0,(r-2-b)/2)
for c1 in np.arange(floor(m1)+1/2,ceil(M1),1):
m2=max(c1-m,a-c1,b/2)
M2=min(c1+m,r-2-a-c1,(r-2-b)/2)
for c2 in np.arange(floor(m2)+1/2,ceil(M2),1):
if (-m<=c2-c0<=m) and (a<=c2+c0<=r-2-a):
c=[c0,c1,c2]
res=res+term(a,b,c,2,r,s)
return res       
\end{lstlisting}
\begin{lstlisting}[language=Python, caption={$TV_{r,s}(M_3,\mathcal{T}_3)$}]
def turaevvirog3(r,s):
res= 0
for a in range(floor((r-2)/2)+1):
for b in np.arange(ceil(a/2),floor((r-2-a)/2)+1,1):
m=min(a,b)
for c0 in np.arange(ceil(b/2),floor((r-2-b)/2)+1,1):
m1=max(c0-m,a-c0,b/2)
M1=min(c0+m,r-2-a-c0,(r-2-b)/2)
for c1 in np.arange(ceil(m1),floor(M1)+1,1):
m2=max(c1-m,a-c1,b/2)
M2=min(c1+m,r-2-a-c1,(r-2-b)/2)
for c2 in np.arange(ceil(m2),floor(M2)+1,1):
m3=max(c2-m,a-c2,b/2)
M3=min(c2+m,r-2-a-c2,(r-2-b)/2)
for c3 in np.arange(ceil(m3),floor(M3)+1,1):
if (-m<=c3-c0<=m) and (a<=c3+c0<=r-2-a):
c=[c0,c1,c2,c3]
res=res+term(a,b,c,3,r,s)

for c0 in np.arange(floor(b/2)+1/2,ceil((r-2-b)/2),1):
m1=max(c0-m,a-c0,b/2)
M1=min(c0+m,r-2-a-c0,(r-2-b)/2)
for c1 in np.arange(floor(m1)+1/2,ceil(M1),1):
m2=max(c1-m,a-c1,b/2)
M2=min(c1+m,r-2-a-c1,(r-2-b)/2)
for c2 in np.arange(floor(m2)+1/2,ceil(M2),1):
m3=max(c2-m,a-c2,b/2)
M3=min(c2+m,r-2-a-c2,(r-2-b)/2)
for c3 in np.arange(floor(m3)+1/2,ceil(M3),1):
if (-m<=c3-c0<=m) and (a<=c3+c0<=r-2-a):
c=[c0,c1,c2,c3]
res=res+term(a,b,c,3,r,s)
return res       
\end{lstlisting}
\subsection{Asymptotic behavior of $QV_{r,2}(M_2)$ and $QV_{r,2}(M_3)$}\label{sec:code:asymp}
The function \textbf{quantumvirog2(r,s)} computes $QV_{r,s}(M_2)=\frac{s\pi}{r-2} \log\left(TV_{r,s}(M_2,\mathcal{T}_2) \right)$ given $(r,s)\in \mathbb{N}^2$ such that $r\geq3$ and $s\geq1$.
\begin{lstlisting}[language=Python, caption={$QV_{r,s}(M_2,\mathcal{T}_2)$}]
def quantumvirog2(r,s):
res=(s*(np.pi)/(r-2))*log(turaevvirog2(r,s))
return res
\end{lstlisting}
The function \textbf{quantumvirog3(r,s)} computes $QV_{r,s}(M_3)=\frac{s\pi}{r-2} \log\left(TV_{r,s}(M_3,\mathcal{T}_3) \right)$ $(r,s)\in \mathbb{N}^2$ such that $r\geq3$ and $s\geq1$.
\begin{lstlisting}[language=Python, caption={$QV_{r,s}(M_3,\mathcal{T}_3)$}]
def quantumvirog3(r,s):
res=(s*(np.pi)/(r-2))*log(turaevvirog3(r,s))
return res
\end{lstlisting}
Similarly, the definitions of \textbf{quantumvirog4(r,s)}, \textbf{quantumvirog5(r,s)}, \textbf{quantumvirog6(r,s)} and \textbf{quantumvirog7(r,s)} follow immediately.

Let us recall that we aim for a numerical test of Conjecture \ref{conjecture}. Thus, we set $s=2$ in order to numerically compute $QV_{r,2}(M_2)$ and $QV_{r,2}(M_3)$.

Note that, for some $r$ it happens that $TV_{r,2}(M_2,\mathcal{T}_2) $ is negative. Since we require the argument of the complex logarithm function to be in $[0,2\pi[$, the imaginary part of $QV_{r,2}(M_2)$ is either 0 when $TV_{r,2}(M_2,\mathcal{T}_2) $ is positive or $\frac{2\pi^2}{r}$ when $TV_{r,2}(M_2,\mathcal{T}_2) $ is negative, which converges to $0$ as $r \to \infty$. Therefore to test the convergence of $QV_{r,2}(M_2)$ we can forget about imaginary parts and consider only the real parts.

We compute the function \textbf{quantumvirog2(r,s)} for $s=2$ and increasing values of $r$, and we obtain the table of values of $\mathcal{R}\left(QV_{r,2}(M_2)\right)$ shown in Figure \ref{fig:table:qviro2-7}. We use the numerical approximation \textbf{.n()} with the best available precision (around \textbf{prec = 180}). 

\begin{lstlisting}[language=Python, caption={Numerical approximation of $\mathcal{R}\left(QV_{r,2}(M_2)\right)$}]
quantumvirog2(r,2).real().n(prec=180)
\end{lstlisting}

Similarly, we use the functions \textbf{quantumvirog3(r,s)}, \ldots, \textbf{quantumvirog7(r,s)} for $s=2$ and increasing values of $r$, and we display the  values of $\mathcal{R}\left(QV_{r,2}(M_g)\right)$   in Figure \ref{fig:table:qviro2-7}.

As we will detail in Section \ref{sec:res}, we obtain surprising values for high $r$ and $g=2, 3$, probably due to numerical errors. Those values are written in red in Figure \ref{fig:table:qviro2-7}.

\begin{figure}[!h]
	\centering
	\begin{tabular}{ccc}
		\begin{tabular}{|c|c|c|}
			\hline 
			$r$ & $\mathcal{R}\left(QV_{r,2}(M_2)\right)$  \\ 
			\hline 
			5 &   8.14385123663626   \\ 
			\hline  7 &   9.18650442759997    \\ 
			\hline  9 &   9.65004427173429   \\ 
			\hline   11&  9.96879239401443    \\  
			\hline   13& 10.20513879726808    \\ 
			\hline 15  & 10.38914324592799  \\ 
			\hline   17& 10.53704472005768   \\ 
			\hline  19 & 10.65879117905018    \\  
			\hline 21 &  10.76091340012164    \\ 
			\hline  23 & 10.84790597624064\\ 
			\hline 25  & 10.92297357052110  \\  
			\hline 27  & 10.98846715752597  \\  
			\hline   29& 11.04614827534519    \\ 
			\hline   31& 11.09819658700029 \\  
			\hline  33 & 11.10744853337351    \\ 
			\hline 35  & \textcolor{red}{10.85076510281595} \\ 
			\hline   37& \textcolor{red}{11.70823932238226}   \\ 
			\hline 39  & \textcolor{red}{12.05034471052339}  \\ 
			\hline  41 & \textcolor{red}{12.57984278565481}  \\ 
			\hline  43 & \textcolor{red}{13.01497045469742}   \\ 
			\hline  45 & \textcolor{red}{13.57883304172589}  \\ 
			\hline  47 & \textcolor{red}{13.99851452347661}   \\ 
			\hline 
		\end{tabular}  & \hspace*{-0.25cm}
		\begin{tabular}{|c|c|c|}
			\hline 
			$r$ & $\mathcal{R}\left(QV_{r,2}(M_3)\right)$  \\ 
			\hline 5 &  11.49177317419101  \\ 
			\hline 7 &  12.80934693191113  \\ 
			\hline 9 &  13.58615197340893  \\ 
			\hline 11 & 14.12955507845825  \\ 
			\hline 13 & 14.53997951590672 \\ 
			\hline 15 & 14.86388896169300  \\ 
			\hline 17 & 15.12724763049115  \\ 
			\hline 19 & 15.34618602238218 \\ 
			\hline 21 & 15.53141775410042   \\ 
			\hline 23 & 15.69039789582600   \\ 
			\hline 25 & 15.82849506550996   \\ 
			\hline 27 & 15.94972272572273  \\ 
			\hline 29 & 16.05847664488577   \\ 
			\hline 31 & 16.12064941438458   \\ 
			\hline 33 & \textcolor{red}{16.64108419344305}   \\ 
			\hline 35 & \textcolor{red}{17.23677472848113}   \\ 
			\hline 37 & \textcolor{red}{17.65793100469928}   \\ 
			\hline 39 & \textcolor{red}{18.19438875927008}  \\ 
			\hline
			\addlinespace[11.5ex]
		\end{tabular}  & \hspace*{-0.25cm}
		\begin{tabular}{|c|c|c|}
			\hline 
			$r$ & $\mathcal{R}\left(QV_{r,2}(M_4)\right)$  \\ 
			\hline 	5 & 
			14.51784517894469  \\ \hline 7 & 16.30280237431099  \\ \hline 9 & 17.32714285662395  \\ \hline 11 & 18.05414567452926  \\ \hline 13 & 18.60945703261760  \\ \hline 15 & 19.05151621992931  \\ \hline 17 & 19.41350816169271  \\ \hline 19 & 19.71628402919349  \\ \hline 21 & 19.97380655712918  \\ \hline 23 & 20.19586182173212  \\ \hline 25 & 20.38962564202214  \\ \hline 27 & 20.54717170623221  \\
			\hline
			\addlinespace[28.7ex]
		\end{tabular} \\
		\addlinespace[2ex]
		\begin{tabular}{|c|c|c|}
			\hline 
			$r$ & $\mathcal{R}\left(QV_{r,2}(M_5)\right)$  \\ 
			\hline 
			5 & 
			17.56864290428003  \\ \hline 7 & 19.74442367439225  \\ \hline 9 & 20.99442151342528  \\ \hline 11 & 21.88836919170208  \\ \hline 13 & 22.57622952582667  \\ \hline 15 & 23.12700521166837  \\ \hline 17 & 23.58015181610567  \\ \hline 19 & 23.96067740594393  \\ \hline 21 & 24.28544874705841  \\ \hline 23 & 24.56622464869820 \\
			\hline
		\end{tabular} 
		&\hspace*{-0.25cm}
		\begin{tabular}{|c|c|c|}
			\hline 
			$r$ & $\mathcal{R}\left(QV_{r,2}(M_6)\right)$  \\ 
			\hline 
			5 & 
			20.59635740610918  \\ \hline 7 & 23.16334886690935  \\ \hline 9 & 24.62826235095652  \\ \hline 11 & 25.68044858255137  \\ \hline 13 & 26.49408736663125  \\ \hline 15 & 27.14829604792329  \\ \hline 17 & 27.68837084809290  \\ \hline 19 & 28.14316996246829  \\ \hline 21 & 28.53221301857429  \\ \hline 23 & 28.85466729936771 \\
			\hline
		\end{tabular}
		&\hspace*{-0.25cm}
		\begin{tabular}{|c|c|c|}
			\hline 
			$r$ & $\mathcal{R}\left(QV_{r,2}(M_7)\right)$  \\ 
			\hline 
			5 & 
			23.62294303366446  \\ \hline 7 & 26.57176683519978  \\ \hline 9 & 28.24541308192440  \\ \hline 11 & 29.45065948405797  \\ \hline 13 & 30.38589828885670  \\ \hline 15 & 31.14019388548824  \\ \hline 17 & 31.76448809338449  \\ \hline 19 & 32.29128792277911  \\
			\hline
			\addlinespace[5.9ex]
		\end{tabular}
	\end{tabular}
	\caption{Values of $\mathcal{R}\left(QV_{r,2}(M_g)\right)$ for $2 \leqslant g \leqslant 7$ and $r \geqslant 5$ (assumed numerical errors are written in \textcolor{red}{red}).}\label{fig:table:qviro2-7}
\end{figure}

%
%

In the following code, we look for the best interpolation of our data for the values $\mathcal{R}\left(QV_{r,2}(M_2)\right)$, when $5 \leqslant r \leqslant 33$. We look for a model of the form 
$a+ b\frac{\ln(r-2)}{r-2} + c\frac{1}{r-2}$ (as in Conjecture \ref{conj:vol:bc}). We use the function \textbf{find\_fit}.

\begin{lstlisting}[language=Python, caption={3-term interpolation of $\mathcal{R}\left(QV_{r,2}(M_2)\right)$ for $5 \leqslant r \leqslant 33$.}]
data=[(2*i+1,quantumvirog2(2*i+1,2).real().n(prec=180)) for i in range(2,15)]

var('a, b, c, r')
model(r)= a+ b*2*pi*ln(r-2)/(r-2) + c/(r-2)

sol = find_fit(data,model)
show(sol)
\end{lstlisting}

We find the values
$$a = 11.86209740389381, b = -0.835561949347834, c = -5.310168450722084,$$
and in particular a constant term $a$ equal to the expected hyperbolic volume up to
$$ \dfrac{12.046092040094381-11.86209740389381}{12.046092040094381} \approx 1.5 \%.$$

We then do the same for $M_3$, in the following code.

\begin{lstlisting}[language=Python, caption={3-term interpolation of $\mathcal{R}\left(QV_{r,2}(M_3)\right)$ for $5 \leqslant r \leqslant 31$.}]
data=[(2*i+1,quantumvirog3(2*i+1,2).real().n(prec=180)) for i in range(2,14)]


var('a, b, c, r')
model(r)= a+ b*2*pi*ln(r-2)/(r-2) + c/(r-2)

sol = find_fit(data,model)
show(sol)
\end{lstlisting}

We find the values
$$a = 17.712568980467715, b = -1.95506206171866, c = -5.092760978446523,$$
and in particular a constant term $a$ equal to the expected hyperbolic volume up to
$$ \dfrac{18.03810545488482-17.712568980467715}{18.03810545488482} \approx 1.8 \%.$$

For each $g \in \{4, \ldots 7\}$, we interpolate all available values of $\mathcal{R}\left(QV_{r,2}(M_g)\right)$ with the same model (since no numerical strangeness occur in these cases). The values for $a,b,c$ are listed in Figure \ref{fig:table:abc}.

\begin{figure}[!h]
	\centering
 \begin{tabular}{|c|c|c|c|c|c|c|}
 	\hline 
$g$ & $r_{max}$ & $\Vol(M_g)$ &  $a$ & $b$ & $c$ & $\frac{\Vol(M_g)-a}{\Vol(M_g)}$\\
 	\hline 
2 & 33 & 
12.04609204 & 
11.86209740 & 
-0.83556194 & 
-5.31016845 & $\approx 1.5 \%$ \\
\hline 
3 & 31 & 
18.03810545 & 
17.71256898 &
-1.95506206 &
-5.09276097
& $\approx 1.8 \%$ \\
\hline 
4 & 27 & 
23.60349490 & 
22.91592390 &
-2.65679563 &
-6.74587906
& $\approx 2.9 \%$ \\
\hline 
5 & 23 & 
28.98945539 & 
27.83557719 &
-3.23491649 &
-8.35921398
& $\approx 3.9 \%$ \\
\hline 
6 & 23 & 
34.28064479 & 
32.73892860 &
-3.85245863 &
-9.69525194
& $\approx 4.5 \%$ \\
\hline 
7 & 19 & 
39.51512785 & 
37.25645299 &
-4.15342419 &
-12.1205935
& $\approx 5.7 \%$ \\
\hline 
 \end{tabular}
	\caption{Values of the interpolating coefficients $a,b,c$ for the model $a+ b \cdot \frac{2\pi \ln(r-2)}{r-2} + c\frac{1}{r-2}$ for $\mathcal{R}\left(QV_{r,2}(M_g)\right)$, with $5\leqslant r \leqslant r_{max}$.}\label{fig:table:abc}
\end{figure}

\section{Numerical Results}\label{sec:res}

\subsection{The case of $M_2$}\label{sec:M2}

Let us now state the relevant structures and invariants of $M_2$, which will then allow us to numerically check the volume conjecture for this manifold.

\subsubsection{Triangulation}

Figure \ref{fig:T2} displays the ideal triangulation $\mathcal{T}_g$ of $M_g$ in the case $g=2$ (note that some gluing information is not directly stated in the picture for clarity). The $0$-skeleton $(\mathcal{T}_2)^{0,\sim}$ has two elements $\nu_1$ (corresponding to the toroidal boundary component) and $\nu_2$ (corresponding to the boundary component of genus $2$). The $1$-skeleton $(\mathcal{T}_2)^{1,\sim}$ contains five classes $\eta_1, \ldots , \eta_5$.

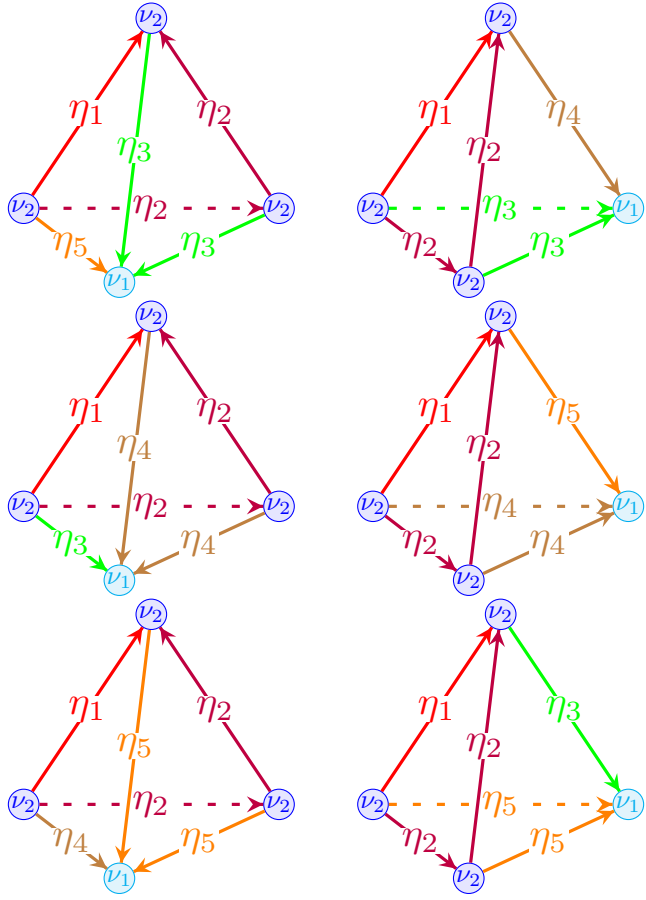
\begin{figure}[!h]
	\centering
	\begin{tikzpicture}[scale=0.7,inner sep=0.05mm, centered]
	\tikzstyle{front}=[very thick,decoration={markings,mark=at position 1 with
		{\arrow[scale=1.2,>=stealth]{>}}},postaction={decorate}]
	\tikzstyle{back}=[loosely dashed,very thick, decoration={markings,mark=at position 1 with
		{\arrow[scale=1.2,>=stealth]{>}}},postaction={decorate}]

	\node[two-hole] (a) at (-2.4,0) {$\nu_2$};
	\node[two-hole] (b) at (0,3.6) {$\nu_2$};
	\node[tore] (c) at (-0.6,-1.4) {$\nu_1$};
	\node[two-hole] (d) at (2.4,0) {$\nu_2$};
	
	\draw[front,red] (a) -- (b) node[labels] { $\color{red} \eta_1$};
	\draw[front,orange] (a) -- (c) node[labels] { $\color{orange} \eta_5$};
	\draw[back,purple] (a) -- (d) node[labels] { $\color{purple} \eta_2$};
	\draw[front,green] (b) -- (c) node[labels] { $\color{green} \eta_3$};
	\draw[front,purple] (d) -- (b) node[labels] { $\color{purple} \eta_2$};
	\draw[front,green] (d) -- (c) node[labels] { $\color{green} \eta_3$};
	
	\end{tikzpicture}
	\qquad
	\begin{tikzpicture}[scale=0.7,inner sep=0.05mm, centered]
	\tikzstyle{front}=[very thick,decoration={markings,mark=at position 1 with
		{\arrow[scale=1.2,>=stealth]{>}}},postaction={decorate}]
	\tikzstyle{back}=[loosely dashed,very thick, decoration={markings,mark=at position 1 with
		{\arrow[scale=1.2,>=stealth]{>}}},postaction={decorate}]
	
	\node[two-hole] (a) at (-2.4,0) {$\nu_2$};
	\node[two-hole] (b) at (0,3.6) {$\nu_2$};
	\node[two-hole] (c) at (-0.6,-1.4) {$\nu_2$};
	\node[tore] (d) at (2.4,0) {$\nu_1$};
	
	\draw[front,red] (a) -- (b) node[labels] { $\color{red} \eta_1$};
	\draw[front,purple] (a) -- (c) node[labels] { $\color{purple} \eta_2$};
	\draw[back,green] (a) -- (d) node[labels] { $\color{green} \eta_3$};
	\draw[front,purple] (c) -- (b) node[labels] { $\color{purple} \eta_2$};
	\draw[front,brown] (b) -- (d) node[labels] { $\color{brown} \eta_4$};
	\draw[front,green] (c) -- (d) node[labels] { $\color{green} \eta_3$};
	
	\end{tikzpicture}
	
	\begin{tikzpicture}[scale=0.7,inner sep=0.05mm, centered]
	\tikzstyle{front}=[very thick,decoration={markings,mark=at position 1 with
		{\arrow[scale=1.2,>=stealth]{>}}},postaction={decorate}]
	\tikzstyle{back}=[loosely dashed,very thick, decoration={markings,mark=at position 1 with
		{\arrow[scale=1.2,>=stealth]{>}}},postaction={decorate}]
	
	\node[two-hole] (a) at (-2.4,0) {$\nu_2$};
	\node[two-hole] (b) at (0,3.6) {$\nu_2$};
	\node[tore] (c) at (-0.6,-1.4) {$\nu_1$};
	\node[two-hole] (d) at (2.4,0) {$\nu_2$};
	
	\draw[front,red] (a) -- (b) node[labels] { $\color{red} \eta_1$};
	\draw[front,green] (a) -- (c) node[labels] { $\color{green} \eta_3$};
	\draw[back,purple] (a) -- (d) node[labels] { $\color{purple} \eta_2$};
	\draw[front,brown] (b) -- (c) node[labels] { $\color{brown} \eta_4$};
	\draw[front,purple] (d) -- (b) node[labels] { $\color{purple} \eta_2$};
	\draw[front,brown] (d) -- (c) node[labels] { $\color{brown} \eta_4$};
	
	\end{tikzpicture}
	\qquad
	\begin{tikzpicture}[scale=0.7,inner sep=0.05mm, centered]
	\tikzstyle{front}=[very thick,decoration={markings,mark=at position 1 with
		{\arrow[scale=1.2,>=stealth]{>}}},postaction={decorate}]
	\tikzstyle{back}=[loosely dashed,very thick, decoration={markings,mark=at position 1 with
		{\arrow[scale=1.2,>=stealth]{>}}},postaction={decorate}]
	
	\node[two-hole] (a) at (-2.4,0) {$\nu_2$};
	\node[two-hole] (b) at (0,3.6) {$\nu_2$};
	\node[two-hole] (c) at (-0.6,-1.4) {$\nu_2$};
	\node[tore] (d) at (2.4,0) {$\nu_1$};
	
	\draw[front,red] (a) -- (b) node[labels] { $\color{red} \eta_1$};
	\draw[front,purple] (a) -- (c) node[labels] { $\color{purple} \eta_2$};
	\draw[back,brown] (a) -- (d) node[labels] { $\color{brown} \eta_4$};
	\draw[front,purple] (c) -- (b) node[labels] { $\color{purple} \eta_2$};
	\draw[front,orange] (b) -- (d) node[labels] { $\color{orange} \eta_5$};
	\draw[front,brown] (c) -- (d) node[labels] { $\color{brown} \eta_4$};
	
	\end{tikzpicture}
	
	\begin{tikzpicture}[scale=0.7,inner sep=0.05mm, centered]
	\tikzstyle{front}=[very thick,decoration={markings,mark=at position 1 with
		{\arrow[scale=1.2,>=stealth]{>}}},postaction={decorate}]
	\tikzstyle{back}=[loosely dashed,very thick, decoration={markings,mark=at position 1 with
		{\arrow[scale=1.2,>=stealth]{>}}},postaction={decorate}]

	\node[two-hole] (a) at (-2.4,0) {$\nu_2$};
	\node[two-hole] (b) at (0,3.6) {$\nu_2$};
	\node[tore] (c) at (-0.6,-1.4) {$\nu_1$};
	\node[two-hole] (d) at (2.4,0) {$\nu_2$};
	
	\draw[front,red] (a) -- (b) node[labels] { $\color{red} \eta_1$};
	\draw[front,brown] (a) -- (c) node[labels] { $\color{brown} \eta_4$};
	\draw[back,purple] (a) -- (d) node[labels] { $\color{purple} \eta_2$};
	\draw[front,orange] (b) -- (c) node[labels] { $\color{orange} \eta_5$};
	\draw[front,purple] (d) -- (b) node[labels] { $\color{purple} \eta_2$};
	\draw[front,orange] (d) -- (c) node[labels] { $\color{orange} \eta_5$};
	
	\end{tikzpicture}
	\qquad
	\begin{tikzpicture}[scale=0.7,inner sep=0.05mm, centered]
	\tikzstyle{front}=[very thick,decoration={markings,mark=at position 1 with
		{\arrow[scale=1.2,>=stealth]{>}}},postaction={decorate}]
	\tikzstyle{back}=[loosely dashed,very thick, decoration={markings,mark=at position 1 with
		{\arrow[scale=1.2,>=stealth]{>}}},postaction={decorate}]
	
	\node[two-hole] (a) at (-2.4,0) {$\nu_2$};
	\node[two-hole] (b) at (0,3.6) {$\nu_2$};
	\node[two-hole] (c) at (-0.6,-1.4) {$\nu_2$};
	\node[tore] (d) at (2.4,0) {$\nu_1$};
	
	\draw[front,red] (a) -- (b) node[labels] { $\color{red} \eta_1$};
	\draw[front,purple] (a) -- (c) node[labels] { $\color{purple} \eta_2$};
	\draw[back,orange] (a) -- (d) node[labels] { $\color{orange} \eta_5$};
	\draw[front,purple] (c) -- (b) node[labels] { $\color{purple} \eta_2$};
	\draw[front,green] (b) -- (d) node[labels] { $\color{green} \eta_3$};
	\draw[front,orange] (c) -- (d) node[labels] { $\color{orange} \eta_5$};
	
	\end{tikzpicture}

	\caption{The ordered ideal triangulation $\mathcal{T}_2$ of the 3-manifold $M_2$} \label{fig:T2}
\end{figure}

\subsubsection{Hyperbolic structure}

As a specific case of Section \ref{sec:Mg:hyp}, 
the unique complete hyperbolic structure on the manifold $M_2$ is given by the angles
$$\alpha_2=\frac{\pi}{6}, \ \ \beta_2=2 \alpha_2=\frac{\pi}{3}, \ \ \gamma_2= \arccos((2\cos\alpha_2)^{-1}),  \ \ \delta_2 = \pi - 2 \gamma_2,$$
and the hyperbolic volume of $M_2$ is computed (via the code of  Section \ref{sec:codehyperbolicvolumeg}) to be $\Vol(M_2)=12.046092040094381...$ (which corresponds to the value computed with \textit{Orb} in \cite{orb}).

\subsubsection{Admissible colorings and Turaev--Viro Invariants}

From Definition \ref{TuraevVirosum}, 
we compute the edge terms and tetrahedron terms 
contributing to $TV_{r,s}(M_2,\mathcal{T}_2)$. Since $\mathcal{T}_2$ has no regular vertices,  the regular vertices term is thus $N=\left(\sum_{i\in I_r}w_i^2\right)^{0}=1$.

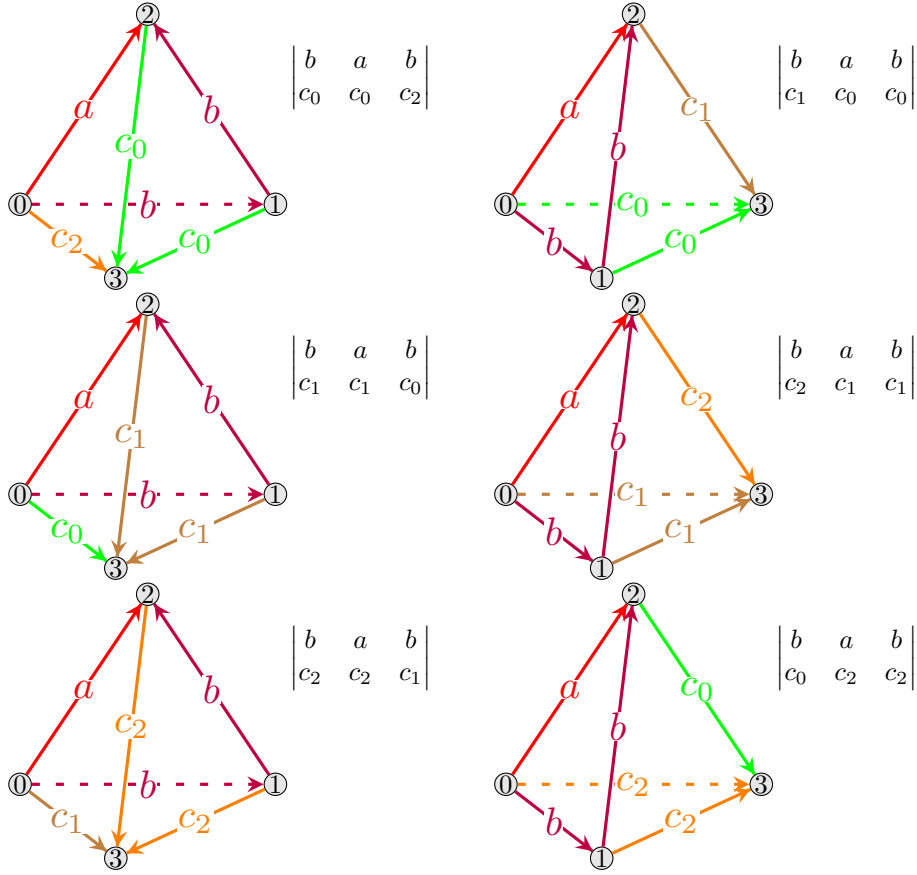
\begin{figure}[!h]
	\centering
	\begin{tikzpicture}[scale=0.7,inner sep=0.05mm, centered]

	\node[rod] (a) at (-2.4,0) {0};
	\node[rod] (b) at (0,3.6) {2};
	\node[rod] (c) at (-0.6,-1.4) {3};
	\node[rod] (d) at (2.4,0) {1};
	
	\draw[frontbis,red] (a) -- (b) node[labels] { $\color{red} a$};
	\draw[frontbis,orange] (a) -- (c) node[labels] { $\color{orange} c_2$};
	\draw[backbis,purple] (a) -- (d) node[labels] { $\color{purple} b$};
	\draw[frontbis,green] (b) -- (c) node[labels] { $\color{green} c_0$};
	\draw[frontbis,purple] (d) -- (b) node[labels] { $\color{purple} b$};
	\draw[frontbis,green] (d) -- (c) node[labels] { $\color{green} c_0$};
	
	\node (ab) at (4,2.4) {$\begin{vmatrix} b & a & b \\ c_0 & c_0 & c_2 \end{vmatrix} $};
	\end{tikzpicture}
	\qquad
	\begin{tikzpicture}[scale=0.7,inner sep=0.05mm, centered]
	\node[rod] (a) at (-2.4,0) {0};
	\node[rod] (b) at (0,3.6) {2};
	\node[rod] (c) at (-0.6,-1.4) {1};
	\node[rod] (d) at (2.4,0) {3};

	\draw[frontbis,red] (a) -- (b) node[labels] { $\color{red} a$};
	\draw[frontbis,purple] (a) -- (c) node[labels] { $\color{purple} b$};
	\draw[backbis,green] (a) -- (d) node[labels] { $\color{green} c_0$};
	\draw[frontbis,purple] (c) -- (b) node[labels] { $\color{purple} b$};
	\draw[frontbis,brown] (b) -- (d) node[labels] { $\color{brown} c_1$};
	\draw[frontbis,green] (c) -- (d) node[labels] { $\color{green} c_0$};
	
	\node (ab) at (4,2.4) {$\begin{vmatrix} b & a & b \\ c_1 & c_0 & c_0 \end{vmatrix} $};
	\end{tikzpicture}
	
	\begin{tikzpicture}[scale=0.7,inner sep=0.05mm, centered]

	\node[rod] (a) at (-2.4,0) {0};
	\node[rod] (b) at (0,3.6) {2};
	\node[rod] (c) at (-0.6,-1.4) {3};
	\node[rod] (d) at (2.4,0) {1};
	
	\draw[frontbis,red] (a) -- (b) node[labels] { $\color{red} a$};
	\draw[frontbis,green] (a) -- (c) node[labels] { $\color{green} c_0$};
	\draw[backbis,purple] (a) -- (d) node[labels] { $\color{purple} b$};
	\draw[frontbis,brown] (b) -- (c) node[labels] { $\color{brown} c_1$};
	\draw[frontbis,purple] (d) -- (b) node[labels] { $\color{purple} b$};
	\draw[frontbis,brown] (d) -- (c) node[labels] { $\color{brown} c_1$};
	
	\node (ab) at (4,2.4) {$\begin{vmatrix} b & a & b \\ c_1 & c_1 & c_0 \end{vmatrix}$};
	\end{tikzpicture}
	\qquad
	\begin{tikzpicture}[scale=0.7,inner sep=0.05mm, centered]
	\node[rod] (a) at (-2.4,0) {0};
	\node[rod] (b) at (0,3.6) {2};
	\node[rod] (c) at (-0.6,-1.4) {1};
	\node[rod] (d) at (2.4,0) {3};
	
	\draw[frontbis,red] (a) -- (b) node[labels] { $\color{red} a$};
	\draw[frontbis,purple] (a) -- (c) node[labels] { $\color{purple} b$};
	\draw[backbis,brown] (a) -- (d) node[labels] { $\color{brown} c_1$};
	\draw[frontbis,purple] (c) -- (b) node[labels] { $\color{purple} b$};
	\draw[frontbis,orange] (b) -- (d) node[labels] { $\color{orange} c_2$};
	\draw[frontbis,brown] (c) -- (d) node[labels] { $\color{brown} c_1$};
	
	\node (ab) at (4,2.4) {$\begin{vmatrix} b & a & b \\ c_2 & c_1 & c_1 \end{vmatrix} $};
	\end{tikzpicture}
	
	\begin{tikzpicture}[scale=0.7,inner sep=0.05mm, centered]

	\node[rod] (a) at (-2.4,0) {0};
	\node[rod] (b) at (0,3.6) {2};
	\node[rod] (c) at (-0.6,-1.4) {3};
	\node[rod] (d) at (2.4,0) {1};
	
	\draw[frontbis,red] (a) -- (b) node[labels] { $\color{red} a$};
	\draw[frontbis,brown] (a) -- (c) node[labels] { $\color{brown} c_1$};
	\draw[backbis,purple] (a) -- (d) node[labels] { $\color{purple} b$};
	\draw[frontbis,orange] (b) -- (c) node[labels] { $\color{orange} c_2$};
	\draw[frontbis,purple] (d) -- (b) node[labels] { $\color{purple} b$};
	\draw[frontbis,orange] (d) -- (c) node[labels] { $\color{orange} c_2$};
	
	\node (ab) at (4,2.4) {$\begin{vmatrix} b & a & b \\ c_2 & c_2 & c_1 \end{vmatrix} $};
	
	\end{tikzpicture}
	\qquad
	\begin{tikzpicture}[scale=0.7,inner sep=0.05mm, centered]
	\node[rod] (a) at (-2.4,0) {0};
	\node[rod] (b) at (0,3.6) {2};
	\node[rod] (c) at (-0.6,-1.4) {1};
	\node[rod] (d) at (2.4,0) {3};

	\draw[frontbis,red] (a) -- (b) node[labels] { $\color{red} a$};
	\draw[frontbis,purple] (a) -- (c) node[labels] { $\color{purple} b$};
	\draw[backbis,orange] (a) -- (d) node[labels] { $\color{orange} c_2$};
	\draw[frontbis,purple] (c) -- (b) node[labels] { $\color{purple} b$};
	\draw[frontbis,green] (b) -- (d) node[labels] { $\color{green} c_0$};
	\draw[frontbis,orange] (c) -- (d) node[labels] { $\color{orange} c_2$};
	
	\node (ab) at (4,2.4) {$\begin{vmatrix} b & a & b \\ c_0 & c_2 & c_2 \end{vmatrix} $};
	\end{tikzpicture}

	\caption{The coloring $c$ of the ideal tetrahedra of $\mathcal{T}_2$ together with their respective tetrahedron terms}\label{fig:coloring2}
\end{figure}

Let $(a,b,c_0,c_1,c_2)$ be a quintuple of elements of $I_r$. A coloring $c:X^1_{\sim} \to I_r$ such as in Figure \ref{fig:coloring2} is admissible if and only if it satisfies the conditions of Theorem \ref{thm:allowedstates}, thus

$$\mathcal{A}_r(M_2,\mathcal{T}_2)=\left\{ 
\begin{pmatrix}
a\\
b\\
c_0\\
c_1\\
c_2
\end{pmatrix}
\in I_r^5
\;
\begin{tabular}{|l}
$a,b \in \mathbb{N}$, \\
$\frac{a}{2} \leq b \leq \frac{r-2-a}{2}$, \\
either $c_0,c_1,c_2 \in \mathbb{N}$ or $c_0,c_1,c_2 \in \frac{\mathbb{N}_{odd}}{2}$, \\
$\max\left(\frac{b}{2},a - c_3,c_3- \min (a,b)\right) \leq c_{0}$, \\
$c_{0} \leq \min\left(\frac{r-2-b}{2},r - 2 - a - c_3, \min (a,b) + c_3\right)$,\\

$\max\left(\frac{b}{2},a - c_0,c_0- \min (a,b)\right) \leq c_{1}$, \\
$c_{1} \leq \min\left(\frac{r-2-b}{2},r - 2 - a - c_0, \min (a,b) + c_0\right)$,\\
$\max\left(\frac{b}{2},a - c_1,c_1- \min (a,b)\right) \leq c_{2}$, \\
$c_{2} \leq \min\left(\frac{r-2-b}{2},r - 2 - a - c_1, \min (a,b) + c_1\right)$.
\end{tabular}
\right\}.$$

From Figure \ref{fig:coloring2},
we can determine the six tetrahedron terms $| T |_c$ and five edge terms $|\eta|_c$ associated to the coloring $c$ which gives us the following equation:
\begin{align*}TV_{r,s}(M_2,\mathcal{T}_2)=\sum_{(a,b,c_0,c_1,c_2)\in \mathcal{A}_r(M_2,\mathcal{T}_2)}w_aw_bw_{c_0}w_{c_1}w_{c_2} \cdot &
\begin{vmatrix} 
b & a & b \\ 
c_0 & c_0 & c_2 
\end{vmatrix}
\begin{vmatrix} 
b & a & b \\ 
c_1 & c_0 & c_0 
\end{vmatrix} \\
\cdot &
\begin{vmatrix} 
b & a & b \\ 
c_1 & c_1 & c_0 
\end{vmatrix}
\begin{vmatrix} 
b & a & b \\ 
c_2 & c_1 & c_1 
\end{vmatrix}\\
\cdot & \begin{vmatrix} 
b & a & b \\ 
c_2 & c_2 & c_1 
\end{vmatrix} \begin{vmatrix} 
b & a & b \\ 
c_0 & c_2 & c_2 
\end{vmatrix}
.
\end{align*}

Using Proposition \ref{prop:allowedperm} and a bit of reordering, we can rewrite the summation as:
\begin{align*}TV_{r,s}(M_2,\mathcal{T}_2)=\sum_{(a,b,c_0,c_1,c_2)\in \mathcal{A}_r(M_2,\mathcal{T}_2)}w_aw_bw_{c_0}w_{c_1}w_{c_2} \cdot &
\begin{vmatrix} 
a & b & b \\ 
c_0 & c_0 & c_2 
\end{vmatrix}
\begin{vmatrix} 
a & b & b \\ 
c_0 & c_0 & c_1 
\end{vmatrix}\\
\cdot & \begin{vmatrix} 
a & b & b \\ 
c_1 & c_1 & c_0 
\end{vmatrix} \begin{vmatrix} 
a & b & b \\ 
c_1 & c_1 & c_2 
\end{vmatrix}\\
\cdot &
\begin{vmatrix} 
a & b & b \\ 
c_2 & c_2 & c_1 
\end{vmatrix}
\begin{vmatrix} 
a & b & b \\ 
c_2 & c_2 & c_0 
\end{vmatrix}
.
\end{align*}\label{TuraevVirog2}

\subsubsection{Numerical check of the volume conjecture}\label{sec:num:disc:M2}

\begin{figure}[!h]
	\centering
	\begin{tikzpicture}[scale=1.4]
	\begin{axis}[axis x line=bottom,axis y line = left, 
	ymin=7,ymax=15,xmin=0,
				xlabel={$r$}, 	xlabel style={at={(ticklabel* cs:1)},anchor=south west},
	legend style={
		at={(0,0)},
		anchor=north west,at={(axis description cs:0.1,-0.1)}}]
		\addplot[only marks, color=cyan, mark=*] coordinates { 
(5, 8.1438512366362676431208456051535904407501220703125000)
(7, 9.1865044275999743206284620100632309913635253906250000)
(9, 9.6500442717342984622064250288531184196472167968750000)
(11, 9.9687923940144393952778045786544680595397949218750000)
(13, 10.205138797268084260849718702957034111022949218750000)
(15, 10.389143245927991543453572376165539026260375976562500)
(17, 10.537044720057689772829689900390803813934326171875000)
(19, 10.658791179050181696652543905656784772872924804687500)
(21, 10.760913400121646432694433315191417932510375976562500)
(23, 10.847905976240641123808927659410983324050903320312500)
(25, 10.922973570521106623232299170922487974166870117187500)
(27, 10.988467157525978024068535887636244297027587890625000)
(29, 11.046148275345199607500035199336707592010498046875000)
(31, 11.098196587000291657432171632535755634307861328125000)
(33, 11.107448533373510457522570504806935787200927734375000)
(35, 10.850765102815952189985182485543191432952880859375000)
(37, 11.708239322382269165245816111564636230468750000000000)
(39, 12.050344710523390290290990378707647323608398437500000)
(41, 12.579842785654816950113854545634239912033081054687500)
(43, 13.014970454697428081658472365234047174453735351562500)
(45, 13.578833041725896890739022637717425823211669921875000)
(47, 13.998514523476618620634326362051069736480712890625000)
};
	\addlegendentry{$\mathcal{R}\left(QV_{r,2}(M_2)\right)$}
	\addplot[mark=none, color=blue, dashed, thick, samples=2,domain=0:50] {12.046092040094381};
	\addlegendentry{Vol$(M_2)=12.046092...$}
	\addplot[mark=none, color=blue]	expression[domain=1:50]
	{ 	11.86209740389381 - 5.249990563380644 * ln(x-2)/(x-2) - 5.310168450722084 /(x-2) };
	\addlegendentry{$11.86 - 5.25 \frac{\ln(r-2)}{r-2} - 5.31 \frac{1}{r-2}$ }
	\end{axis}
	\end{tikzpicture}
	\caption{Graph of the values of $\mathcal{R}\left(QV_{r,2}(M_2)\right)$ (for $5 \leqslant r \leqslant 47$, blue dots) compared with the hyperbolic volume $\Vol(M_2)$ (dashed and blue) and an interpolation of $\mathcal{R}\left(QV_{r,2}(M_2)\right)$  for $r \leqslant 33$ (blue curve).}\label{fig:asymptotics2}
\end{figure}
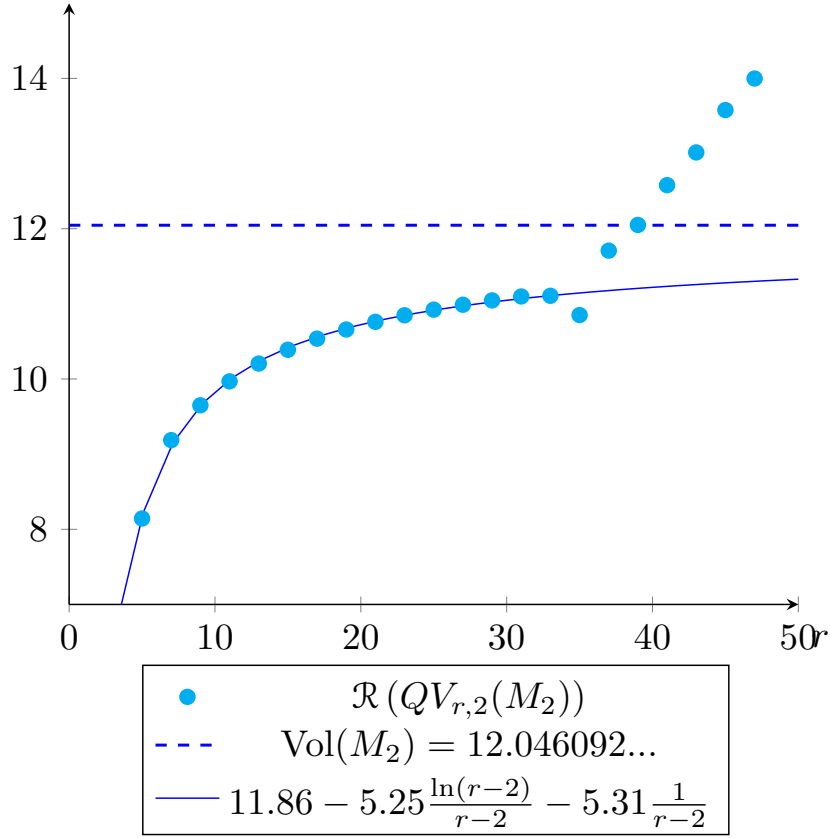

Recall that we compute and study the behavior of $ QV_{r,2}(M) := \frac{2 \pi }{r-2} \log \left( TV_{r,2} (M) \right)$.
Figure \ref{fig:asymptotics2} displays the values of $\mathcal{R}\left(QV_{r,2}(M_2)\right)$ for $5 \leqslant r \leqslant 47$ (blue dots), computed with maximal available precision via the code of Section \ref{sec:codetvr}. 

For $r \leqslant 33$, we observe the expected convergence to the hyperbolic volume $\Vol(M_2)$ (displayed with the blue dashed line); more precisely, when we fit the data for $5\leqslant r \leqslant 33$ with the model 
$a +b \frac{\ln(r-2)}{r-2} +c \frac{1}{r-2}$ (see Section \ref{sec:code:asymp}), we find a constant term which is very close to $\Vol(M_2)$ (up to $1.5 \%$), and the interpolating function found by SageMath (displayed in a full blue curve) appears to be  close to the blue dots.

However, a strange behavior starts at $r=35$, and the values of $\mathcal{R}\left(QV_{r,2}(M_2)\right)$  break the expected pattern. We offer a possible explanation that is still compatible with  Conjecture \ref{conj:vol}: since the Turaev--Viro invariants are computed as sums of a large set of terms (exponentially many in $r$) which may have vastly different orders of magnitude, the numerical approximations of the computer might truncate away subtle parts of the terms in the sum, which translates to larger and larger errors in the final results.

To test this theory, we computed the same values with smaller precision, on the same machine, expecting to find different values for high $r$. However, it was not the case. Nevertheless, we observed that a \textit{different computer} yielded slightly different numerical values than the one stated in this paper, the difference increasing as $r$ grew larger. 

Such numerical errors for large $r$ seem unavoidable when computing Turaev--Viro invariants, especially for triangulations with a large number of tetrahedra. One can for instance compare the maximum values of $r$ that were studied for triangulations with less than $4$ tetrahedra in \cite{chen2018volume}, and the ones in this paper, where $M_2$ has $6$ tetrahedra, $M_3$ has $8$ tetrahedra, and so on.

If we restrict to the values for $r \leqslant 33$, Conjectures \ref{conjecture} and \ref{conj:vol:bc} seem satisfied for $M_2$, as we observe the expected asymptotic behavior with small deviations.

\subsection{The case of $M_3$}\label{sec:M3}

\subsubsection{Triangulation}

Figure \ref{fig:T3} displays the ideal triangulation $\mathcal{T}_g$ of $M_g$ in the case $g=3$. The $0$-skeleton $(\mathcal{T}_3)^{0,\sim}$ has two elements $\nu_1$ (corresponding to the toroidal boundary component) and $\nu_2$ (corresponding to the boundary component of genus $3$). The $1$-skeleton $(\mathcal{T}_3)^{1,\sim}$ contains six classes $\eta_1, \ldots , \eta_6$.

\begin{figure}[!h]
	\centering
	\begin{tikzpicture}[scale=0.7,inner sep=0.05mm, centered]
	\tikzstyle{front}=[very thick,decoration={markings,mark=at position 1 with
		{\arrow[scale=1.2,>=stealth]{>}}},postaction={decorate}]
	\tikzstyle{back}=[loosely dashed,very thick, decoration={markings,mark=at position 1 with
		{\arrow[scale=1.2,>=stealth]{>}}},postaction={decorate}]

	\node[two-hole] (a) at (-2.4,0) {$\nu_2$};
	\node[two-hole] (b) at (0,3.6) {$\nu_2$};
	\node[tore] (c) at (-0.6,-1.4) {$\nu_1$};
	\node[two-hole] (d) at (2.4,0) {$\nu_2$};
	
	\draw[front,red] (a) -- (b) node[labels] { $\color{red} \eta_1$};
	\draw[front,blue] (a) -- (c) node[labels] { $\color{blue} \eta_6$};
	\draw[back,purple] (a) -- (d) node[labels] { $\color{purple} \eta_2$};
	\draw[front,green] (b) -- (c) node[labels] { $\color{green} \eta_3$};
	\draw[front,purple] (d) -- (b) node[labels] { $\color{purple} \eta_2$};
	\draw[front,green] (d) -- (c) node[labels] { $\color{green} \eta_3$};
	
	\end{tikzpicture}
	\qquad
	\begin{tikzpicture}[scale=0.7,inner sep=0.05mm, centered]
	\tikzstyle{front}=[very thick,decoration={markings,mark=at position 1 with
		{\arrow[scale=1.2,>=stealth]{>}}},postaction={decorate}]
	\tikzstyle{back}=[loosely dashed,very thick, decoration={markings,mark=at position 1 with
		{\arrow[scale=1.2,>=stealth]{>}}},postaction={decorate}]
	
	\node[two-hole] (a) at (-2.4,0) {$\nu_2$};
	\node[two-hole] (b) at (0,3.6) {$\nu_2$};
	\node[two-hole] (c) at (-0.6,-1.4) {$\nu_2$};
	\node[tore] (d) at (2.4,0) {$\nu_1$};
	
	\draw[front,red] (a) -- (b) node[labels] { $\color{red} \eta_1$};
	\draw[front,purple] (a) -- (c) node[labels] { $\color{purple} \eta_2$};
	\draw[back,green] (a) -- (d) node[labels] { $\color{green} \eta_3$};
	\draw[front,purple] (c) -- (b) node[labels] { $\color{purple} \eta_2$};
	\draw[front,brown] (b) -- (d) node[labels] { $\color{brown} \eta_4$};
	\draw[front,green] (c) -- (d) node[labels] { $\color{green} \eta_3$};
	
	\end{tikzpicture}
	
	\begin{tikzpicture}[scale=0.7,inner sep=0.05mm, centered]
	\tikzstyle{front}=[very thick,decoration={markings,mark=at position 1 with
		{\arrow[scale=1.2,>=stealth]{>}}},postaction={decorate}]
	\tikzstyle{back}=[loosely dashed,very thick, decoration={markings,mark=at position 1 with
		{\arrow[scale=1.2,>=stealth]{>}}},postaction={decorate}]
	
	\node[two-hole] (a) at (-2.4,0) {$\nu_2$};
	\node[two-hole] (b) at (0,3.6) {$\nu_2$};
	\node[tore] (c) at (-0.6,-1.4) {$\nu_1$};
	\node[two-hole] (d) at (2.4,0) {$\nu_2$};
	
	\draw[front,red] (a) -- (b) node[labels] { $\color{red} \eta_1$};
	\draw[front,green] (a) -- (c) node[labels] { $\color{green} \eta_3$};
	\draw[back,purple] (a) -- (d) node[labels] { $\color{purple} \eta_2$};
	\draw[front,brown] (b) -- (c) node[labels] { $\color{brown} \eta_4$};
	\draw[front,purple] (d) -- (b) node[labels] { $\color{purple} \eta_2$};
	\draw[front,brown] (d) -- (c) node[labels] { $\color{brown} \eta_4$};
	
	\end{tikzpicture}
	\qquad
	\begin{tikzpicture}[scale=0.7,inner sep=0.05mm, centered]
	\tikzstyle{front}=[very thick,decoration={markings,mark=at position 1 with
		{\arrow[scale=1.2,>=stealth]{>}}},postaction={decorate}]
	\tikzstyle{back}=[loosely dashed,very thick, decoration={markings,mark=at position 1 with
		{\arrow[scale=1.2,>=stealth]{>}}},postaction={decorate}]
	
	\node[two-hole] (a) at (-2.4,0) {$\nu_2$};
	\node[two-hole] (b) at (0,3.6) {$\nu_2$};
	\node[two-hole] (c) at (-0.6,-1.4) {$\nu_2$};
	\node[tore] (d) at (2.4,0) {$\nu_1$};
	
	\draw[front,red] (a) -- (b) node[labels] { $\color{red} \eta_1$};
	\draw[front,purple] (a) -- (c) node[labels] { $\color{purple} \eta_2$};
	\draw[back,brown] (a) -- (d) node[labels] { $\color{brown} \eta_4$};
	\draw[front,purple] (c) -- (b) node[labels] { $\color{purple} \eta_2$};
	\draw[front,orange] (b) -- (d) node[labels] { $\color{orange} \eta_5$};
	\draw[front,brown] (c) -- (d) node[labels] { $\color{brown} \eta_4$};
	
	\end{tikzpicture}
	
	\begin{tikzpicture}[scale=0.7,inner sep=0.05mm, centered]
	\tikzstyle{front}=[very thick,decoration={markings,mark=at position 1 with
		{\arrow[scale=1.2,>=stealth]{>}}},postaction={decorate}]
	\tikzstyle{back}=[loosely dashed,very thick, decoration={markings,mark=at position 1 with
		{\arrow[scale=1.2,>=stealth]{>}}},postaction={decorate}]

	\node[two-hole] (a) at (-2.4,0) {$\nu_2$};
	\node[two-hole] (b) at (0,3.6) {$\nu_2$};
	\node[tore] (c) at (-0.6,-1.4) {$\nu_1$};
	\node[two-hole] (d) at (2.4,0) {$\nu_2$};
	
	\draw[front,red] (a) -- (b) node[labels] { $\color{red} \eta_1$};
	\draw[front,brown] (a) -- (c) node[labels] { $\color{brown} \eta_4$};
	\draw[back,purple] (a) -- (d) node[labels] { $\color{purple} \eta_2$};
	\draw[front,orange] (b) -- (c) node[labels] { $\color{orange} \eta_5$};
	\draw[front,purple] (d) -- (b) node[labels] { $\color{purple} \eta_2$};
	\draw[front,orange] (d) -- (c) node[labels] { $\color{orange} \eta_5$};
	
	\end{tikzpicture}
	\qquad
	\begin{tikzpicture}[scale=0.7,inner sep=0.05mm, centered]
	\tikzstyle{front}=[very thick,decoration={markings,mark=at position 1 with
		{\arrow[scale=1.2,>=stealth]{>}}},postaction={decorate}]
	\tikzstyle{back}=[loosely dashed,very thick, decoration={markings,mark=at position 1 with
		{\arrow[scale=1.2,>=stealth]{>}}},postaction={decorate}]
	
	\node[two-hole] (a) at (-2.4,0) {$\nu_2$};
	\node[two-hole] (b) at (0,3.6) {$\nu_2$};
	\node[two-hole] (c) at (-0.6,-1.4) {$\nu_2$};
	\node[tore] (d) at (2.4,0) {$\nu_1$};
	
	\draw[front,red] (a) -- (b) node[labels] { $\color{red} \eta_1$};
	\draw[front,purple] (a) -- (c) node[labels] { $\color{purple} \eta_2$};
	\draw[back,orange] (a) -- (d) node[labels] { $\color{orange} \eta_5$};
	\draw[front,purple] (c) -- (b) node[labels] { $\color{purple} \eta_2$};
	\draw[front,blue] (b) -- (d) node[labels] { $\color{blue} \eta_6$};
	\draw[front,orange] (c) -- (d) node[labels] { $\color{orange} \eta_5$};
	
	\end{tikzpicture}
	
	\begin{tikzpicture}[scale=0.7,inner sep=0.05mm, centered]
	\tikzstyle{front}=[very thick,decoration={markings,mark=at position 1 with
		{\arrow[scale=1.2,>=stealth]{>}}},postaction={decorate}]
	\tikzstyle{back}=[loosely dashed,very thick, decoration={markings,mark=at position 1 with
		{\arrow[scale=1.2,>=stealth]{>}}},postaction={decorate}]

	\node[two-hole] (a) at (-2.4,0) {$\nu_2$};
	\node[two-hole] (b) at (0,3.6) {$\nu_2$};
	\node[tore] (c) at (-0.6,-1.4) {$\nu_1$};
	\node[two-hole] (d) at (2.4,0) {$\nu_2$};
	
	\draw[front,red] (a) -- (b) node[labels] { $\color{red} \eta_1$};
	\draw[front,orange] (a) -- (c) node[labels] { $\color{orange} \eta_5$};
	\draw[back,purple] (a) -- (d) node[labels] { $\color{purple} \eta_2$};
	\draw[front,blue] (b) -- (c) node[labels] { $\color{blue} \eta_6$};
	\draw[front,purple] (d) -- (b) node[labels] { $\color{purple} \eta_2$};
	\draw[front,blue] (d) -- (c) node[labels] { $\color{blue} \eta_6$};
	
	\end{tikzpicture}
	\qquad
	\begin{tikzpicture}[scale=0.7,inner sep=0.05mm, centered]
	\tikzstyle{front}=[very thick,decoration={markings,mark=at position 1 with
		{\arrow[scale=1.2,>=stealth]{>}}},postaction={decorate}]
	\tikzstyle{back}=[loosely dashed,very thick, decoration={markings,mark=at position 1 with
		{\arrow[scale=1.2,>=stealth]{>}}},postaction={decorate}]
	
	\node[two-hole] (a) at (-2.4,0) {$\nu_2$};
	\node[two-hole] (b) at (0,3.6) {$\nu_2$};
	\node[two-hole] (c) at (-0.6,-1.4) {$\nu_2$};
	\node[tore] (d) at (2.4,0) {$\nu_1$};
	
	\draw[front,red] (a) -- (b) node[labels] { $\color{red} \eta_1$};
	\draw[front,purple] (a) -- (c) node[labels] { $\color{purple} \eta_2$};
	\draw[back,blue] (a) -- (d) node[labels] { $\color{blue} \eta_6$};
	\draw[front,purple] (c) -- (b) node[labels] { $\color{purple} \eta_2$};
	\draw[front,green] (b) -- (d) node[labels] { $\color{green} \eta_3$};
	\draw[front,blue] (c) -- (d) node[labels] { $\color{blue} \eta_6$};
	
	\end{tikzpicture}
	
	\caption{The ordered ideal triangulation $\mathcal{T}_3$ of the 3-manifold $M_3$} \label{fig:T3}
\end{figure}
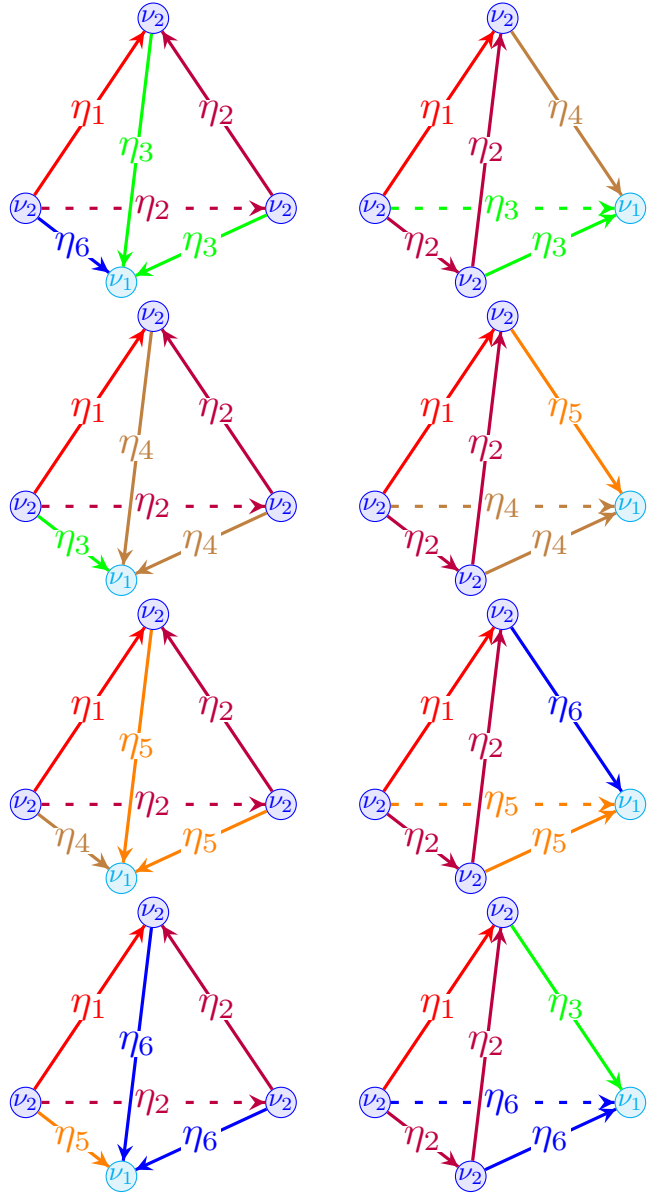

\subsubsection{Hyperbolic structure}
As a specific case of Section \ref{sec:Mg:hyp}, 
the unique complete hyperbolic structure on the manifold $M_3$ is given by the angles
$$\alpha_3=\frac{\pi}{8}, \ \ \beta_3=2 \alpha_3=\frac{\pi}{4}, \ \ \gamma_3= \arccos((2\cos\alpha_3)^{-1}),  \ \ \delta_3 = \pi - 2 \gamma_3,$$
and the hyperbolic volume of $M_3$ is computed (via the code of  Section \ref{sec:codehyperbolicvolumeg}) to be $\Vol(M_3)=18.03810545488482...$.

\subsubsection{Admissible colorings and Turaev--Viro Invariants} Using Definition \ref{TuraevVirosum} for $(M_3,\mathcal{T}_3)$, 
we can determine the edge terms and tetrahedron terms 
contributing to $TV_{r,s}(M_3,\mathcal{T}_3)$.
Since there are no regular vertices in $\mathcal{T}_3$, the regular vertices term  is thus $N=\left(\sum_{i\in I_r}w_i^2\right)^{0}=1$.

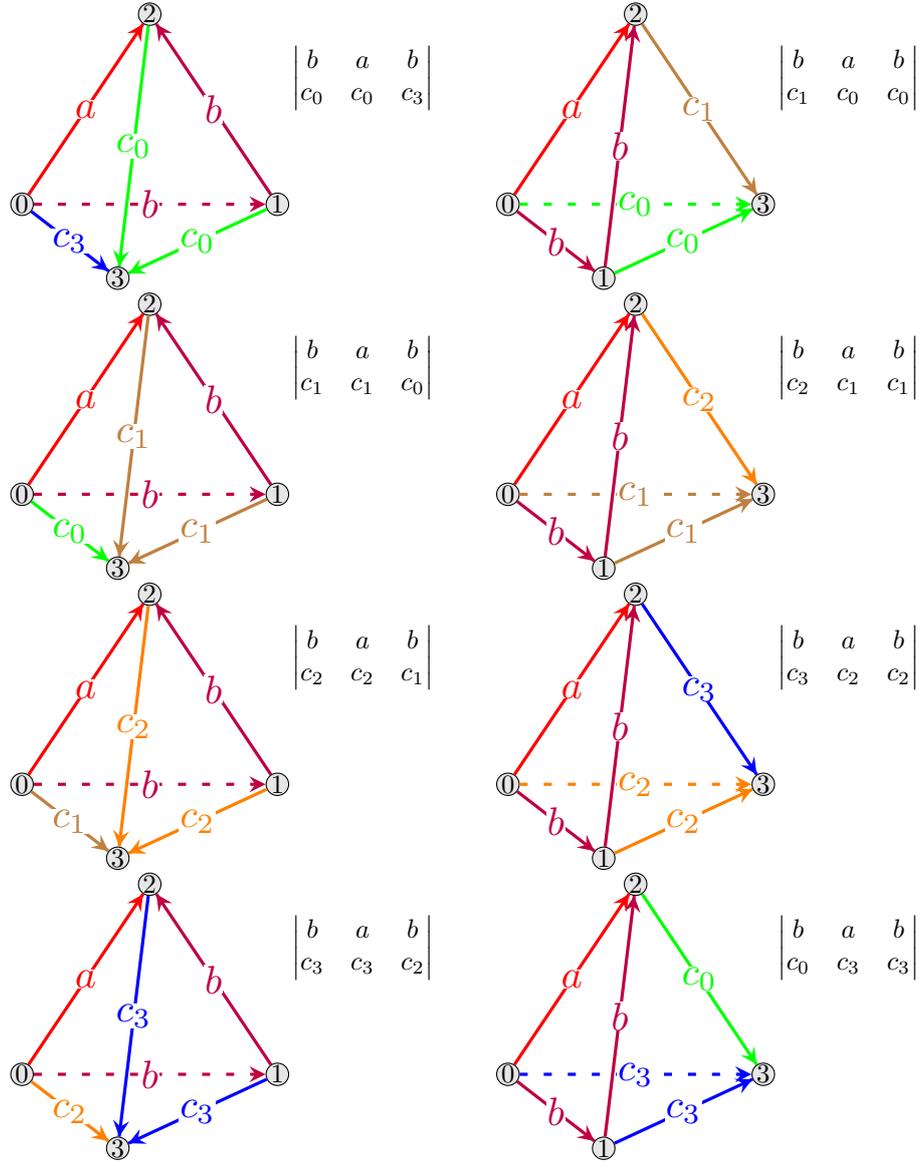
\begin{figure}[!h]
	\centering
	\begin{tikzpicture}[scale=0.7,inner sep=0.05mm, centered]

	\node[rod] (a) at (-2.4,0) {0};
	\node[rod] (b) at (0,3.6) {2};
	\node[rod] (c) at (-0.6,-1.4) {3};
	\node[rod] (d) at (2.4,0) {1};
	
	\draw[frontbis,red] (a) -- (b) node[labels] { $\color{red} a$};
	\draw[frontbis,blue] (a) -- (c) node[labels] { $\color{blue} c_3$};
	\draw[backbis,purple] (a) -- (d) node[labels] { $\color{purple} b$};
	\draw[frontbis,green] (b) -- (c) node[labels] { $\color{green} c_0$};
	\draw[frontbis,purple] (d) -- (b) node[labels] { $\color{purple} b$};
	\draw[frontbis,green] (d) -- (c) node[labels] { $\color{green} c_0$};
	
	\node (ab) at (4,2.4) {$\begin{vmatrix} b & a & b \\ c_0 & c_0 & c_3 \end{vmatrix} $};
	\end{tikzpicture}
	\qquad
	\begin{tikzpicture}[scale=0.7,inner sep=0.05mm, centered]
	\node[rod] (a) at (-2.4,0) {0};
	\node[rod] (b) at (0,3.6) {2};
	\node[rod] (c) at (-0.6,-1.4) {1};
	\node[rod] (d) at (2.4,0) {3};

	\draw[frontbis,red] (a) -- (b) node[labels] { $\color{red} a$};
	\draw[frontbis,purple] (a) -- (c) node[labels] { $\color{purple} b$};
	\draw[backbis,green] (a) -- (d) node[labels] { $\color{green} c_0$};
	\draw[frontbis,purple] (c) -- (b) node[labels] { $\color{purple} b$};
	\draw[frontbis,brown] (b) -- (d) node[labels] { $\color{brown} c_1$};
	\draw[frontbis,green] (c) -- (d) node[labels] { $\color{green} c_0$};
	
	\node (ab) at (4,2.4) {$\begin{vmatrix} b & a & b \\ c_1 & c_0 & c_0 \end{vmatrix} $};
	\end{tikzpicture}
	
	\begin{tikzpicture}[scale=0.7,inner sep=0.05mm, centered]

	\node[rod] (a) at (-2.4,0) {0};
	\node[rod] (b) at (0,3.6) {2};
	\node[rod] (c) at (-0.6,-1.4) {3};
	\node[rod] (d) at (2.4,0) {1};
	
	\draw[frontbis,red] (a) -- (b) node[labels] { $\color{red} a$};
	\draw[frontbis,green] (a) -- (c) node[labels] { $\color{green} c_0$};
	\draw[backbis,purple] (a) -- (d) node[labels] { $\color{purple} b$};
	\draw[frontbis,brown] (b) -- (c) node[labels] { $\color{brown} c_1$};
	\draw[frontbis,purple] (d) -- (b) node[labels] { $\color{purple} b$};
	\draw[frontbis,brown] (d) -- (c) node[labels] { $\color{brown} c_1$};
	
	\node (ab) at (4,2.4) {$\begin{vmatrix} b & a & b \\ c_1 & c_1 & c_0 \end{vmatrix}$};
	\end{tikzpicture}
	\qquad
	\begin{tikzpicture}[scale=0.7,inner sep=0.05mm, centered]
	\node[rod] (a) at (-2.4,0) {0};
	\node[rod] (b) at (0,3.6) {2};
	\node[rod] (c) at (-0.6,-1.4) {1};
	\node[rod] (d) at (2.4,0) {3};
	
	\draw[frontbis,red] (a) -- (b) node[labels] { $\color{red} a$};
	\draw[frontbis,purple] (a) -- (c) node[labels] { $\color{purple} b$};
	\draw[backbis,brown] (a) -- (d) node[labels] { $\color{brown} c_1$};
	\draw[frontbis,purple] (c) -- (b) node[labels] { $\color{purple} b$};
	\draw[frontbis,orange] (b) -- (d) node[labels] { $\color{orange} c_2$};
	\draw[frontbis,brown] (c) -- (d) node[labels] { $\color{brown} c_1$};
	
	\node (ab) at (4,2.4) {$\begin{vmatrix} b & a & b \\ c_2 & c_1 & c_1 \end{vmatrix} $};
	\end{tikzpicture}
	
	\begin{tikzpicture}[scale=0.7,inner sep=0.05mm, centered]

	\node[rod] (a) at (-2.4,0) {0};
	\node[rod] (b) at (0,3.6) {2};
	\node[rod] (c) at (-0.6,-1.4) {3};
	\node[rod] (d) at (2.4,0) {1};
	
	\draw[frontbis,red] (a) -- (b) node[labels] { $\color{red} a$};
	\draw[frontbis,brown] (a) -- (c) node[labels] { $\color{brown} c_1$};
	\draw[backbis,purple] (a) -- (d) node[labels] { $\color{purple} b$};
	\draw[frontbis,orange] (b) -- (c) node[labels] { $\color{orange} c_2$};
	\draw[frontbis,purple] (d) -- (b) node[labels] { $\color{purple} b$};
	\draw[frontbis,orange] (d) -- (c) node[labels] { $\color{orange} c_2$};
	
	\node (ab) at (4,2.4) {$\begin{vmatrix} b & a & b \\ c_2 & c_2 & c_1 \end{vmatrix} $};
	
	\end{tikzpicture}
	\qquad
	\begin{tikzpicture}[scale=0.7,inner sep=0.05mm, centered]
	\node[rod] (a) at (-2.4,0) {0};
	\node[rod] (b) at (0,3.6) {2};
	\node[rod] (c) at (-0.6,-1.4) {1};
	\node[rod] (d) at (2.4,0) {3};

	\draw[frontbis,red] (a) -- (b) node[labels] { $\color{red} a$};
	\draw[frontbis,purple] (a) -- (c) node[labels] { $\color{purple} b$};
	\draw[backbis,orange] (a) -- (d) node[labels] { $\color{orange} c_2$};
	\draw[frontbis,purple] (c) -- (b) node[labels] { $\color{purple} b$};
	\draw[frontbis,blue] (b) -- (d) node[labels] { $\color{blue} c_3$};
	\draw[frontbis,orange] (c) -- (d) node[labels] { $\color{orange} c_2$};
	
	\node (ab) at (4,2.4) {$\begin{vmatrix} b & a & b \\ c_3 & c_2 & c_2 \end{vmatrix} $};
	\end{tikzpicture}
	
	\begin{tikzpicture}[scale=0.7,inner sep=0.05mm, centered]

	\node[rod] (a) at (-2.4,0) {0};
	\node[rod] (b) at (0,3.6) {2};
	\node[rod] (c) at (-0.6,-1.4) {3};
	\node[rod] (d) at (2.4,0) {1};
	
	\draw[frontbis,red] (a) -- (b) node[labels] { $\color{red} a$};
	\draw[frontbis,orange] (a) -- (c) node[labels] { $\color{orange} c_2$};
	\draw[backbis,purple] (a) -- (d) node[labels] { $\color{purple} b$};
	\draw[frontbis,blue] (b) -- (c) node[labels] { $\color{blue} c_3$};
	\draw[frontbis,purple] (d) -- (b) node[labels] { $\color{purple} b$};
	\draw[frontbis,blue] (d) -- (c) node[labels] { $\color{blue} c_3$};
	
	\node (ab) at (4,2.4) {$\begin{vmatrix} b & a & b \\ c_3 & c_3 & c_2 \end{vmatrix} $};
	\end{tikzpicture}
	\qquad
	\begin{tikzpicture}[scale=0.7,inner sep=0.05mm, centered]
	\node[rod] (a) at (-2.4,0) {0};
	\node[rod] (b) at (0,3.6) {2};
	\node[rod] (c) at (-0.6,-1.4) {1};
	\node[rod] (d) at (2.4,0) {3};

	\draw[frontbis,red] (a) -- (b) node[labels] { $\color{red} a$};
	\draw[frontbis,purple] (a) -- (c) node[labels] { $\color{purple} b$};
	\draw[backbis,blue] (a) -- (d) node[labels] { $\color{blue} c_3$};
	\draw[frontbis,purple] (c) -- (b) node[labels] { $\color{purple} b$};
	\draw[frontbis,green] (b) -- (d) node[labels] { $\color{green} c_0$};
	\draw[frontbis,blue] (c) -- (d) node[labels] { $\color{blue} c_3$};
	
	\node (ab) at (4,2.4) {$\begin{vmatrix} b & a & b \\ c_0 & c_3 & c_3 \end{vmatrix} $};
	\end{tikzpicture}
	
	\caption{The coloring $c$ of the ideal tetrahedra of $\mathcal{T}_3$ together with their respective tetrahedron terms}\label{fig:coloring3}
\end{figure}

Let $(a,b,c_0,c_1,c_2,c_3)$ be a sextuple of elements of $I_r$. A coloring $c:X^1_{\sim} \to I_r$ such as in Figure \ref{fig:coloring3} is admissible if and only if it satisfies the conditions of Theorem \ref{thm:allowedstates}, thus

$$\mathcal{A}_r(M_3,\mathcal{T}_3)=\left\{ 
\begin{pmatrix}
a\\
b\\
c_0\\
c_1\\
c_2\\
c_3
\end{pmatrix}
\in I_r^5\;
\begin{tabular}{|l}
$a,b \in \mathbb{N}$, \\
$\frac{a}{2} \leq b \leq \frac{r-2-a}{2}$, \\
either $c_0,c_1,c_2 \in \mathbb{N}$ or $c_0,c_1,c_2 \in \frac{\mathbb{N}_{odd}}{2}$, \\
$\max\left(\frac{b}{2},a - c_3,c_3- \min (a,b)\right) \leq c_{0}$, \\
$c_{0} \leq \min\left(\frac{r-2-b}{2},r - 2 - a - c_3, \min (a,b) + c_3\right)$,\\

$\max\left(\frac{b}{2},a - c_0,c_0- \min (a,b)\right) \leq c_{1}$, \\
$c_{1} \leq \min\left(\frac{r-2-b}{2},r - 2 - a - c_0, \min (a,b) + c_0\right)$,\\
$\max\left(\frac{b}{2},a - c_1,c_1- \min (a,b)\right) \leq c_{2}$, \\
$c_{2} \leq \min\left(\frac{r-2-b}{2},r - 2 - a - c_1, \min (a,b) + c_1\right)$,\\
$\max\left(\frac{b}{2},a - c_2,c_2- \min (a,b)\right) \leq c_{3}$, \\
$c_{3} \leq \min\left(\frac{r-2-b}{2},r - 2 - a - c_2, \min (a,b) + c_2\right)$

\end{tabular}
\right\}.$$

From Figure \ref{fig:coloring3},
we can determine the eight tetrahedron terms $| T |_c$ and six edge terms $|\eta|_c$ associated to the coloring $c$ which gives us the following equation:
\begin{align*}TV_{r,s}(M_3,\mathcal{T}_3)=\sum_{(a,b,c_0,c_1,c_2,c_3)\in \mathcal{A}_r(M_3,\mathcal{T}_3)}
\hspace*{-1cm}w_aw_bw_{c_0}w_{c_1}w_{c_2}w_{c_3}
\cdot &
\begin{vmatrix} 
b & a & b \\ 
c_0 & c_0 & c_3 
\end{vmatrix}
\begin{vmatrix} 
b & a & b \\ 
c_1 & c_0 & c_0 
\end{vmatrix}\\
\cdot &
\begin{vmatrix} 
b & a & b \\ 
c_1 & c_1 & c_0 
\end{vmatrix}
\begin{vmatrix} 
b & a & b \\ 
c_2 & c_1 & c_1 
\end{vmatrix}\\
\cdot & \begin{vmatrix} 
b & a & b \\ 
c_2 & c_2 & c_1 
\end{vmatrix} \begin{vmatrix} 
b & a & b \\ 
c_3 & c_2 & c_2 
\end{vmatrix}\\
\cdot &
\begin{vmatrix} 
b & a & b \\ 
c_3 & c_3 & c_2 
\end{vmatrix} \begin{vmatrix} 
b & a & b \\ 
c_0 & c_3 & c_3 
\end{vmatrix}
.
\end{align*}

Using Proposition \ref{prop:allowedperm} and a bit of reordering, we can rewrite the summation as:
\begin{align*}TV_{r,s}(M_3,\mathcal{T}_3)=\sum_{(a,b,c_0,c_1,c_2,c_3)\in \mathcal{A}_r(M_3,\mathcal{T}_3)}
\hspace*{-1cm}
w_aw_bw_{c_0}w_{c_1}w_{c_2}w_{c_3} \cdot &
\begin{vmatrix} 
a & b & b \\ 
c_0 & c_0 & c_3 
\end{vmatrix}
\begin{vmatrix} 
a & b & b \\ 
c_0 & c_0 & c_1 
\end{vmatrix}\\
\cdot & \begin{vmatrix} 
a & b & b \\ 
c_1 & c_1 & c_0 
\end{vmatrix} \begin{vmatrix} 
a & b & b \\ 
c_1 & c_1 & c_2 
\end{vmatrix}\\
\cdot &
\begin{vmatrix} 
a & b & b \\ 
c_2 & c_2 & c_1 
\end{vmatrix}
\begin{vmatrix} 
a & b & b \\ 
c_2 & c_2 & c_3 
\end{vmatrix}\\
\cdot &
\begin{vmatrix} 
a & b & b \\ 
c_3 & c_3 & c_2 
\end{vmatrix} \begin{vmatrix} 
a & b & b \\ 
c_3 & c_3 & c_0 
\end{vmatrix}
.
\end{align*}\label{TuraevVirog3}

\subsubsection{Numerical check of the volume conjecture}\label{sec:num:disc:M3}

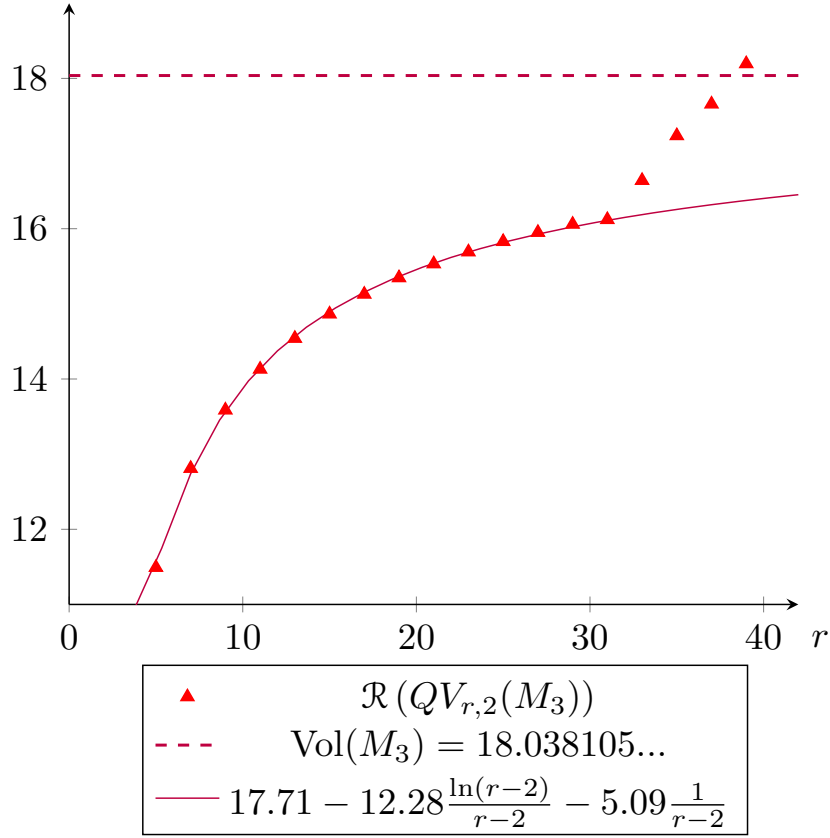
\begin{figure}[!h]
	\begin{tikzpicture}[scale=1.4]
	\begin{axis}[axis x line=bottom,axis y line = left,
	ymin=11,ymax=19,xmin=0,
				xlabel={$r$}, 	xlabel style={at={(ticklabel* cs:1)},anchor=south west},
	legend style={
		at={(0,0)},
		anchor=north west,at={(axis description cs:0.1,-0.1)}}]
	\addplot[only marks, color=red, mark=triangle*] coordinates { 
	(5, 11.491773174191012074629725248087197542190551757812500)
	(7, 12.809346931911131051151642168406397104263305664062500)
	(9, 13.586151973408931326048332266509532928466796875000000)
	(11, 14.129555078458251315964844252448529005050659179687500)
	(13, 14.539979515906729545804410008713603019714355468750000)
	(15, 14.863888961693005441588866233360022306442260742187500)
	(17, 15.127247630491153174148166726808995008468627929687500)
	(19, 15.346186022382187985613199998624622821807861328125000)
	(21, 15.531417754100420580698482808656990528106689453125000)
	(23, 15.690397895826002638841600855812430381774902343750000)
	(25, 15.828495065509965655792257166467607021331787109375000)
	(27, 15.949722725722734750775089196395128965377807617187500)
	(29, 16.058476644885779194282804382964968681335449218750000)
	(31, 16.120649414384580211390129989013075828552246093750000)
	(33, 16.641084193443052896554945618845522403717041015625000)
	(35, 17.236774728481130125601339386776089668273925781250000)
	(37, 17.657931004699285892911575501784682273864746093750000)
	(39, 18.194388759270086808328414917923510074615478515625000)	
};
\addlegendentry{$\mathcal{R}\left(QV_{r,2}(M_3)\right)$}
	\addplot[mark=none,  color=purple,  dashed, thick, samples=2,domain=0:42] {18.03810545488482};
	\addlegendentry{Vol$(M_3)=18.038105...$}
		\addplot[mark=none, color=purple]	expression[domain=2:42]
	{ 	17.712568980467715 - 12.28401722081491 * ln(x-2)/(x-2) - 5.092760978446523 /(x-2) };
	\addlegendentry{$17.71 - 12.28 \frac{\ln(r-2)}{r-2} - 5.09 \frac{1}{r-2}$ }
	\end{axis}
	\end{tikzpicture}
	\caption{Graph of the values of $\mathcal{R}\left(QV_{r,2}(M_3)\right)$ (for $5 \leqslant r \leqslant 39$, red triangles) compared with the hyperbolic volume $\Vol(M_3)$ (red dashed line) and an interpolation of $\mathcal{R}\left(QV_{r,2}(M_3)\right)$  for $r \leqslant 31$ (red curve).}\label{graph3}
\end{figure}

Our conclusions and discussions are almost the same as for the example of $M_2$ in Section \ref{sec:num:disc:M2}. The only differences are that the strange behavior starts appearing earlier,  at $r=31$ (which makes sense since there are more tetrahedra and thus more terms in the sums), and the constant term from the interpolating function is equal to the volume $\Vol(M_3)$ up to $1.8 \%$.

\subsection{The cases of $M_4$ to $M_7$}\label{sec:M4-7}

For $4 \leqslant g \leqslant 7$, we do not observe a strange behavior of $\mathcal{R}\left(QV_{r,2}(M_g)\right)$ as $r$ increases, and we can thus interpolate all available values with the model $a+ b \cdot \frac{2\pi \ln(r-2)}{r-2} + c\frac{1}{r-2}$. 
The values of $a,b,c$ are listed in Figure \ref{fig:table:abc}; all computed values of $\mathcal{R}\left(QV_{r,2}(M_g)\right)$ and the associated interpolating functions are displayed in Figure \ref{fig:graphs:2-7}.

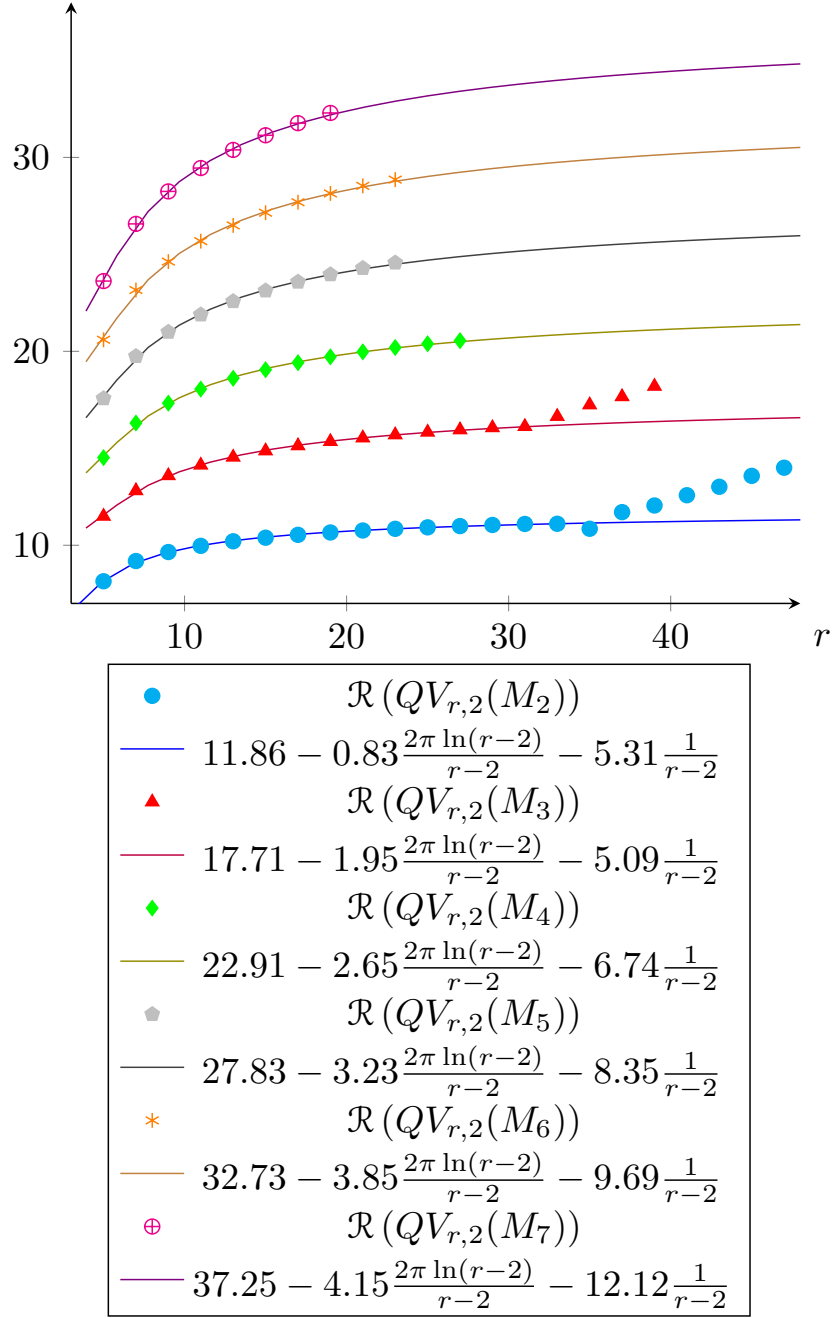
\begin{figure}[!h]
	\begin{tikzpicture}[scale=1.4]
	\begin{axis}[axis x line=bottom,axis y line = left,
	ymin=7,ymax=38,xmin=3,xmax=48,
				xlabel={$r$}, 	xlabel style={at={(ticklabel* cs:1)},anchor=south west},
	legend style={
		at={(0,0)},
		anchor=north west,at={(axis description cs:0.05,-0.1)}}]
			\addplot[only marks, color=cyan, mark=*] coordinates { 
		(5, 8.1438512366362676431208456051535904407501220703125000)
		(7, 9.1865044275999743206284620100632309913635253906250000)
		(9, 9.6500442717342984622064250288531184196472167968750000)
		(11, 9.9687923940144393952778045786544680595397949218750000)
		(13, 10.205138797268084260849718702957034111022949218750000)
		(15, 10.389143245927991543453572376165539026260375976562500)
		(17, 10.537044720057689772829689900390803813934326171875000)
		(19, 10.658791179050181696652543905656784772872924804687500)
		(21, 10.760913400121646432694433315191417932510375976562500)
		(23, 10.847905976240641123808927659410983324050903320312500)
		(25, 10.922973570521106623232299170922487974166870117187500)
		(27, 10.988467157525978024068535887636244297027587890625000)
		(29, 11.046148275345199607500035199336707592010498046875000)
		(31, 11.098196587000291657432171632535755634307861328125000)
		(33, 11.107448533373510457522570504806935787200927734375000)
		(35, 10.850765102815952189985182485543191432952880859375000)
		(37, 11.708239322382269165245816111564636230468750000000000)
		(39, 12.050344710523390290290990378707647323608398437500000)
		(41, 12.579842785654816950113854545634239912033081054687500)
		(43, 13.014970454697428081658472365234047174453735351562500)
		(45, 13.578833041725896890739022637717425823211669921875000)
		(47, 13.998514523476618620634326362051069736480712890625000)
	};
	\addlegendentry{$\mathcal{R}\left(QV_{r,2}(M_2)\right)\hspace*{-0.05cm}$}
		\addplot[mark=none, color=blue]	expression[domain=1:48]
	{ 	11.86209740389381 - 5.249990563380644 * ln(x-2)/(x-2) - 5.310168450722084 /(x-2) };
	\addlegendentry{$11.86 - 0.83 \frac{2\pi \ln(r-2)}{r-2} - 5.31 \frac{1}{r-2}$ }
	\addplot[only marks, color=red, mark=triangle*] coordinates { 
		(5, 11.491773174191012074629725248087197542190551757812500)
		(7, 12.809346931911131051151642168406397104263305664062500)
		(9, 13.586151973408931326048332266509532928466796875000000)
		(11, 14.129555078458251315964844252448529005050659179687500)
		(13, 14.539979515906729545804410008713603019714355468750000)
		(15, 14.863888961693005441588866233360022306442260742187500)
		(17, 15.127247630491153174148166726808995008468627929687500)
		(19, 15.346186022382187985613199998624622821807861328125000)
		(21, 15.531417754100420580698482808656990528106689453125000)
		(23, 15.690397895826002638841600855812430381774902343750000)
		(25, 15.828495065509965655792257166467607021331787109375000)
		(27, 15.949722725722734750775089196395128965377807617187500)
		(29, 16.058476644885779194282804382964968681335449218750000)
		(31, 16.120649414384580211390129989013075828552246093750000)
		(33, 16.641084193443052896554945618845522403717041015625000)
		(35, 17.236774728481130125601339386776089668273925781250000)
		(37, 17.657931004699285892911575501784682273864746093750000)
		(39, 18.194388759270086808328414917923510074615478515625000)	
	};
	\addlegendentry{$\mathcal{R}\left(QV_{r,2}(M_3)\right)\hspace*{-0.05cm}$}
		\addplot[mark=none, color=purple]	expression[domain=2:48]
	{ 	17.712568980467715 - 12.28401722081491 * ln(x-2)/(x-2) - 5.092760978446523 /(x-2) };
	\addlegendentry{$17.71 - 1.95 \frac{2\pi \ln(r-2)}{r-2} - 5.09 \frac{1}{r-2}$ }
		\addplot[only marks, color=green, mark=diamond*] coordinates { 
		(5, 14.517845178944693174116764566861093044281005859375000) (7, 16.302802374310992661321506602689623832702636718750000) (9, 17.327142856623950706307368818670511245727539062500000) (11, 18.054145674529269882668813806958496570587158203125000) (13, 18.609457032617608973623646306805312633514404296875000) (15, 19.051516219929318651793437311425805091857910156250000) (17, 19.413508161692710984880250180140137672424316406250000) (19, 19.716284029193495541676384164020419120788574218750000) (21, 19.973806557129183403276329045183956623077392578125000) (23, 20.195861821732123075889830943197011947631835937500000) (25, 20.389625642022146223553136223927140235900878906250000) (27, 20.547171706232212784470902988687157630920410156250000)
	};
	\addlegendentry{$\mathcal{R}\left(QV_{r,2}(M_4)\right)\hspace*{-0.05cm}$}
	\addplot[mark=none, color=olive]	expression[domain=2:48]
	{ 		22.915923906664954 - 2.6567956302608406 *2*3.141592* ln(x-2)/(x-2) - 6.745879063222808 /(x-2) };
	\addlegendentry{$22.91 - 2.65 \frac{2\pi \ln(r-2)}{r-2} - 6.74 \frac{1}{r-2}$ }
		\addplot[only marks, color=lightgray, mark=pentagon*] coordinates { 
		(5, 17.568642904280036276531973271630704402923583984375000) (7, 19.744423674392258760690310737118124961853027343750000) (9, 20.994421513425283620790651184506714344024658203125000) (11, 21.888369191702082616757252253592014312744140625000000) (13, 22.576229525826679633837557048536837100982666015625000) (15, 23.127005211668375750377890653908252716064453125000000) (17, 23.580151816105676232382393209263682365417480468750000) (19, 23.960677405943933138132706517353653907775878906250000) (21, 24.285448747058417495736648561432957649230957031250000) (23, 24.566224648698209875874454155564308166503906250000000)
	};
	\addlegendentry{$\mathcal{R}\left(QV_{r,2}(M_5)\right)\hspace*{-0.05cm}$}
\addplot[mark=none, color=darkgray]	expression[domain=2:48]
{	27.83557719553294 - 3.2349164955231555 *2*3.141592* ln(x-2)/(x-2) - 8.359213988753478 /(x-2) };
\addlegendentry{$27.83 - 3.23 \frac{2\pi \ln(r-2)}{r-2} - 8.35 \frac{1}{r-2}$ }
		\addplot[only marks, color=orange, mark=asterisk] coordinates { 
		(5, 20.596357406109184751130669610574841499328613281250000) (7, 23.163348866909352352649875683709979057312011718750000) (9, 24.628262350956529047607546090148389339447021484375000) (11, 25.680448582551370861892792163416743278503417968750000) (13, 26.494087366631251967419302673079073429107666015625000) (15, 27.148296047923299312287781503982841968536376953125000) (17, 27.688370848092908715898374794051051139831542968750000) (19, 28.143169962468299161173490574583411216735839843750000) (21, 28.532213018574292817675086553208529949188232421875000) (23, 28.854667299367719124347786419093608856201171875000000)
	};
	\addlegendentry{$\mathcal{R}\left(QV_{r,2}(M_6)\right)\hspace*{-0.05cm}$}
\addplot[mark=none, color=brown]	expression[domain=2:48]
{	32.73892860575029 - 3.8524586349135586 *2*3.141592* ln(x-2)/(x-2) - 9.695251949481 /(x-2) };
\addlegendentry{$32.73 - 3.85 \frac{2\pi \ln(r-2)}{r-2} - 9.69 \frac{1}{r-2}$ }
		\addplot[only marks, color=magenta, mark=oplus] coordinates { 
		(5, 23.622943033664466128129788557998836040496826171875000) (7, 26.571766835199785106169656501151621341705322265625000) (9, 28.245413081924404252731619635596871376037597656250000) (11, 29.450659484057972292703198036178946495056152343750000) (13, 30.385898288856704851923495880328118801116943359375000) (15, 31.140193885488240965742079424671828746795654296875000) (17, 31.764488093384493083704001037403941154479980468750000) (19, 32.291287922779112307125615188851952552795410156250000)
	};
	\addlegendentry{$\mathcal{R}\left(QV_{r,2}(M_7)\right)\hspace*{-0.05cm}$}
\addplot[mark=none, color=violet]	expression[domain=2:48]
{ 		37.25645299703772 - 4.153424193352254 *2*3.141592* ln(x-2)/(x-2) - 12.120593520842961 /(x-2) };
\addlegendentry{$37.25 - 4.15 \frac{2\pi \ln(r-2)}{r-2} - 12.12 \frac{1}{r-2}$ }
	\end{axis}
	\end{tikzpicture}
	\caption{Graphs of the values of $\mathcal{R}\left(QV_{r,2}(M_g)\right)$ in function of $r \geqslant 5$ (for $2 \leqslant g \leqslant 7$), compared with their respective best  interpolations in the model $a+ b \cdot \frac{2\pi \ln(r-2)}{r-2} + c\frac{1}{r-2}$.}
	\label{fig:graphs:2-7}
\end{figure}

We observe good fits to the model $a+ b \cdot \frac{2\pi \ln(r-2)}{r-2} + c\frac{1}{r-2}$, with $a$ equal to the expected hyperbolic volume $\Vol(M_g)$ up to a few percents (see Figure \ref{fig:table:abc}). 
We conclude that the manifolds 
$M_2, \ldots, M_7$ seem to satisfy Conjectures \ref{conjecture} and \ref{conj:vol:bc} numerically.

\subsection{Behavior of the coefficient $b(M_g)$ relative to $g$}\label{sec:num:aff}

Let us now delve into Conjecture \ref{conj:vol:b:aff}. In this section, we assume that Conjecture \ref{conj:vol:bc} holds for the \mbox{manifolds} $M_2, \ldots, M_7$ (which is suggested numerically by the results of the previous sections). We then study whether or not the coefficient $b$ grows linearly in $g$.

Since we assume that Conjecture \ref{conj:vol:bc} holds for the manifolds $M_2, \ldots, M_7$, we can now fix $a=\Vol(M_g)$ in the model
$a+ b \cdot \frac{2\pi \ln(r-2)}{r-2} + c\frac{1}{r-2}$, and look once again for the best interpolation. 

Using \textbf{find\_fit} with the new model $\Vol(M_g)+ b \cdot \frac{2\pi \ln(r-2)}{r-2} + c\frac{1}{r-2}$ yields different values for $b,c$ than in Section \ref{sec:code:asymp}. These new values are listed in Figure \ref{fig:table:bc}
and the corresponding interpolating functions are displayed in Figure \ref{fig:graphs:vol:fixed:2-7}.

\begin{figure}[!h]
	\centering
	\begin{tabular}{|c|c|c|c|c|}
		\hline 
		$g$ & $r_{max}$ & $\Vol(M_g)$ &  $b$ & $c$ \\
		\hline 
		2 & 33 & 
		12.04609204 & 
		-1.07486449 & 
		-4.06269480
		 \\
		\hline 
		3 & 31 & 
		18.03810545 & 
		-2.36670389 &
		-2.98774665
	  \\
		\hline 
		4 & 27 & 
		23.60349490 & 
		-3.47345292 &
		-2.75451472 \\
		\hline 
		5 & 23 & 
		28.98945539 & 
		-4.50837608 &
		-2.48549875\\
		\hline 
		6 & 23 & 
		34.28064479 & 
		-5.55394983 &
		-1.84727854 \\
		\hline 
		7 & 19 & 
		39.51512785 & 
		-6.43483298 &
		-2.38715613 \\
		\hline 
	\end{tabular}
	\caption{Values of the interpolating coefficients $a,b,c$ for the model $\Vol(M_g)+ b \cdot \frac{2\pi \ln(r-2)}{r-2} + c\frac{1}{r-2}$ for $\mathcal{R}\left(QV_{r,2}(M_g)\right)$, with $5\leqslant r \leqslant r_{max}$.}\label{fig:table:bc}
\end{figure}

\begin{figure}[!h]
	\begin{tikzpicture}[scale=1.4]
	\begin{axis}[axis x line=bottom,axis y line = left,
	ymin=7,ymax=38,xmin=3,xmax=48,
			xlabel={$r$}, 	xlabel style={at={(ticklabel* cs:1)},anchor=south west},
	legend style={
		at={(0,0)},
		anchor=north west,at={(axis description cs:0.05,-0.1)}}]
	\addplot[only marks, color=cyan, mark=*] coordinates { 
		(5, 8.1438512366362676431208456051535904407501220703125000)
		(7, 9.1865044275999743206284620100632309913635253906250000)
		(9, 9.6500442717342984622064250288531184196472167968750000)
		(11, 9.9687923940144393952778045786544680595397949218750000)
		(13, 10.205138797268084260849718702957034111022949218750000)
		(15, 10.389143245927991543453572376165539026260375976562500)
		(17, 10.537044720057689772829689900390803813934326171875000)
		(19, 10.658791179050181696652543905656784772872924804687500)
		(21, 10.760913400121646432694433315191417932510375976562500)
		(23, 10.847905976240641123808927659410983324050903320312500)
		(25, 10.922973570521106623232299170922487974166870117187500)
		(27, 10.988467157525978024068535887636244297027587890625000)
		(29, 11.046148275345199607500035199336707592010498046875000)
		(31, 11.098196587000291657432171632535755634307861328125000)
		(33, 11.107448533373510457522570504806935787200927734375000)
		(35, 10.850765102815952189985182485543191432952880859375000)
		(37, 11.708239322382269165245816111564636230468750000000000)
		(39, 12.050344710523390290290990378707647323608398437500000)
		(41, 12.579842785654816950113854545634239912033081054687500)
		(43, 13.014970454697428081658472365234047174453735351562500)
		(45, 13.578833041725896890739022637717425823211669921875000)
		(47, 13.998514523476618620634326362051069736480712890625000)
	};
	\addlegendentry{$\mathcal{R}\left(QV_{r,2}(M_2)\right)\hspace*{-0.05cm}$}
	\addplot[mark=none, color=blue]	expression[domain=4:48]
	{		 	12.046092040094381 - 1.0748644910370262*2*3.141592 * ln(x-2)/(x-2) - 4.062694803667878 /(x-2) };
	\addlegendentry{$\Vol(M_2) - 1.07 \frac{2\pi \ln(r-2)}{r-2} - 4.06 \frac{1}{r-2}$ }
	\addplot[only marks, color=red, mark=triangle*] coordinates { 
		(5, 11.491773174191012074629725248087197542190551757812500)
		(7, 12.809346931911131051151642168406397104263305664062500)
		(9, 13.586151973408931326048332266509532928466796875000000)
		(11, 14.129555078458251315964844252448529005050659179687500)
		(13, 14.539979515906729545804410008713603019714355468750000)
		(15, 14.863888961693005441588866233360022306442260742187500)
		(17, 15.127247630491153174148166726808995008468627929687500)
		(19, 15.346186022382187985613199998624622821807861328125000)
		(21, 15.531417754100420580698482808656990528106689453125000)
		(23, 15.690397895826002638841600855812430381774902343750000)
		(25, 15.828495065509965655792257166467607021331787109375000)
		(27, 15.949722725722734750775089196395128965377807617187500)
		(29, 16.058476644885779194282804382964968681335449218750000)
		(31, 16.120649414384580211390129989013075828552246093750000)
		(33, 16.641084193443052896554945618845522403717041015625000)
		(35, 17.236774728481130125601339386776089668273925781250000)
		(37, 17.657931004699285892911575501784682273864746093750000)
		(39, 18.194388759270086808328414917923510074615478515625000)	
	};
	\addlegendentry{$\mathcal{R}\left(QV_{r,2}(M_3)\right)\hspace*{-0.05cm}$}
	\addplot[mark=none, color=purple]	expression[domain=4:48]
	{ 	18.03810545488482 - 2.3667038975251122*2*3.141592 * ln(x-2)/(x-2) - 2.9877466578670298 /(x-2) };
	\addlegendentry{$\Vol(M_3) - 2.36 \frac{2\pi \ln(r-2)}{r-2} - 2.98 \frac{1}{r-2}$ }
	\addplot[only marks, color=green, mark=diamond*] coordinates { 
		(5, 14.517845178944693174116764566861093044281005859375000) (7, 16.302802374310992661321506602689623832702636718750000) (9, 17.327142856623950706307368818670511245727539062500000) (11, 18.054145674529269882668813806958496570587158203125000) (13, 18.609457032617608973623646306805312633514404296875000) (15, 19.051516219929318651793437311425805091857910156250000) (17, 19.413508161692710984880250180140137672424316406250000) (19, 19.716284029193495541676384164020419120788574218750000) (21, 19.973806557129183403276329045183956623077392578125000) (23, 20.195861821732123075889830943197011947631835937500000) (25, 20.389625642022146223553136223927140235900878906250000) (27, 20.547171706232212784470902988687157630920410156250000)
	};
	\addlegendentry{$\mathcal{R}\left(QV_{r,2}(M_4)\right)\hspace*{-0.05cm}$}
	\addplot[mark=none, color=olive]	expression[domain=4:48]
	{ 		23.603494908554772  - 3.47345292029456 *2*3.141592* ln(x-2)/(x-2) - 2.7545147254229114 /(x-2) };
	\addlegendentry{$\Vol(M_4) - 3.47 \frac{2\pi \ln(r-2)}{r-2} - 2.75 \frac{1}{r-2}$ }
	\addplot[only marks, color=lightgray, mark=pentagon*] coordinates { 
		(5, 17.568642904280036276531973271630704402923583984375000) (7, 19.744423674392258760690310737118124961853027343750000) (9, 20.994421513425283620790651184506714344024658203125000) (11, 21.888369191702082616757252253592014312744140625000000) (13, 22.576229525826679633837557048536837100982666015625000) (15, 23.127005211668375750377890653908252716064453125000000) (17, 23.580151816105676232382393209263682365417480468750000) (19, 23.960677405943933138132706517353653907775878906250000) (21, 24.285448747058417495736648561432957649230957031250000) (23, 24.566224648698209875874454155564308166503906250000000)
	};
	\addlegendentry{$\mathcal{R}\left(QV_{r,2}(M_5)\right)\hspace*{-0.05cm}$}
	\addplot[mark=none, color=darkgray]	expression[domain=0:48]
	{	28.989455390245897  - 4.508376089192973 *2*3.141592* ln(x-2)/(x-2) - 2.485498757732156 /(x-2) };
	\addlegendentry{$\Vol(M_5)- 4.51 \frac{2\pi \ln(r-2)}{r-2} - 2.48 \frac{1}{r-2}$ }
	\addplot[only marks, color=orange, mark=asterisk] coordinates { 
		(5, 20.596357406109184751130669610574841499328613281250000) (7, 23.163348866909352352649875683709979057312011718750000) (9, 24.628262350956529047607546090148389339447021484375000) (11, 25.680448582551370861892792163416743278503417968750000) (13, 26.494087366631251967419302673079073429107666015625000) (15, 27.148296047923299312287781503982841968536376953125000) (17, 27.688370848092908715898374794051051139831542968750000) (19, 28.143169962468299161173490574583411216735839843750000) (21, 28.532213018574292817675086553208529949188232421875000) (23, 28.854667299367719124347786419093608856201171875000000)
	};
	\addlegendentry{$\mathcal{R}\left(QV_{r,2}(M_6)\right)\hspace*{-0.05cm}$}
	\addplot[mark=none, color=brown]	expression[domain=4:48]
	{	34.28064479640226 - 5.5539498337681374 *2*3.141592* ln(x-2)/(x-2) - 1.8472785430067207 /(x-2) };
	\addlegendentry{$\Vol(M_6) - 5.55 \frac{2\pi \ln(r-2)}{r-2} - 1.84 \frac{1}{r-2}$ }
	\addplot[only marks, color=magenta, mark=oplus] coordinates { 
		(5, 23.622943033664466128129788557998836040496826171875000) (7, 26.571766835199785106169656501151621341705322265625000) (9, 28.245413081924404252731619635596871376037597656250000) (11, 29.450659484057972292703198036178946495056152343750000) (13, 30.385898288856704851923495880328118801116943359375000) (15, 31.140193885488240965742079424671828746795654296875000) (17, 31.764488093384493083704001037403941154479980468750000) (19, 32.291287922779112307125615188851952552795410156250000)
	};
	\addlegendentry{$\mathcal{R}\left(QV_{r,2}(M_7)\right)\hspace*{-0.05cm}$}
	\addplot[mark=none, color=violet]	expression[domain=4:48]
	{ 		39.51512785426426 - 6.434832984145784 *2*3.141592* ln(x-2)/(x-2) - 2.387156138256412 /(x-2) };
	\addlegendentry{$\Vol(M_7) - 6.43 \frac{2\pi \ln(r-2)}{r-2} - 2.38 \frac{1}{r-2}$ }
	\end{axis}
	\end{tikzpicture}
	\caption{Graphs of the values of $\mathcal{R}\left(QV_{r,2}(M_g)\right)$ in function of $r \geqslant 5$ (for $2 \leqslant g \leqslant 7$), compared with their respective best  interpolations in the model $\Vol(M_g)+ b \cdot \frac{2\pi \ln(r-2)}{r-2} + c\frac{1}{r-2}$.}
	\label{fig:graphs:vol:fixed:2-7}
\end{figure}
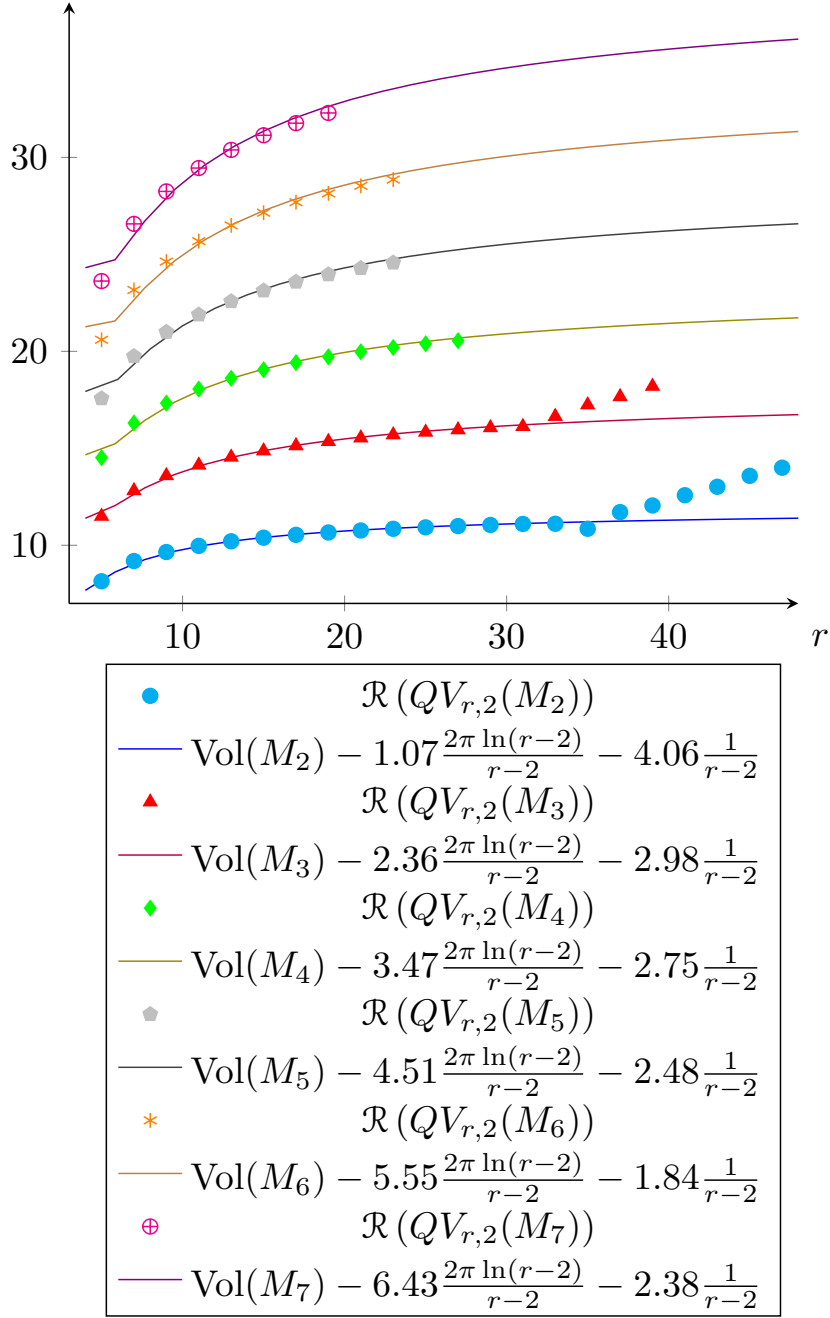

As expected by their definitions, the interpolations of Figure \ref{fig:graphs:vol:fixed:2-7} are less fitting than the ones of Figure \ref{fig:graphs:2-7}, but they still seem satisfactory.

Figure \ref{fig:aff:b} displays the values of the coefficient $b$ in function of $g$ (as black asterisks), with the corresponding best linear interpolation (the green line). The coefficient of determination of this linear interpolation is $R^2 =0.9967$, which gives much credit to the hypothesis of the affine behaviour of $b$.

\begin{figure}[!h]
	\begin{tikzpicture}[scale=1.4]
	\begin{axis}[axis x line=bottom, xlabel={$g$},  axis y line = left,
	ymin=-7,ymax=-0.5,xmin=1,xmax=7.8,
	xlabel style={at={(ticklabel* cs:1)},anchor=south west},
	legend style={
		at={(0,0)},
		anchor=north west,at={(axis description cs:0.1,-0.1)}}
	]
	\addplot[only marks, color=black, mark=star] coordinates { 
		(2, -1.0748644910370262)
		(3, -2.3667038975251122)
		(4, -3.47345292029456)
		(5, -4.508376089192973)
		(6, -5.5539498337681374)
		(7, -6.434832984145784)
	};
	\addlegendentry{$b$ coefficient for $\mathcal{R}\left(QV_{r,2}(M_g)\right)\hspace*{-0.05cm}$}
	\addplot[mark=none, color=teal]	expression[domain=0:8]
	{ 	0.9061 -1.068*x };
	\addlegendentry{$0.9061 -1.068 \ g$ }
	\end{axis}
	\end{tikzpicture}
	\caption{Values of the coefficient $b(M_g)$ for $2 \leqslant g \leqslant 7$ (black asterisks), and the associated interpolating affine function (green line), with   coefficient of determination  $R^2 = 0.9967$.}
	\label{fig:aff:b}
\end{figure}
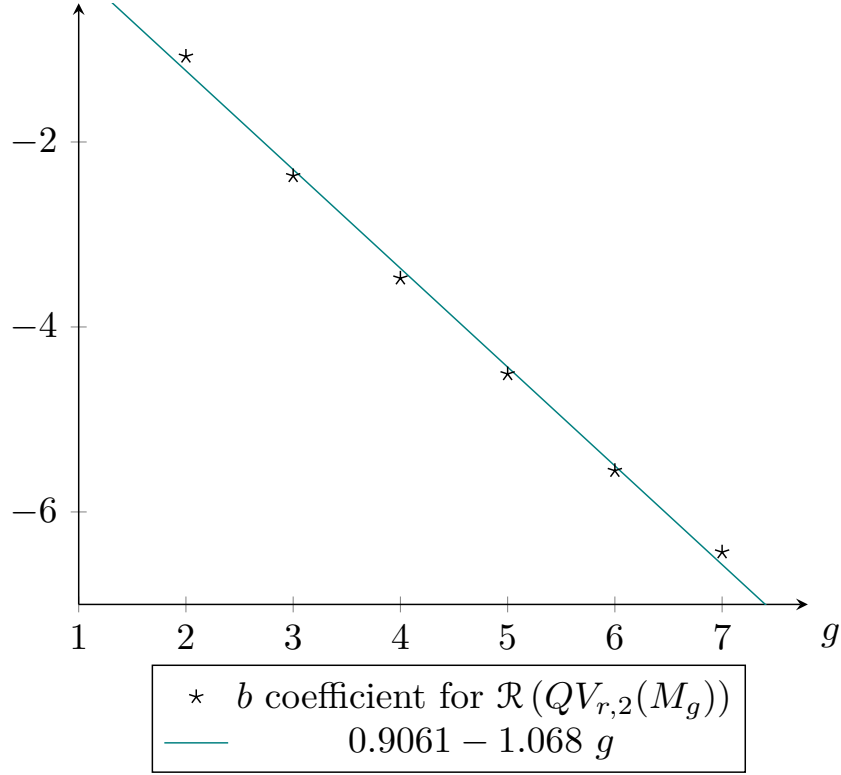

What precedes thus yields a satisfying numerical check of Conjecture \ref{conj:vol:b:aff} for the family of manifolds $M_g$.

More precisely, the interpolating affine function is computed to be 
$$0.9061 -1.068 g.$$
Of course, more data would give us an interpolating affine function closer to the expected one, but as the slope $-1.068$ is already quite close to $-1$, it is not unreasonable to look for a general behavior of $b$ in the form of 
$$ b(M) \overset{?}{=} \text{constant} - \frac{1}{2}\chi(\partial M),$$
since in the specific case of Frigerio's manifolds we have 
$$b(M_g) \underset{2 \leqslant g \leqslant 7}{\approx} \text{constant} - g = \text{constant} - \frac{1}{2} \chi(\partial M_g).$$

\section{Discussion and further directions}

\begin{itemize}
\item It would be interesting to understand the origin of the pattern breaks for $\mathcal{R}\left(QV_{r,2}(M_2)\right )$ and $\mathcal{R}\left(QV_{r,2}(M_3)\right )$. If, as we surmise, they come from numerical approximations by the machine for terms of different magnitudes, then this hypothesis could be tested by refining our code and examining the range of magnitudes of the terms in the sum when $r$ grows larger. The works of Maria-Rouillé \cite{MaRo} seem like a promising direction to follow.
\item Conjecture \ref{conj:vol:b:aff} appears to be satisfied for the manifolds $M_g$, but it would be interesting to test it for other families of manifolds with diverse boundary components. Furthermore, one could try to prove (or disprove!) rigorously that $b$ has an affine behavior of the form $- \frac{1}{2}\chi(\partial M)+constant$, via combinatorial arguments on the triangulations (how the numbers of vertices, edges and tetrahedra are related to $\chi(\partial M)$) and asymptotics in $r$ of the  terms associated to regular vertices, edges and tetrahedra in the definition of the Turaev--Viro invariants.
\item The extended volume conjectures as stated in Conjectures \ref{conj:vol:bc} and \ref{conj:vol:b:aff} already (or may possibly) admit variants for other quantum invariants. Can the methods used in the present paper be applied for these other invariants? The manifolds $M_g$ seem especially convenient to study for invariants defined on (ordered) triangulations.
\end{itemize}

\section*{Acknowledgements}
The first author was supported by the FNRS in his "Research Fellow" position at UCLouvain, under under Grant
no. 1B03320F. 
The second author would like to thank his two supervisors (Pedro Vaz and the first author) for their intellectual and emotional support throughout his Master's thesis that gave rise to the current paper.
Both authors thank Pedro Vaz for his continuous involvement
 in the project, and Renaud Detcherry and François Costantino for helpful discussions.

\end{document}